\documentclass{article}
\usepackage{amssymb}
\usepackage{amsfonts}
\usepackage{url}
\usepackage{amsmath}
\usepackage{graphicx}

\newtheorem{theorem}{Theorem}

\newtheorem{corollary}{Corollary}

\newtheorem{definition}{Definition}

\newtheorem{lemma}{Lemma}

\newtheorem{proposition}{Proposition}
\newtheorem{remark}{Remark}

\newenvironment{proof}[1][Proof]{\textbf{#1.} }{\ \rule{0.5em}{0.5em}}

\begin{document}

\begin{center}
{\LARGE Linearized Boltzmann Collision Operator: I. Polyatomic Molecules
Modeled by a }

{\LARGE Discrete Internal Energy Variable and Multicomponent Mixtures}%
\bigskip

{\large Niclas Bernhoff\smallskip }

Department of Mathematics, Karlstad University, 65188 Karlstad, Sweden

niclas.bernhoff@kau.se
\end{center}

\textbf{Abstract:}{\small \ The linearized collision operator of the
Boltzmann equation can in a natural way be written as a sum of a positive
multiplication operator, the collision frequency, and an integral operator.
Compactness of the integral operator for monatomic single species is a
classical result, while corresponding result for mixtures is more recently
obtained. In this work the compactness of the operator for polyatomic single
species, where the polyatomicity is modeled by a discrete internal energy
variable, is studied. With a probabilistic formulation of the collision
operator as a starting point, compactness is obtained by proving that the
integral operator is a sum of Hilbert-Schmidt integral operators and
approximately Hilbert-Schmidt integral operators, under some assumptions on
the collision kernel. Self-adjointness of the linearized collision operator
follows. Moreover, bounds on - including coercivity of - the collision
frequency are obtained for a hard sphere model. Then it follows that the
linearized collision operator is a Fredholm operator.}

{\small The results can be extended to mixtures. For brevity, only the case
of mixtures for monatomic species is accounted for.}

\section{Introduction\label{S1}}

The Boltzmann equation is a fundamental equation of kinetic theory of gases.
Considering deviations of an equilibrium, or Maxwellian, distribution, a
linearized collision operator is obtained. The linearized collision operator
can in a natural way be written as a sum of a positive multiplication
operator, the collision frequency, and an integral operator $-K$. Compact
properties of the integral operator $K$ (for angular cut-off kernels) are
extensively studied for monatomic single species, see e.g. \cite{Gr-63,
Dr-75, Cercignani-88, LS-10}. The integral operator can be written as the
sum of a Hilbert-Schmidt integral operator and an approximately
Hilbert-Schmidt integral operator (cf. Lemma $\ref{LGD}$ in Section $\ref%
{PT1}$) \cite{Glassey}, and so compactness of the integral operator $K$ can
be obtained. More recently, compactness results were also obtained for
monatomic multi-component mixtures \cite{BGPS-13}. In this work, we consider
polyatomic single species, where the polyatomicity is modeled by a discrete
internal energy variable \cite{ErnGio-94,GS-99}. We also consider the case
of multi-component mixtures. For brevity - and clearness - we restrict
ourselves to monatomic mixtures, even if compactness is already developed in
that case in \cite{BGPS-13}. Our approach is different - maybe, also more
direct and transparent - even if also similarities are obvious.\ The
approach was crucial for us to understand the polyatomic case. In addition
to trying to make the ideas and approach more clear, our intention by
addressing the two cases separately - but in a corresponding way -, is to
make the extension to multi-component mixtures of polyatomic species, at
least at an formal level, clear. The case of polyatomic single species,
where the polyatomicity is modeled by a continuous internal energy variable
is considered in \cite{Be-21b}.

Motivated by an approach by Kogan in \cite[Sect. 2.8]{Kogan} for the
monatomic single species case, a probabilistic formulation of the collision
operator is considered as the starting point. With this approach, it is
shown, based on modified arguments from the monatomic case, that the
integral operator $K$ can be written as a sum of Hilbert-Schmidt integral
operators and approximately Hilbert-Schmidt integral operators - and so
compactness of the integral operator $K$ follows -, also for the polyatomic
model. The operator $K$ is self-adjoint, as well as the collision frequency,
why the linearized collision operator as the sum of two self-adjoint
operators of which one is bounded, is also self-adjoint.

For hard sphere models, bounds on the collision frequency are obtained. Then
the collision frequency is coercive and becomes a Fredholm operator. The set
of Fredholm operators is closed under addition with compact operators, why
also the linearized collision operator becomes a Fredholm operator by the
compactness of the integral operator $K$.

The rest of the paper is organized as follows. In Section $\ref{S2}$, the
models considered are presented. The probabilistic formulation of the
collision operators considered and its relations to more classical
formulations \cite{ErnGio-94,GS-99} are accounted for in Section $\ref%
{S2.1.1}$ for polyatomic molecules and in Section $\ref{S2.2.1}$ for
mixtures.\ Some classical results for the collision operators in Sections $%
\ref{S2.1.2}$ and $\ref{S2.2.2}$, respectively, and the linearized collision
operators in Sections $\ref{S2.1.3}$ and $\ref{S2.2.3}$, respectively, are
reviewed. Section $\ref{S3}$ is devoted to the main results of this paper,
while the main proofs are addressed in Section $\ref{S4}$; proofs of
compactness of the integral operators $K$ are presented in Sections $\ref%
{PT1}$ and $\ref{PT3}$, respectively, while proof of the bounds on the
collision frequency appears in Sections $\ref{PT2}$ and $\ref{PT4}$,
respectively. Finally, the appendix concerns a new proof of a crucial - for
the compactness in the mixture case - lemma by \cite{BGPS-13}.

\section{Models\label{S2}}

This section concerns the models considered. Probabilistic formulations of
the collision operators are considered, whose relations to more classical
formulations are accounted for. Known properties of the models and
corresponding linearized collision operators are also reviewed.

\subsection{Polyatomic molecules modeled by a discrete internal energy
variable\label{S2.1}}

Consider a single species of polyatomic molecules with mass $m$, where the
polyatomicity is modeled by $r$ different internal energies $I_{1},...,I_{r}$%
. The internal energies $I_{i}$, $i\in \left\{ 1,...,r\right\} $, are
assumed to be nonnegative real numbers; $\left\{ I_{1},...,I_{r}\right\}
\subset $ $\mathbb{R}_{+}$. The distribution functions are of the form $%
f=\left( f_{1},...,f_{r}\right) $, where $f_{i}=f_{i}\left( t,\mathbf{x},%
\boldsymbol{\xi }\right) =f\left( t,\mathbf{x},\boldsymbol{\xi }%
,I_{i}\right) $, with $t\in \mathbb{R}_{+}$, $\mathbf{x}=\left( x,y,z\right)
\in \mathbb{R}^{3}$, and $\boldsymbol{\xi }=\left( \xi _{x},\xi _{y},\xi
_{z}\right) \in \mathbb{R}^{3}$.

Moreover, consider the real Hilbert space $\mathcal{\mathfrak{h}}%
^{(r)}:=\left( L^{2}\left( d\boldsymbol{\xi }\right) \right) ^{r}$, with
inner product%
\begin{equation*}
\left( f,g\right) =\sum_{i=1}^{r}\int_{\mathbb{R}^{3}}f_{i}g_{i}\,d%
\boldsymbol{\xi }\text{, }f,g\in \left( L^{2}\left( d\boldsymbol{\xi }%
\right) \right) ^{r}\text{.}
\end{equation*}

The evolution of the distribution functions is (in the absence of external
forces) described by the (vector) Boltzmann equation%
\begin{equation}
\frac{\partial f}{\partial t}+\left( \boldsymbol{\xi }\cdot \nabla _{\mathbf{%
x}}\right) f=Q\left( f,f\right) \text{,}  \label{BE1}
\end{equation}%
where the (vector) collision operator $Q=\left( Q_{1},...,Q_{r}\right) $ is
a quadratic bilinear operator that accounts for the change of velocities and
internal energies of particles due to binary collisions (assuming that the
gas is rarefied, such that other collisions are negligible).

\subsubsection{Collision operator\label{S2.1.1}}

The (vector) collision operator $Q=\left( Q_{1},...,Q_{r}\right) $ has
components that can be written in the following form 
\begin{equation*}
Q_{i}(f,f)=\sum\limits_{j,k,l=1}^{r}\int_{\left( \mathbb{R}^{3}\right)
^{3}}W(\boldsymbol{\xi },\boldsymbol{\xi }_{\ast },I_{i},I_{j}\left\vert 
\boldsymbol{\xi }^{\prime },\boldsymbol{\xi }_{\ast }^{\prime
},I_{k},I_{l}\right. )\left( \frac{f_{k}^{\prime }f_{l\ast }^{\prime }}{%
\varphi _{k}\varphi _{l}}-\frac{f_{i}f_{j\ast }}{\varphi _{i}\varphi _{j}}%
\right) d\boldsymbol{\xi }_{\ast }d\boldsymbol{\xi }^{\prime }d\boldsymbol{%
\xi }_{\ast }^{\prime }\!\text{,}
\end{equation*}%
for some constant $\varphi =\left( \varphi _{1},...,\varphi _{r}\right) \in 
\mathbb{R}^{r}$ . Here and below the abbreviations%
\begin{equation}
f_{i\ast }=f_{i}\left( t,\mathbf{x,}\boldsymbol{\xi }_{\ast }\right) \text{, 
}f_{i}^{\prime }=f_{i}\left( t,\mathbf{x,}\boldsymbol{\xi }^{\prime }\right) 
\text{, and }f_{i\ast }^{\prime }=f_{i}\left( t,\mathbf{x,}\boldsymbol{\xi }%
_{\ast }^{\prime }\right)  \label{a1}
\end{equation}%
are used. The transition probabilities are of the form, cf. \cite{HD-69}, as
well as \cite{Kogan, SNB-85, BPS-90} for the monatomic case, 
\begin{eqnarray}
&&W(\boldsymbol{\xi },\boldsymbol{\xi }_{\ast },I_{i},I_{j}\left\vert 
\boldsymbol{\xi }^{\prime },\boldsymbol{\xi }_{\ast }^{\prime
},I_{k},I_{l}\right. )  \notag \\
&=&4m\varphi _{k}\varphi _{l}\sigma _{kl}^{ij}\frac{\left\vert \mathbf{g}%
^{\prime }\right\vert }{\left\vert \mathbf{g}\right\vert }\delta _{3}\left( 
\boldsymbol{\xi }+\boldsymbol{\xi }_{\ast }-\boldsymbol{\xi }^{\prime }-%
\boldsymbol{\xi }_{\ast }^{\prime }\right)  \notag \\
&&\times \delta _{1}\left( \frac{m}{2}\left( \left\vert \boldsymbol{\xi }%
\right\vert ^{2}+\left\vert \boldsymbol{\xi }_{\ast }\right\vert
^{2}-\left\vert \boldsymbol{\xi }^{\prime }\right\vert ^{2}-\left\vert 
\boldsymbol{\xi }_{\ast }^{\prime }\right\vert ^{2}\right) -\Delta
I_{ij}^{kl}\right)  \notag \\
&=&4m\varphi _{i}\varphi _{j}\sigma _{ij}^{kl}\frac{\left\vert \mathbf{g}%
\right\vert }{\left\vert \mathbf{g}^{\prime }\right\vert }\delta _{3}\left( 
\boldsymbol{\xi }+\boldsymbol{\xi }_{\ast }-\boldsymbol{\xi }^{\prime }-%
\boldsymbol{\xi }_{\ast }^{\prime }\right)  \notag \\
&&\times \delta _{1}\left( \frac{m}{2}\left( \left\vert \boldsymbol{\xi }%
\right\vert ^{2}+\left\vert \boldsymbol{\xi }_{\ast }\right\vert
^{2}-\left\vert \boldsymbol{\xi }^{\prime }\right\vert ^{2}-\left\vert 
\boldsymbol{\xi }_{\ast }^{\prime }\right\vert ^{2}\right) -\Delta
I_{ij}^{kl}\right) \text{,}  \notag \\
\sigma _{ij}^{kl} &=&\sigma _{ij}^{kl}\left( \left\vert \mathbf{g}%
\right\vert ,\left\vert \cos \theta \right\vert \right) >0\text{ and }\sigma
_{kl}^{ij}=\sigma _{kl}^{ij}\left( \left\vert \mathbf{g}^{\prime
}\right\vert ,\left\vert \cos \theta \right\vert \right) >0\text{ a.e., with 
}  \notag \\
\cos \theta &=&\frac{\mathbf{g}\cdot \mathbf{g}^{\prime }}{\left\vert 
\mathbf{g}\right\vert \left\vert \mathbf{g}^{\prime }\right\vert }\text{, }%
\mathbf{g}=\boldsymbol{\xi }-\boldsymbol{\xi }_{\ast }\!\text{, }\mathbf{g}%
^{\prime }=\boldsymbol{\xi }^{\prime }-\boldsymbol{\xi }_{\ast }^{\prime }\!%
\text{, and }\Delta I_{ij}^{kl}=I_{k}+I_{l}-I_{i}-I_{j}\text{,}  \label{tp}
\end{eqnarray}%
where $\delta _{3}$ and $\delta _{1}$ denote the Dirac's delta function in $%
\mathbb{R}^{3}$ and $\mathbb{R}$, respectively; taking the conservation of
momentum and total energy into account. Here it is assumed that the
scattering cross sections $\sigma _{ij}^{kl}$, $\left\{ i,j,k,l\right\}
\subseteq \left\{ 1,...,r\right\} $, satisfy the microreversibility
conditions%
\begin{equation}
\varphi _{i}\varphi _{j}\left\vert \mathbf{g}\right\vert ^{2}\sigma
_{ij}^{kl}\left( \left\vert \mathbf{g}\right\vert ,\left\vert \cos \theta
\right\vert \right) =\varphi _{k}\varphi _{l}\left\vert \mathbf{g}^{\prime
}\right\vert ^{2}\sigma _{kl}^{ij}\left( \left\vert \mathbf{g}^{\prime
}\right\vert ,\left\vert \cos \theta \right\vert \right) \text{.}  \label{mr}
\end{equation}%
Furthermore, to obtain invariance of \ change of particles in a collision,
it is assumed that the scattering cross sections $\sigma _{ij}^{kl}$, $%
\left\{ i,j,k,l\right\} \subseteq \left\{ 1,...,r\right\} $, satisfy the
symmetry relations 
\begin{equation}
\sigma _{ij}^{kl}=\sigma _{ij}^{lk}=\sigma _{ji}^{lk}\text{.}  \label{sr}
\end{equation}%
The invariance under change of particles in a collision, which follows
directly by the definition of the transition probability $\left( \ref{tp}%
\right) $ and the symmetry relations $\left( \ref{sr}\right) $ for the
collision frequency, and the microreversibility of the collisions $\left( %
\ref{mr}\right) $, implies that the transition probabilities $\left( \ref{tp}%
\right) $ satisfy the relations

\begin{eqnarray}
W(\boldsymbol{\xi },\boldsymbol{\xi }_{\ast },I_{i},I_{j}\left\vert 
\boldsymbol{\xi }^{\prime },\boldsymbol{\xi }_{\ast }^{\prime
},I_{k},I_{l}\right. ) &=&W(\boldsymbol{\xi }_{\ast },\boldsymbol{\xi }%
,I_{j},I_{i}\left\vert \boldsymbol{\xi }_{\ast }^{\prime },\boldsymbol{\xi }%
^{\prime },I_{l},I_{k}\right. )  \notag \\
W(\boldsymbol{\xi },\boldsymbol{\xi }_{\ast },I_{i},I_{j}\left\vert 
\boldsymbol{\xi }^{\prime },\boldsymbol{\xi }_{\ast }^{\prime
},I_{k},I_{l}\right. ) &=&W(\boldsymbol{\xi }^{\prime },\boldsymbol{\xi }%
_{\ast }^{\prime },I_{k},I_{l}\left\vert \boldsymbol{\xi },\boldsymbol{\xi }%
_{\ast },I_{i},I_{j}\right. )  \notag \\
W(\boldsymbol{\xi },\boldsymbol{\xi }_{\ast },I_{i},I_{j}\left\vert 
\boldsymbol{\xi }^{\prime },\boldsymbol{\xi }_{\ast }^{\prime
},I_{k},I_{l}\right. ) &=&W(\boldsymbol{\xi },\boldsymbol{\xi }_{\ast
},I_{i},I_{j}\left\vert \boldsymbol{\xi }_{\ast }^{\prime },\boldsymbol{\xi }%
^{\prime },I_{l},I_{k}\right. )\text{.}  \label{rel1}
\end{eqnarray}

Applying known properties of Dirac's delta function, the transition
probabilities may be transformed to 
\begin{eqnarray*}
&&W(\boldsymbol{\xi },\boldsymbol{\xi }_{\ast },I_{i},I_{j}\left\vert 
\boldsymbol{\xi }^{\prime },\boldsymbol{\xi }_{\ast }^{\prime
},I_{k},I_{l}\right. ) \\
&=&4m\varphi _{k}\varphi _{l}\sigma _{kl}^{ij}\frac{\left\vert \mathbf{g}%
^{\prime }\right\vert }{\left\vert \mathbf{g}\right\vert }\delta _{3}\left(
2\left( \mathbf{G}-\mathbf{G}^{\prime }\right) \right) \delta _{1}\left( 
\frac{m}{4}\left( \left\vert \mathbf{g}\right\vert ^{2}-\left\vert \mathbf{g}%
^{\prime }\right\vert ^{2}\right) -\Delta I_{ij}^{kl}\right) \\
&=&2\varphi _{k}\varphi _{l}\sigma _{kl}^{ij}\frac{\left\vert \mathbf{g}%
^{\prime }\right\vert }{\left\vert \mathbf{g}\right\vert }\delta _{3}\left( 
\mathbf{G}-\mathbf{G}^{\prime }\right) \delta _{1}\left( \left( \left\vert 
\mathbf{g}\right\vert ^{2}-\left\vert \mathbf{g}^{\prime }\right\vert
^{2}\right) -\frac{4}{m}\Delta I_{ij}^{kl}\right) \\
&=&\varphi _{k}\varphi _{l}\sigma _{kl}^{ij}\frac{1}{\left\vert \mathbf{g}%
\right\vert }\mathbf{1}_{m\left\vert \mathbf{g}\right\vert ^{2}>4\Delta
I_{ij}^{kl}}\delta _{3}\left( \mathbf{G}-\mathbf{G}^{\prime }\right) \delta
_{1}\left( \sqrt{\left\vert \mathbf{g}\right\vert ^{2}-\frac{4}{m}\Delta
I_{ij}^{kl}}-\left\vert \mathbf{g}^{\prime }\right\vert \right) \\
&=&\varphi _{i}\varphi _{j}\sigma _{ij}^{kl}\frac{\left\vert \mathbf{g}%
\right\vert }{\left\vert \mathbf{g}^{\prime }\right\vert ^{2}}\mathbf{1}%
_{m\left\vert \mathbf{g}\right\vert ^{2}>4\Delta I_{ij}^{kl}}\delta
_{3}\left( \mathbf{G}-\mathbf{G}^{\prime }\right) \delta _{1}\left( \sqrt{%
\left\vert \mathbf{g}\right\vert ^{2}-\frac{4}{m}\Delta I_{ij}^{kl}}%
-\left\vert \mathbf{g}^{\prime }\right\vert \right) \text{,} \\
\text{ } &&\text{with }\mathbf{G}=\frac{\boldsymbol{\xi }+\boldsymbol{\xi }%
_{\ast }}{2}\text{ and }\mathbf{G}^{\prime }=\frac{\boldsymbol{\xi }^{\prime
}+\boldsymbol{\xi }_{\ast }^{\prime }}{2}\text{.}
\end{eqnarray*}

\begin{remark}
Note that, cf. \cite{HD-69},%
\begin{equation*}
\delta _{1}\left( \frac{m}{4}\left( \left\vert \mathbf{g}\right\vert
^{2}-\left\vert \mathbf{g}^{\prime }\right\vert ^{2}\right) -\Delta
I_{ij}^{kl}\right) =\delta _{1}\left( E_{ij}-E_{kl}\right) ,
\end{equation*}%
for $E_{ij}=\dfrac{m}{4}\left\vert \mathbf{g}\right\vert ^{2}+I_{i}+I_{j}\ $%
\ and $E_{kl}=\dfrac{m}{4}\left\vert \mathbf{g}^{\prime }\right\vert
^{2}+I_{k}+I_{l}.$
\end{remark}

By a change of variables $\left\{ \boldsymbol{\xi }^{\prime },\boldsymbol{%
\xi }_{\ast }^{\prime }\right\} \rightarrow \left\{ \mathbf{g}^{\prime }=%
\boldsymbol{\xi }^{\prime }-\boldsymbol{\xi }_{\ast }^{\prime },\mathbf{G}%
^{\prime }=\dfrac{\boldsymbol{\xi }^{\prime }+\boldsymbol{\xi }_{\ast
}^{\prime }}{2}\right\} $, noting that%
\begin{equation}
d\boldsymbol{\xi }^{\prime }d\boldsymbol{\xi }_{\ast }^{\prime }=d\mathbf{G}%
^{\prime }d\mathbf{g}^{\prime }=\left\vert \mathbf{g}^{\prime }\right\vert
^{2}d\mathbf{G}^{\prime }d\left\vert \mathbf{g}^{\prime }\right\vert d%
\boldsymbol{\omega }\text{, }\boldsymbol{\omega }=\frac{\mathbf{g}^{\prime }%
}{\left\vert \mathbf{g}^{\prime }\right\vert }\text{,}  \label{df1}
\end{equation}%
the observation that%
\begin{eqnarray*}
&&Q_{i}(f,f) \\
&=&\sum\limits_{j,k,l=1}^{r}\int_{\left( \mathbb{R}^{3}\right) ^{2}\times 
\mathbb{R}_{+}\mathbb{\times S}^{2}}W(\boldsymbol{\xi },\boldsymbol{\xi }%
_{\ast },I_{i},I_{j}\left\vert \boldsymbol{\xi }^{\prime },\boldsymbol{\xi }%
_{\ast }^{\prime },I_{k},I_{l}\right. ) \\
&&\times \left( \frac{f_{k}^{\prime }f_{l\ast }^{\prime }}{\varphi
_{k}\varphi _{l}}-\frac{f_{i}f_{j\ast }}{\varphi _{i}\varphi _{j}}\right)
\,\left\vert \mathbf{g}^{\prime }\right\vert ^{2}\,d\boldsymbol{\xi }_{\ast
}d\mathbf{G}^{\prime }d\left\vert \mathbf{g}^{\prime }\right\vert d%
\boldsymbol{\omega } \\
&=&\sum\limits_{j,k,l=1}^{r}\int_{\mathbb{R}^{3}\mathbb{\times S}%
^{2}}\left\vert \mathbf{g}\right\vert \sigma _{ij}^{kl}\left( \left\vert 
\mathbf{g}\right\vert ,\cos \theta \right) \left( f_{k}^{\prime }f_{l\ast
}^{\prime }\frac{\varphi _{i}\varphi _{j}}{\varphi _{k}\varphi _{l}}%
-f_{i}f_{j\ast }\right) \mathbf{1}_{m\left\vert \mathbf{g}\right\vert
^{2}>4\Delta I_{ij}^{kl}}\,d\boldsymbol{\xi }_{\ast }d\boldsymbol{\omega ,}
\end{eqnarray*}%
where%
\begin{equation*}
\left\{ 
\begin{array}{l}
\boldsymbol{\xi }^{\prime }=\dfrac{\boldsymbol{\xi }+\boldsymbol{\xi }_{\ast
}}{2}+\dfrac{\sqrt{\left\vert \boldsymbol{\xi }-\boldsymbol{\xi }_{\ast
}\right\vert ^{2}-\dfrac{4}{m}\Delta I_{ij}^{kl}}}{2}\omega =\mathbf{G}+%
\dfrac{\sqrt{\left\vert \mathbf{g}\right\vert ^{2}-\dfrac{4}{m}\Delta
I_{ij}^{kl}}}{2}\omega \medskip \\ 
\boldsymbol{\xi }_{\ast }^{\prime }=\dfrac{\boldsymbol{\xi }+\boldsymbol{\xi 
}_{\ast }}{2}-\dfrac{\sqrt{\left\vert \boldsymbol{\xi }-\boldsymbol{\xi }%
_{\ast }\right\vert ^{2}-\dfrac{4}{m}\Delta I_{ij}^{kl}}}{2}\omega =\mathbf{G%
}-\dfrac{\sqrt{\left\vert \mathbf{g}\right\vert ^{2}-\dfrac{4}{m}\Delta
I_{ij}^{kl}}}{2}\omega%
\end{array}%
\right. \text{, }\omega \in S^{2}\text{,}
\end{equation*}%
can be made, resulting in a more familiar form of the Boltzmann collision
operator for polyatomic molecules modeled with a discrete energy variable,
cf. e.g. \cite{ErnGio-94,GS-99}.

\subsubsection{Collision invariants and Maxwellian distributions\label%
{S2.1.2}}

The following lemma follows directly by the relations $\left( \ref{rel1}%
\right) $.

\begin{lemma}
\label{L0}For any $\left\{ i,j,k,l\right\} \subseteq \left\{ 1,\ldots
,r\right\} $ the measure 
\begin{equation*}
dA_{ij}^{kl}=W(\boldsymbol{\xi },\boldsymbol{\xi }_{\ast
},I_{i},I_{j}\left\vert \boldsymbol{\xi }^{\prime },\boldsymbol{\xi }_{\ast
}^{\prime },I_{k},I_{l}\right. )\,d\boldsymbol{\xi }\,d\boldsymbol{\xi }%
_{\ast }d\boldsymbol{\xi }^{\prime }d\boldsymbol{\xi }_{\ast }^{\prime }
\end{equation*}%
is invariant under the interchanges of variables%
\begin{eqnarray}
\left( \boldsymbol{\xi },\boldsymbol{\xi }_{\ast },I_{i},I_{j}\right)
&\leftrightarrow &\left( \boldsymbol{\xi }^{\prime },\boldsymbol{\xi }_{\ast
}^{\prime },I_{k},I_{l}\right) \text{,}  \notag \\
\left( \boldsymbol{\xi },I_{i}\right) &\leftrightarrow &\left( \boldsymbol{%
\xi }_{\ast },I_{j}\right) \text{, and}  \notag \\
\left( \boldsymbol{\xi }^{\prime },I_{k}\right) &\leftrightarrow &\left( 
\boldsymbol{\xi }_{\ast }^{\prime },I_{l}\right) \text{,}  \label{tr}
\end{eqnarray}%
respectively.
\end{lemma}

The weak form of the collision operator $Q(f,f)$ reads%
\begin{eqnarray*}
\left( Q(f,f),g\right) &=&\sum\limits_{i,j,k,l=1}^{r}\int_{\left( \mathbb{R}%
^{3}\right) ^{4}}\left( \frac{f_{k}^{\prime }f_{l\ast }^{\prime }}{\varphi
_{k}\varphi _{l}}-\frac{f_{i}f_{j\ast }}{\varphi _{i}\varphi _{j}}\right)
g_{i}\,dA_{ij}^{kl} \\
&=&\sum\limits_{i,j,k,l=1}^{r}\int_{\left( \mathbb{R}^{3}\right) ^{4}}\left( 
\frac{f_{k}^{\prime }f_{l\ast }^{\prime }}{\varphi _{k}\varphi _{l}}-\frac{%
f_{i}f_{j\ast }}{\varphi _{i}\varphi _{j}}\right) g_{j\ast }\,dA_{ij}^{kl} \\
&=&-\sum\limits_{i,j,k,l=1}^{r}\int_{\left( \mathbb{R}^{3}\right)
^{4}}\left( \frac{f_{k}^{\prime }f_{l\ast }^{\prime }}{\varphi _{k}\varphi
_{l}}-\frac{f_{i}f_{j\ast }}{\varphi _{i}\varphi _{j}}\right) g_{k}^{\prime
}\,dA_{ij}^{kl} \\
&=&-\sum\limits_{i,j,k,l=1}^{r}\int_{\left( \mathbb{R}^{3}\right)
^{4}}\left( \frac{f_{k}^{\prime }f_{l\ast }^{\prime }}{\varphi _{k}\varphi
_{l}}-\frac{f_{i}f_{j\ast }}{\varphi _{i}\varphi _{j}}\right) g_{l\ast
}^{\prime }\,dA_{ij}^{kl},
\end{eqnarray*}%
for any function $g=(g_{1},...,g_{r})$, such that the first integrals are
defined for all $\left\{ i,j,k,l\right\} \subseteq \left\{ 1,\ldots
,r\right\} $, while the following equalities are obtained by applying Lemma $%
\ref{L0}$.

We have the following proposition.

\begin{proposition}
\label{P1}Let $g=(g_{1},...,g_{r})$ be such that 
\begin{equation*}
\int_{\left( \mathbb{R}^{3}\right) ^{4}}\left( \frac{f_{k}^{\prime }f_{l\ast
}^{\prime }}{\varphi _{k}\varphi _{l}}-\frac{f_{i}f_{j\ast }}{\varphi
_{i}\varphi _{j}}\right) g_{i}\,dA_{ij}^{kl}
\end{equation*}%
is defined for all $\left\{ i,j,k,l\right\} \subseteq \left\{ 1,\ldots
,r\right\} $. Then%
\begin{equation*}
\left( Q(f,f),g\right) =\frac{1}{4}\sum\limits_{i,j,k,l=1}^{r}\int_{\left( 
\mathbb{R}^{3}\right) ^{4}}\left( \frac{f_{k}^{\prime }f_{l\ast }^{\prime }}{%
\varphi _{k}\varphi _{l}}-\frac{f_{i}f_{j\ast }}{\varphi _{i}\varphi _{j}}%
\right) \left( g_{i}+g_{j\ast }-g_{k}^{\prime }-g_{l\ast }^{\prime }\right)
\,dA_{ij}^{kl}.
\end{equation*}
\end{proposition}

\begin{definition}
A function $g=(g_{1},...,g_{r})\ $is a collision invariant if 
\begin{equation*}
\left( g_{i}+g_{j\ast }-g_{k}^{\prime }-g_{l\ast }^{\prime }\right) W(%
\boldsymbol{\xi },\boldsymbol{\xi }_{\ast },I_{i},I_{j}\left\vert 
\boldsymbol{\xi }^{\prime },\boldsymbol{\xi }_{\ast }^{\prime
},I_{k},I_{l}\right. )=0\text{ a.e.}
\end{equation*}%
for all $\left\{ i,j,k,l\right\} \subseteq \left\{ 1,\ldots ,r\right\} $.
\end{definition}

It is clear that $\mathbf{1},$ $\xi _{x}\mathbf{1},$ $\xi _{y}\mathbf{1},$ $%
\xi _{z}\mathbf{1}\ $and $m\left\vert \boldsymbol{\xi }\right\vert ^{2}%
\mathbf{1}+2\mathbb{I}$, where $\mathbf{1}=(1,...,1)\in \mathbb{R}^{r}$ and $%
\mathbb{I}=(I_{1},...,I_{r})$, are collision invariants - corresponding to
conservation of mass, momentum, and total energy.

In fact, we have the following proposition, cf. \cite{GS-99, Cercignani-88}.

\begin{proposition}
\label{P2}The vector space of collision invariants is generated by 
\begin{equation*}
\left\{ \mathbf{1},\xi _{x}\mathbf{1},\xi _{y}\mathbf{1},\xi _{z}\mathbf{1}%
,m\left\vert \boldsymbol{\xi }\right\vert ^{2}\mathbf{1}+2\mathbb{I}\right\}
,
\end{equation*}%
where $\mathbf{1}=(1,...,1)\in \mathbb{R}^{r}$ and $\mathbb{I}%
=(I_{1},...,I_{r})$.
\end{proposition}

Define%
\begin{equation*}
\mathcal{W}\left[ f\right] :=\left( Q(f,f),\log \left( \varphi ^{-1}f\right)
\right) ,
\end{equation*}%
where $\varphi =\mathrm{diag}\left( \varphi _{1},...,\varphi _{r}\right) $.
It follows by Proposition $\ref{P1}$ that%
\begin{equation*}
\mathcal{W}\left[ f\right] =-\frac{1}{4}\sum\limits_{i,j,k,l=1}^{r}\int%
\limits_{\left( \mathbb{R}^{3}\right) ^{4}}\frac{f_{i}f_{j\ast }}{\varphi
_{i}\varphi _{j}}\left( \frac{\varphi _{i}\varphi _{j}f_{k}^{\prime
}f_{l\ast }^{\prime }}{f_{i}f_{j\ast }\varphi _{k}\varphi _{l}}-1\right)
\log \left( \frac{\varphi _{i}\varphi _{j}f_{k}^{\prime }f_{l\ast }^{\prime }%
}{f_{i}f_{j\ast }\varphi _{k}\varphi _{l}}\right) \,dA_{ij}^{kl}\text{.}
\end{equation*}%
Since $\left( x-1\right) \mathrm{log}\left( x\right) \geq 0$ for $x>0$, with
equality if and only if $x=1$,%
\begin{equation*}
\mathcal{W}\left[ f\right] \leq 0\text{,}
\end{equation*}%
with equality if and only if for all $\left\{ i,j,k,l\right\} \subseteq
\left\{ 1,\ldots ,r\right\} $ 
\begin{equation}
\left( \frac{f_{i}f_{j\ast }}{\varphi _{i}\varphi _{j}}-\frac{f_{k}^{\prime
}f_{l\ast }^{\prime }}{\varphi _{k}\varphi _{l}}\right) W(\boldsymbol{\xi },%
\boldsymbol{\xi }_{\ast },I_{i},I_{j}\left\vert \boldsymbol{\xi }^{\prime },%
\boldsymbol{\xi }_{\ast }^{\prime },I_{k},I_{l}\right. )=0\text{ a.e.,}
\label{m1}
\end{equation}%
or, equivalently, if and only if%
\begin{equation*}
Q(f,f)\equiv 0\text{.}
\end{equation*}

For any equilibrium, or Maxwellian, distribution $M=(M_{1},...,M_{r})$, it
follows by equation $\left( \ref{m1}\right) $, since $Q(M,M)\equiv 0$, that%
\begin{eqnarray*}
&&\left( \log \frac{M_{i}}{\varphi _{i}}+\log \frac{M_{j\ast }}{\varphi _{j}}%
-\log \frac{M_{k}^{\prime }}{\varphi _{k}}-\log \frac{M_{l\ast }^{\prime }}{%
\varphi _{l}}\right) \\
\times &&W(\boldsymbol{\xi },\boldsymbol{\xi }_{\ast },I_{i},I_{j}\left\vert 
\boldsymbol{\xi }^{\prime },\boldsymbol{\xi }_{\ast }^{\prime
},I_{k},I_{l}\right. )=0\text{ a.e. .}
\end{eqnarray*}%
Hence, $\log \left( \varphi ^{-1}M\right) =\left( \log \dfrac{M_{1}}{\varphi
_{1}},...,\log \dfrac{M_{r}}{\varphi _{r}}\right) $ is a collision
invariant, and the components of the Maxwellian distributions $%
M=(M_{1},...,M_{r})$ are of the form 
\begin{equation*}
M_{i}=\frac{n\varphi _{i}m^{3/2}}{\left( 2\pi T\right) ^{3/2}q}%
e^{-m\left\vert \boldsymbol{\xi }-\mathbf{u}\right\vert ^{2}/2T}e^{-I_{i}/T}%
\text{,}
\end{equation*}%
where $n=\left( M,\mathbf{1}\right) $, $\mathbf{u}=\dfrac{1}{n}\left( M,%
\boldsymbol{\xi }\mathbf{1}\right) $, and $T=\dfrac{m}{3n}\left(
M,\left\vert \boldsymbol{\xi }-\mathbf{u}\right\vert ^{2}\mathbf{1}\right) $%
, denoting $\mathbf{1}=(1,...,1)\in \mathbb{R}^{r}$, while $%
q=\sum\limits_{i=1}^{r}\varphi _{i}e^{-I_{i}/T}$.

Note that by equation $\left( \ref{m1}\right) $ any Maxwellian distribution $%
M=(M_{1},...,M_{r})$ satisfies the relations 
\begin{equation}
\left( \frac{M_{k}^{\prime }M_{l\ast }^{\prime }}{\varphi _{k}\varphi _{l}}-%
\frac{M_{i}M_{j\ast }}{\varphi _{i}\varphi _{j}}\right) W(\boldsymbol{\xi },%
\boldsymbol{\xi }_{\ast },I_{i},I_{j}\left\vert \boldsymbol{\xi }^{\prime },%
\boldsymbol{\xi }_{\ast }^{\prime },I_{k},I_{l}\right. )=0\text{ a.e. }.
\label{M1}
\end{equation}

\begin{remark}
Introducing the $\mathcal{H}$-functional%
\begin{equation*}
\mathcal{H}\left[ f\right] =\left( f,\log f\right) \text{,}
\end{equation*}%
an $\mathcal{H}$-theorem can be obtained.
\end{remark}

\subsubsection{Linearized collision operator\label{S2.1.3}}

Considering a deviation of a Maxwellian distribution $M=(M_{1},...,M_{r})$,
where $M_{i}=\dfrac{\varphi _{i}m^{3/2}}{\left( 2\pi \right) ^{3/2}}%
e^{-m\left\vert \boldsymbol{\xi }\right\vert ^{2}/2}e^{-I_{i}}$, of the form%
\begin{equation}
f=M+M^{1/2}h  \label{s1}
\end{equation}%
results, by insertion in the Boltzmann equation $\left( \ref{BE1}\right) $,
in the system%
\begin{equation}
\frac{\partial h}{\partial t}+\left( \boldsymbol{\xi }\cdot \nabla _{\mathbf{%
x}}\right) h+\mathcal{L}h=\Gamma \left( h,h\right) \text{,}  \label{LBE}
\end{equation}%
where the components of the linearized collision operator $\mathcal{L}%
=\left( \mathcal{L}_{1},...,\mathcal{L}_{r}\right) $ are given by \ 
\begin{eqnarray}
\mathcal{L}_{i}h &=&-M_{i}^{-1/2}\left(
Q_{i}(M,M^{1/2}h)+Q_{i}(M^{1/2}h,M)\right)  \notag \\
&=&M_{i}^{-1/2}\sum\limits_{j,k,l=1}^{r}\int_{\left( \mathbb{R}^{3}\right)
^{3}}\left( \frac{M_{i}M_{j\ast }M_{k}^{\prime }M_{l\ast }^{\prime }}{%
\varphi _{i}\varphi _{j}\varphi _{k}\varphi _{l}}\right) ^{1/2}W(\boldsymbol{%
\xi },\boldsymbol{\xi }_{\ast },I_{i},I_{j}\left\vert \boldsymbol{\xi }%
^{\prime },\boldsymbol{\xi }_{\ast }^{\prime },I_{k},I_{l}\right. )  \notag
\\
&&\times \left( \frac{h_{i}}{M_{i}^{1/2}}+\frac{h_{j\ast }}{M_{j\ast }^{1/2}}%
-\frac{h_{k}^{\prime }}{\left( M_{k}^{\prime }\right) ^{1/2}}-\frac{h_{l\ast
}^{\prime }}{\left( M_{l\ast }^{\prime }\right) ^{1/2}}\right) \,d%
\boldsymbol{\xi }_{\ast }d\boldsymbol{\xi }^{\prime }d\boldsymbol{\xi }%
_{\ast }^{\prime }  \notag \\
&=&\upsilon _{i}h_{i}-K_{i}\left( h\right) \text{,}  \label{dec2}
\end{eqnarray}%
with%
\begin{eqnarray}
\upsilon _{i} &=&\sum\limits_{j,k,l=1}^{r}\int_{\left( \mathbb{R}^{3}\right)
^{3}}\frac{M_{j\ast }}{\varphi _{i}\varphi _{j}}W(\boldsymbol{\xi },%
\boldsymbol{\xi }_{\ast },I_{i},I_{j}\left\vert \boldsymbol{\xi }^{\prime },%
\boldsymbol{\xi }_{\ast }^{\prime },I_{k},I_{l}\right. )\,d\boldsymbol{\xi }%
_{\ast }d\boldsymbol{\xi }^{\prime }d\boldsymbol{\xi }_{\ast }^{\prime }%
\text{, and}  \notag \\
K_{i}\left( h\right)
&=&M_{i}^{-1/2}\!\!\sum\limits_{j,k,l=1}^{r}\int_{\left( \mathbb{R}%
^{3}\right) ^{3}}\!\left( \frac{M_{i}M_{j\ast }M_{k}^{\prime }M_{l\ast
}^{\prime }}{\varphi _{i}\varphi _{j}\varphi _{k}\varphi _{l}}\right)
^{1/2}\!\!W(\boldsymbol{\xi },\boldsymbol{\xi }_{\ast
},I_{i},I_{j}\left\vert \boldsymbol{\xi }^{\prime },\boldsymbol{\xi }_{\ast
}^{\prime },I_{k},I_{l}\right. )  \notag \\
&&\times \left( \frac{h_{k}^{\prime }}{\left( M_{k}^{\prime }\right) ^{1/2}}+%
\frac{h_{l\ast }^{\prime }}{\left( M_{l\ast }^{\prime }\right) ^{1/2}}-\frac{%
h_{j\ast }}{M_{j\ast }^{1/2}}\right) \,d\boldsymbol{\xi }_{\ast }d%
\boldsymbol{\xi }^{\prime }d\boldsymbol{\xi }_{\ast }^{\prime }\text{,}
\label{dec1}
\end{eqnarray}%
while the components of the quadratic term $\Gamma =\left( \Gamma
_{1},...,\Gamma _{r}\right) $ are given by%
\begin{equation}
\Gamma _{i}\left( h,h\right) =M_{i}^{-1/2}Q_{i}(M^{1/2}h,M^{1/2}h)\text{.}
\label{nl1}
\end{equation}%
The multiplication operator $\Lambda $ defined by 
\begin{equation*}
\Lambda (f)=\nu f\text{, where }\nu =\mathrm{diag}\left( \nu _{1},...,\nu
_{r}\right) \text{,}
\end{equation*}%
is a closed, densely defined, self-adjoint operator on $\left( L^{2}\left( d%
\boldsymbol{\xi }\right) \right) ^{r}$. It is Fredholm, as well, if and only
if $\Lambda $ is coercive.

The following lemma follows immediately by Lemma $\ref{L0}$.

\begin{lemma}
\label{L1}For any $\left\{ i,j,k,l\right\} \subseteq \left\{ 1,\ldots
,r\right\} $ the measure 
\begin{equation*}
d\widetilde{A}_{ij}^{kl}=\left( \frac{M_{i}M_{j\ast }M_{k}^{\prime }M_{l\ast
}^{\prime }}{\varphi _{i}\varphi _{j}\varphi _{k}\varphi _{l}}\right)
^{1/2}dA_{ij}^{kl}
\end{equation*}%
is invariant under the interchanges $\left( \ref{tr}\right) $ of variables
respectively.
\end{lemma}

The weak form of the linearized collision operator $\mathcal{L}$ reads%
\begin{eqnarray*}
&&\left( \mathcal{L}h,g\right) \\
&=&\sum\limits_{i,j,k,l=1}^{r}\int_{\left( \mathbb{R}^{3}\right) ^{4}}\left( 
\frac{h_{i}}{M_{i}^{1/2}}+\frac{h_{j\ast }}{M_{j\ast }^{1/2}}-\frac{%
h_{k}^{\prime }}{\left( M_{k}^{\prime }\right) ^{1/2}}-\frac{h_{l\ast
}^{\prime }}{\left( M_{l\ast }^{\prime }\right) ^{1/2}}\right) \frac{g_{i}}{%
M_{i}^{1/2}}\,d\widetilde{A}_{ij}^{kl} \\
&=&\sum\limits_{i,j,k,l=1}^{r}\int_{\left( \mathbb{R}^{3}\right) ^{4}}\left( 
\frac{h_{i}}{M_{i}^{1/2}}+\frac{h_{j\ast }}{M_{j\ast }^{1/2}}-\frac{%
h_{k}^{\prime }}{\left( M_{k}^{\prime }\right) ^{1/2}}-\frac{h_{l\ast
}^{\prime }}{\left( M_{l\ast }^{\prime }\right) ^{1/2}}\right) \frac{%
g_{j\ast }}{M_{j\ast }^{1/2}}\,d\widetilde{A}_{ij}^{kl} \\
&=&-\sum\limits_{i,j,k,l=1}^{r}\int_{\left( \mathbb{R}^{3}\right)
^{4}}\left( \frac{h_{i}}{M_{i}^{1/2}}+\frac{h_{j\ast }}{M_{j\ast }^{1/2}}-%
\frac{h_{k}^{\prime }}{\left( M_{k}^{\prime }\right) ^{1/2}}-\frac{h_{l\ast
}^{\prime }}{\left( M_{l\ast }^{\prime }\right) ^{1/2}}\right) \frac{%
g_{k}^{\prime }}{\left( M_{k}^{\prime }\right) ^{1/2}}\,d\widetilde{A}%
_{ij}^{kl} \\
&=&-\sum\limits_{i,j,k,l=1}^{r}\int_{\left( \mathbb{R}^{3}\right)
^{4}}\left( \frac{h_{i}}{M_{i}^{1/2}}+\frac{h_{j\ast }}{M_{j\ast }^{1/2}}-%
\frac{h_{k}^{\prime }}{\left( M_{k}^{\prime }\right) ^{1/2}}-\frac{h_{l\ast
}^{\prime }}{\left( M_{l\ast }^{\prime }\right) ^{1/2}}\right) \frac{%
g_{l\ast }^{\prime }}{\left( M_{l\ast }^{\prime }\right) ^{1/2}}\,d%
\widetilde{A}_{ij}^{kl},
\end{eqnarray*}%
for any function $g=(g_{1},...,g_{r})$, such that the first integrals are
defined for all $\left\{ i,j,k,l\right\} \subseteq \left\{ 1,\ldots
,r\right\} $, while the following equalities are obtained by applying Lemma $%
\ref{L1}$.

We have the following lemma.

\begin{lemma}
\label{L2}Let $g=(g_{1},...,g_{r})$ be such that%
\begin{equation*}
\int_{\left( \mathbb{R}^{3}\right) ^{4}}\left( \frac{h_{i}}{M_{i}^{1/2}}+%
\frac{h_{j\ast }}{M_{j\ast }^{1/2}}-\frac{h_{k}^{\prime }}{\left(
M_{k}^{\prime }\right) ^{1/2}}-\frac{h_{l\ast }^{\prime }}{\left( M_{l\ast
}^{\prime }\right) ^{1/2}}\right) \frac{g_{i}}{M_{i}^{1/2}}\,d\widetilde{A}%
_{ij}^{kl}
\end{equation*}%
is defined for all $\left\{ i,j,k,l\right\} \subseteq \left\{ 1,\ldots
,r\right\} $. Then%
\begin{eqnarray*}
\left( \mathcal{L}h,g\right) &=&\frac{1}{4}\sum\limits_{i,j,k,l=1}^{r}\int_{%
\left( \mathbb{R}^{3}\right) ^{4}}\left( \frac{h_{i}}{M_{i}^{1/2}}+\frac{%
h_{j\ast }}{M_{j\ast }^{1/2}}-\frac{h_{k}^{\prime }}{\left( M_{k}^{\prime
}\right) ^{1/2}}-\frac{h_{l\ast }^{\prime }}{\left( M_{l\ast }^{\prime
}\right) ^{1/2}}\right) \\
&&\times \left( \frac{g_{i}}{M_{i}^{1/2}}+\frac{g_{j\ast }}{M_{j\ast }^{1/2}}%
-\frac{g_{k}^{\prime }}{\left( M_{k}^{\prime }\right) ^{1/2}}-\frac{g_{l\ast
}^{\prime }}{\left( M_{l\ast }^{\prime }\right) ^{1/2}}\right) d\widetilde{A}%
_{ij}^{kl}.
\end{eqnarray*}
\end{lemma}

\begin{proposition}
The linearized collision operator is symmetric and nonnegative,%
\begin{equation*}
\left( \mathcal{L}h,g\right) =\left( h,\mathcal{L}g\right) \text{ and }%
\left( \mathcal{L}h,h\right) \geq 0\text{,}
\end{equation*}%
and\ the kernel of $\mathcal{L}$, $\ker \mathcal{L}$, is generated by%
\begin{equation*}
\left\{ M^{1/2},\xi _{x}M^{1/2},\xi _{y}M^{1/2},\xi _{z}M^{1/2},m\left\vert 
\boldsymbol{\xi }\right\vert ^{2}M^{1/2}+2\mathcal{M}^{1/2}I\right\} ,
\end{equation*}%
where $I=\left( I_{1},...,I_{r}\right) $ and $\mathcal{M}=\mathrm{diag}%
\left( M_{1},...,M_{r}\right) $.
\end{proposition}

\begin{proof}
By Lemma $\ref{L2}$, it is immediate that $\left( \mathcal{L}h,g\right)
=\left( h,\mathcal{L}g\right) $, and%
\begin{equation*}
\left( \mathcal{L}h,h\right) =\frac{1}{4}\sum\limits_{i,j,k,l=1}^{r}\int_{%
\left( \mathbb{R}^{3}\right) ^{4}}\left( \frac{h_{i}}{M_{i}^{1/2}}+\frac{%
h_{j\ast }}{M_{j\ast }^{1/2}}-\frac{h_{k}^{\prime }}{\left( M_{k}^{\prime
}\right) ^{1/2}}-\frac{h_{l\ast }^{\prime }}{\left( M_{l\ast }^{\prime
}\right) ^{1/2}}\right) ^{2}d\widetilde{A}_{ij}^{kl}\geq 0.
\end{equation*}%
Furthermore, $h\in \ker \mathcal{L}$ if and only if $\left( \mathcal{L}%
h,h\right) =0$, which will be fulfilled\ if and only if for all $\left\{
i,j,k,l\right\} \subseteq \left\{ 1,\ldots ,r\right\} $%
\begin{equation*}
\left( \frac{h_{i}}{M_{i}^{1/2}}+\frac{h_{j\ast }}{M_{j\ast }^{1/2}}-\frac{%
h_{k}^{\prime }}{\left( M_{k}^{\prime }\right) ^{1/2}}-\frac{h_{l\ast
}^{\prime }}{\left( M_{l\ast }^{\prime }\right) ^{1/2}}\right) W(\boldsymbol{%
\xi },\boldsymbol{\xi }_{\ast },I_{i},I_{j}\left\vert \boldsymbol{\xi }%
^{\prime },\boldsymbol{\xi }_{\ast }^{\prime },I_{k},I_{l}\right. )\,=0\text{
a.e.,}
\end{equation*}%
i.e. if and only if $\mathcal{M}^{-1/2}h$ is a collision invariant. The last
part of the lemma now follows by Proposition $\ref{P2}.$
\end{proof}

\begin{remark}
Note also that the quadratic term is orthogonal to the kernel of $\mathcal{L}
$, i.e. $\Gamma \left( h,h\right) \in \left( \ker \mathcal{L}\right) ^{\perp
_{\mathcal{\mathfrak{h}}^{(r)}}}$.
\end{remark}

\subsection{Multicomponent mixtures\label{S2.2}}

Consider a mixture of $s$ monatomic species $a_{1},...,a_{s}$, with masses $%
m_{\alpha _{1}},...,m_{\alpha _{s}}$, respectively ($s=1$ corresponds to the
case of a single species). The distribution functions are of the form $%
f=\left( f_{1},...,f_{s}\right) $, where $f_{\alpha }=f_{\alpha }\left( t,%
\mathbf{x},\boldsymbol{\xi }\right) $, with $t\in \mathbb{R}_{+}$, $\mathbf{x%
}=\left( x,y,z\right) \in \mathbb{R}^{3}$, and $\boldsymbol{\xi }=\left( \xi
_{x},\xi _{y},\xi _{z}\right) \in \mathbb{R}^{3}$.

Moreover, consider the real Hilbert space $\mathcal{\mathfrak{h}}%
^{(s)}:=\left( L^{2}\left( d\boldsymbol{\xi }\right) \right) ^{s}$, with
inner product%
\begin{equation*}
\left( \left. f\right\vert g\right) =\sum_{\alpha =1}^{s}\int_{\mathbb{R}%
^{3}}f_{\alpha }g_{\alpha }\,d\boldsymbol{\xi }\text{, }f,g\in \left(
L^{2}\left( d\mathbf{v}\right) \right) ^{s}\text{.}
\end{equation*}

The evolution of the distribution functions is (in the absence of external
forces) described by the (vector) Boltzmann equation $\left( \ref{BE1}%
\right) $,\ where the (vector) collision operator $Q=\left(
Q_{1},...,Q_{s}\right) $ is a quadratic bilinear operator that accounts for
the change of velocities of particles due to binary collisions (assuming
that the gas is rarefied, such that other collisions are negligible).

\subsubsection{Collision operator\label{S2.2.1}}

The (vector) collision operator $Q=\left( Q_{1},...,Q_{s}\right) $ has
components that can be written in the following form - reminding the
abbreviations $\left( \ref{a1}\right) $, 
\begin{equation*}
Q_{\alpha }(f,f)=\sum\limits_{\beta =1}^{s}\int_{\left( \mathbb{R}%
^{3}\right) ^{3}}W_{\alpha \beta }(\boldsymbol{\xi },\boldsymbol{\xi }_{\ast
}\left\vert \boldsymbol{\xi }^{\prime },\boldsymbol{\xi }_{\ast }^{\prime
}\right. )\,\left( f_{\alpha }^{\prime }f_{\beta \ast }^{\prime }-f_{\alpha
}f_{\beta \ast }\right) \,d\boldsymbol{\xi }_{\ast }d\boldsymbol{\xi }%
^{\prime }d\boldsymbol{\xi }_{\ast }^{\prime }\text{.}
\end{equation*}%
The transition probabilities are of the form cf. \cite[p.65]{Cercignani-88} 
\begin{eqnarray}
&&W_{\alpha \beta }(\boldsymbol{\xi },\boldsymbol{\xi }_{\ast }\left\vert 
\boldsymbol{\xi }^{\prime },\boldsymbol{\xi }_{\ast }^{\prime }\right. ) 
\notag \\
&=&2\left( m_{\alpha }+m_{\beta }\right) ^{2}m_{\alpha }m_{\beta }\sigma
_{\alpha \beta }\delta _{1}\left( m_{\alpha }\left\vert \boldsymbol{\xi }%
\right\vert ^{2}+m_{\beta }\left\vert \boldsymbol{\xi }_{\ast }\right\vert
^{2}-m_{\alpha }\left\vert \boldsymbol{\xi }^{\prime }\right\vert
^{2}-m_{\beta }\left\vert \boldsymbol{\xi }_{\ast }^{\prime }\right\vert
^{2}\right)  \notag \\
&&\times \delta _{3}\left( m_{\alpha }\boldsymbol{\xi }+m_{\beta }%
\boldsymbol{\xi }_{\ast }-m_{\alpha }\boldsymbol{\xi }^{\prime }-m_{\beta }%
\boldsymbol{\xi }_{\ast }^{\prime }\right) \text{, with }\sigma _{\alpha
\beta }=\sigma _{\alpha \beta }\left( \left\vert \mathbf{g}\right\vert ,\cos
\theta \right) >0\text{ a.e.,}  \notag \\
&&\text{ }\cos \theta =\frac{\mathbf{g}\cdot \mathbf{g}^{\prime }}{%
\left\vert \mathbf{g}\right\vert \left\vert \mathbf{g}^{\prime }\right\vert }%
\text{, \ }\mathbf{g}=\boldsymbol{\xi }-\boldsymbol{\xi }_{\ast }\text{, and 
}\mathbf{g}^{\prime }=\boldsymbol{\xi }^{\prime }-\boldsymbol{\xi }_{\ast
}^{\prime }\text{.}  \label{tp2a}
\end{eqnarray}%
where $\delta _{3}$ and $\delta _{1}$ denote the Dirac's delta function in $%
\mathbb{R}^{3}$ and $\mathbb{R}$, respectively; taking the conservation of
momentum and kinetic energy into account. The scattering cross sections $%
\sigma _{\alpha \beta }$, $\left\{ \alpha ,\beta \right\} \subseteq \left\{
1,...,s\right\} $, satisfy the symmetry relation%
\begin{equation}
\sigma _{\alpha \beta }=\sigma _{\beta \alpha }\text{,}  \label{sr2}
\end{equation}%
while 
\begin{equation}
\text{ }\sigma _{\alpha \alpha }=\sigma _{\alpha \alpha }\left( \left\vert 
\mathbf{g}\right\vert ,\left\vert \cos \theta \right\vert \right) \text{.}
\label{cf1}
\end{equation}%
Applying known properties of Dirac's delta function, the transition
probabilities may be transformed to 
\begin{eqnarray}
&&W_{\alpha \beta }(\boldsymbol{\xi },\boldsymbol{\xi }_{\ast }\left\vert 
\boldsymbol{\xi }^{\prime },\boldsymbol{\xi }_{\ast }^{\prime }\right. ) 
\notag \\
&=&2\left( m_{\alpha }+m_{\beta }\right) ^{3}\mu _{\alpha \beta }\delta
_{3}\left( \left( m_{\alpha }+m_{\beta }\right) \left( \mathbf{G}_{\alpha
\beta }-\mathbf{G}_{\alpha \beta }^{\prime }\right) \right) \delta
_{1}\left( \mu _{\alpha \beta }\left( \left\vert \mathbf{g}\right\vert
^{2}-\left\vert \mathbf{g}^{\prime }\right\vert ^{2}\right) \right)  \notag
\\
&=&\frac{\sigma _{\alpha \beta }}{\left\vert \mathbf{g}\right\vert }\delta
_{3}\left( \mathbf{G}_{\alpha \beta }-\mathbf{G}_{\alpha \beta }^{\prime
}\right) \delta _{1}\left( \left\vert \mathbf{g}\right\vert -\left\vert 
\mathbf{g}^{\prime }\right\vert \right) \text{, where}  \notag \\
&&\mu _{\alpha \beta }=\frac{m_{\alpha }m_{\beta }}{m_{\alpha }+m_{\beta }}%
\text{, }\mathbf{G}_{\alpha \beta }=\frac{m_{\alpha }\boldsymbol{\xi }%
+m_{\beta }\boldsymbol{\xi }_{\ast }}{m_{\alpha }+m_{\beta }}\text{, and }%
\mathbf{G}_{\alpha \beta }^{\prime }=\frac{m_{\alpha }\boldsymbol{\xi }%
^{\prime }+m_{\beta }\boldsymbol{\xi }_{\ast }^{\prime }}{m_{\alpha
}+m_{\beta }}\text{,}  \label{tp2}
\end{eqnarray}%
Due to invariance under change of particles in a collision and
microreversibility of the collisions, which follows directly by the
definition of the transition probability $\left( \ref{tp2a}\right) $, $%
\left( \ref{tp2}\right) $, $\left( \ref{cf1}\right) $, and the symmetry
relation $\left( \ref{sr2}\right) $ for the collision frequency, the
transition probabilities $\left( \ref{tp2a}\right) $ satisfy the relations%
\begin{eqnarray}
W_{\alpha \beta }(\boldsymbol{\xi },\boldsymbol{\xi }_{\ast }\left\vert 
\boldsymbol{\xi }^{\prime },\boldsymbol{\xi }_{\ast }^{\prime }\right. )
&=&W_{\beta \alpha }(\boldsymbol{\xi }_{\ast },\boldsymbol{\xi }\left\vert 
\boldsymbol{\xi }_{\ast }^{\prime },\boldsymbol{\xi }^{\prime }\right. ) 
\notag \\
W_{\alpha \beta }(\boldsymbol{\xi },\boldsymbol{\xi }_{\ast }\left\vert 
\boldsymbol{\xi }^{\prime },\boldsymbol{\xi }_{\ast }^{\prime }\right. )
&=&W_{\alpha \beta }(\boldsymbol{\xi }^{\prime },\boldsymbol{\xi }_{\ast
}^{\prime }\left\vert \boldsymbol{\xi },\boldsymbol{\xi }_{\ast }\right. ) 
\notag \\
W_{\alpha \alpha }(\boldsymbol{\xi },\boldsymbol{\xi }_{\ast }\left\vert 
\boldsymbol{\xi }^{\prime },\boldsymbol{\xi }_{\ast }^{\prime }\right. )
&=&W_{\alpha \alpha }(\boldsymbol{\xi },\boldsymbol{\xi }_{\ast }\left\vert 
\boldsymbol{\xi }_{\ast }^{\prime },\boldsymbol{\xi }^{\prime }\right. )%
\text{.}  \label{rel3}
\end{eqnarray}

\bigskip By a change of variables $\left\{ \boldsymbol{\xi }^{\prime },%
\boldsymbol{\xi }_{\ast }^{\prime }\right\} \rightarrow \left\{ \mathbf{g}%
^{\prime }=\boldsymbol{\xi }^{\prime }-\boldsymbol{\xi }_{\ast }^{\prime },%
\mathbf{G}_{\alpha \beta }^{\prime }=\dfrac{m_{\alpha }\boldsymbol{\xi }%
^{\prime }+m_{\beta }\boldsymbol{\xi }_{\ast }^{\prime }}{m_{\alpha
}+m_{\beta }}\right\} $, noting that%
\begin{equation}
d\boldsymbol{\xi }^{\prime }d\boldsymbol{\xi }_{\ast }^{\prime }=d\mathbf{G}%
_{\alpha \beta }^{\prime }d\mathbf{g}^{\prime }=\left\vert \mathbf{g}%
^{\prime }\right\vert ^{2}d\mathbf{G}_{\alpha \beta }^{\prime }d\left\vert 
\mathbf{g}^{\prime }\right\vert d\boldsymbol{\omega }\text{, }\boldsymbol{%
\omega }=\frac{\mathbf{g}^{\prime }}{\left\vert \mathbf{g}^{\prime
}\right\vert }\text{,}  \label{df2}
\end{equation}%
the observation that%
\begin{eqnarray*}
&&Q_{\alpha }(f,f) \\
&=&\sum\limits_{\beta =1}^{s}\int_{\left( \mathbb{R}^{3}\right) ^{2}\times 
\mathbb{R}_{+}\times \mathbb{S}^{2}}\!\!W_{\alpha \beta }(\boldsymbol{\xi },%
\boldsymbol{\xi }_{\ast }\left\vert \boldsymbol{\xi }^{\prime },\boldsymbol{%
\xi }_{\ast }^{\prime }\right. )\left( f_{\alpha }^{\prime }f_{\beta \ast
}^{\prime }-f_{\alpha }f_{\beta \ast }\right) \left\vert \mathbf{g}^{\prime
}\right\vert ^{2}d\boldsymbol{\xi }_{\ast }d\mathbf{G}_{\alpha \beta
}^{\prime }d\left\vert \mathbf{g}^{\prime }\right\vert d\boldsymbol{\omega }
\\
&=&\sum\limits_{\beta =1}^{s}\int_{\mathbb{R}\times \mathbb{S}%
^{2}}\left\vert \mathbf{g}\right\vert \,\sigma _{\alpha \beta }\left(
f_{\alpha }^{\prime }f_{\beta \ast }^{\prime }-f_{\alpha }f_{\beta \ast
}\right) \,d\boldsymbol{\xi }_{\ast }d\boldsymbol{\omega }\text{,}
\end{eqnarray*}%
where%
\begin{equation*}
\left\{ 
\begin{array}{c}
\boldsymbol{\xi }^{\prime }=\dfrac{m_{\alpha }\boldsymbol{\xi }+m_{\beta }%
\boldsymbol{\xi }_{\ast }}{m_{\alpha }+m_{\alpha _{j}}}+\dfrac{m_{\beta }}{%
m_{\alpha }+m_{\beta }}\left\vert \boldsymbol{\xi }-\boldsymbol{\xi }_{\ast
}\right\vert \omega =\mathbf{G}_{\alpha \beta }+\dfrac{m_{\beta }}{m_{\alpha
}+m_{\beta }}\left\vert \mathbf{g}\right\vert \omega \medskip \\ 
\boldsymbol{\xi }_{\ast }^{\prime }=\dfrac{m_{\alpha }\boldsymbol{\xi }%
+m_{\beta }\boldsymbol{\xi }_{\ast }}{m_{\alpha }+m_{\beta }}-\dfrac{%
m_{\alpha }}{m_{\alpha }+m_{\beta }}\left\vert \boldsymbol{\xi }-\boldsymbol{%
\xi }_{\ast }\right\vert \omega =\mathbf{G}_{\alpha \beta }-\dfrac{m_{\alpha
}}{m_{\alpha }+m_{\beta }}\left\vert \mathbf{g}\right\vert \omega%
\end{array}%
\right. \!\!\!\text{, }\omega \in S^{2}\!\text{,}
\end{equation*}%
can be made, resulting in a more familiar form of the Boltzmann collision
operator for mixtures.

\subsubsection{Collision invariants and Maxwellian distributions\label%
{S2.2.2}}

The following lemma follows directly by the relations $\left( \ref{rel3}%
\right) $.

\begin{lemma}
\label{L0a}The measures 
\begin{equation*}
dA_{\alpha \beta }=W_{\alpha \beta }(\boldsymbol{\xi },\boldsymbol{\xi }%
_{\ast }\left\vert \boldsymbol{\xi }^{\prime },\boldsymbol{\xi }_{\ast
}^{\prime }\right. )\,d\boldsymbol{\xi }\,d\boldsymbol{\xi }_{\ast }d%
\boldsymbol{\xi }^{\prime }d\boldsymbol{\xi }_{\ast }^{\prime }\text{, }%
\left\{ \alpha ,\beta \right\} \subseteq \left\{ 1,\ldots ,s\right\} \text{,}
\end{equation*}%
are invariant under the (ordered) interchange 
\begin{equation}
\left( \boldsymbol{\xi },\boldsymbol{\xi }_{\ast }\right) \leftrightarrow
\left( \boldsymbol{\xi }^{\prime },\boldsymbol{\xi }_{\ast }^{\prime }\right)
\label{tr1}
\end{equation}%
of variables, while%
\begin{equation*}
dA_{\alpha \beta }+dA_{\beta \alpha }\text{, }\left\{ \alpha ,\beta \right\}
\subseteq \left\{ 1,\ldots ,s\right\} \text{,}
\end{equation*}%
are invariant under the (ordered) interchange of variables%
\begin{equation}
\left( \boldsymbol{\xi },\boldsymbol{\xi }^{\prime }\right) \leftrightarrow
\left( \boldsymbol{\xi }_{\ast },\boldsymbol{\xi }_{\ast }^{\prime }\right) 
\text{.}  \label{tr1a}
\end{equation}
\end{lemma}

The weak form of the collision operator $Q(f,f)$ reads%
\begin{eqnarray*}
\left( Q(f,f),g\right) &=&\sum\limits_{\alpha ,\beta =1}^{s}\int_{\left( 
\mathbb{R}^{3}\right) ^{4}}\left( f_{\alpha }^{\prime }f_{\beta \ast
}^{\prime }-f_{\alpha }f_{\beta \ast }\right) g_{\alpha }\,dA_{\alpha \beta }
\\
&=&\sum\limits_{\alpha ,\beta =1}^{s}\int_{\left( \mathbb{R}^{3}\right)
^{4}}\left( f_{\alpha }^{\prime }f_{\beta \ast }^{\prime }-f_{\alpha
}f_{\beta \ast }\right) g_{\beta \ast }\,dA_{\alpha \beta } \\
&=&-\sum\limits_{\alpha ,\beta =1}^{s}\int_{\left( \mathbb{R}^{3}\right)
^{4}}\left( f_{\alpha }^{\prime }f_{\beta \ast }^{\prime }-f_{\alpha
}f_{\beta \ast }\right) g_{\alpha }^{\prime }\,dA_{\alpha \beta } \\
&=&-\sum\limits_{\alpha ,\beta =1}^{s}\int_{\left( \mathbb{R}^{3}\right)
^{4}}\left( f_{\alpha }^{\prime }f_{\beta \ast }^{\prime }-f_{\alpha
}f_{\beta \ast }\right) g_{\beta \ast }^{\prime }\,dA_{\alpha \beta },
\end{eqnarray*}%
for any function $g=\left( g_{1},...,g_{s}\right) $, such that the first
integrals are defined for all $\left\{ \alpha ,\beta \right\} \subseteq
\left\{ 1,\ldots ,s\right\} $, while the following equalities are obtained
by applying Lemma $\ref{L0a}$.

We have the following proposition.

\begin{proposition}
\label{P1a}Let $g=\left( g_{1},...,g_{s}\right) $ be such that 
\begin{equation*}
\int_{\left( \mathbb{R}^{3}\right) ^{4}}\left( f_{\alpha }^{\prime }f_{\beta
\ast }^{\prime }-f_{\alpha }f_{\beta \ast }\right) g_{\alpha }\,dA_{\alpha
\beta }
\end{equation*}%
is defined for any $\left\{ \alpha ,\beta \right\} \subseteq \left\{
1,\ldots ,s\right\} $. Then%
\begin{equation*}
\left( Q(f,f),g\right) =\frac{1}{4}\sum\limits_{\alpha ,\beta
=1}^{s}\int_{\left( \mathbb{R}^{3}\right) ^{4}}\left( f_{\alpha }^{\prime
}f_{\beta \ast }^{\prime }-f_{\alpha }f_{\beta \ast }\right) \left(
g_{\alpha }+g_{\beta \ast }-g_{\alpha }^{\prime }-g_{\beta \ast }^{\prime
}\right) \,dA_{\alpha \beta }.
\end{equation*}
\end{proposition}

\begin{definition}
A function $g=g\left( \boldsymbol{\xi },I\right) \ $is a collision invariant
if 
\begin{equation*}
\left( g_{\alpha }+g_{\beta \ast }-g_{\alpha }^{\prime }-g_{\beta \ast
}^{\prime }\right) W_{\alpha \beta }(\boldsymbol{\xi },\boldsymbol{\xi }%
_{\ast }\left\vert \boldsymbol{\xi }^{\prime },\boldsymbol{\xi }_{\ast
}^{\prime }\right. )=0\text{ a.e.}
\end{equation*}%
for all $\left\{ \alpha ,\beta \right\} \subseteq \left\{ 1,\ldots
,s\right\} $.
\end{definition}

It is clear that $e_{1},$ $...,$ $e_{s},$ $\mathbf{m}\xi _{x},$ $\mathbf{m}%
\xi _{y},$ $\mathbf{m}\xi _{z},$ and $\mathbf{m}\left\vert \boldsymbol{\xi }%
\right\vert ^{2}$, where $\left\{ e_{1},...,e_{s}\right\} $ is the standard
basis of $R^{s}$ and $\mathbf{m}=\left( m_{1},...,m_{s}\right) $, are
collision invariants - corresponding to conservation of mass(es), momentum,
and kinetic energy.

In fact, we have the following proposition, cf. \cite{GS-99, Cercignani-88}.

\begin{proposition}
\label{P2a}The vector space of collision invariants is generated by 
\begin{equation*}
\left\{ e_{1},...,e_{s},\mathbf{m}\xi _{x},\mathbf{m}\xi _{y},\mathbf{m}\xi
_{z},\mathbf{m}\left\vert \boldsymbol{\xi }\right\vert ^{2}\right\} \text{,}
\end{equation*}%
where $\mathbf{m}=\left( m_{1},...,m_{s}\right) $ and $\left\{
e_{1},...,e_{s}\right\} $ is the standard basis of $R^{s}$.
\end{proposition}

Define%
\begin{equation*}
\mathcal{W}\left[ f\right] :=\left( Q(f,f),\log f\right) .
\end{equation*}%
It follows by Proposition $\ref{P1a}$ that%
\begin{equation*}
\mathcal{W}\left[ f\right] =-\frac{1}{4}\sum\limits_{\alpha ,\beta
=1}^{s}\int_{\left( \mathbb{R}^{3}\right) ^{4}}f_{\alpha }f_{\beta \ast
}\left( \frac{f_{\alpha }^{\prime }f_{\beta \ast }^{\prime }}{f_{\alpha
}f_{\beta \ast }}-1\right) \log \left( \frac{f_{\alpha }^{\prime }f_{\beta
\ast }^{\prime }}{f_{\alpha }f_{\beta \ast }}\right) \,dA_{\alpha \beta
}\leq 0\text{,}
\end{equation*}%
with equality if and only if for all $\left\{ \alpha ,\beta \right\}
\subseteq \left\{ 1,\ldots ,s\right\} $ 
\begin{equation}
\left( f_{\alpha }f_{\beta \ast }-f_{\alpha }^{\prime }f_{\beta \ast
}^{\prime }\right) W_{\alpha \beta }(\boldsymbol{\xi },\boldsymbol{\xi }%
_{\ast }\left\vert \boldsymbol{\xi }^{\prime },\boldsymbol{\xi }_{\ast
}^{\prime }\right. )=0\text{ a.e.,}  \label{m1a}
\end{equation}%
or, equivalently, if and only if%
\begin{equation*}
Q(f,f)\equiv 0\text{.}
\end{equation*}

For any equilibrium, or Maxwellian, distribution $M=\left(
M_{1},...,M_{s}\right) $, it follows, by equation $\left( \ref{m1a}\right) $%
, since $Q(M,M)\equiv 0$, that for all $\left\{ \alpha ,\beta \right\}
\subseteq \left\{ 1,\ldots ,s\right\} $%
\begin{equation*}
\left( \log M_{\alpha }+\log M_{\beta \ast }-\log M_{\alpha }^{\prime }-\log
M\right) W_{\alpha \beta }(\boldsymbol{\xi },\boldsymbol{\xi }_{\ast
}\left\vert \boldsymbol{\xi }^{\prime },\boldsymbol{\xi }_{\ast }^{\prime
}\right. )=0\text{ a.e. .}
\end{equation*}%
Hence, $\log M=\left( \log M_{1},...,\log M_{s}\right) $ is a collision
invariant, and the components of the Maxwellian distributions $M=\left(
M_{1},...,M_{s}\right) $ are Gaussians 
\begin{equation*}
M_{\alpha }=n_{\alpha }\left( \frac{m_{\alpha }}{2\pi T}\right)
^{3/2}e^{-m_{\alpha }\left\vert \boldsymbol{\xi }-\mathbf{u}\right\vert
^{2}/\left( 2T\right) }\text{,}
\end{equation*}%
where $n_{\alpha }=\left( M,e_{\alpha }\right) $, $\mathbf{u}=\dfrac{1}{\rho 
}\left( M,\mathbf{m}\boldsymbol{\xi }\right) $, and $T=\dfrac{1}{3n}\left( M,%
\mathbf{m}\left\vert \boldsymbol{\xi }-\mathbf{u}\right\vert ^{2}\right) $,
with mass vector $\mathbf{m}=\left( m_{1},...,m_{s}\right) $, while $%
n=\sum\limits_{\alpha =1}^{s}n_{\alpha }$ and $\rho =\sum\limits_{\alpha
=1}^{s}m_{\alpha }n_{\alpha }$.

Note that by equation $\left( \ref{m1a}\right) $ any Maxwellian distribution 
$M=(M_{1},...,M_{s})$ satisfies the relations 
\begin{equation}
\left( M_{\alpha }^{\prime }M_{\beta \ast }^{\prime }-M_{\alpha }M_{\beta
\ast }\right) W_{\alpha \beta }(\boldsymbol{\xi },\boldsymbol{\xi }_{\ast
}\left\vert \boldsymbol{\xi }^{\prime },\boldsymbol{\xi }_{\ast }^{\prime
}\right. )=0.  \label{M2}
\end{equation}

\begin{remark}
Introducing the $\mathcal{H}$-functional%
\begin{equation*}
\mathcal{H}\left[ f\right] =\left( f,\log f\right) \text{,}
\end{equation*}%
an $\mathcal{H}$-theorem can be obtained.
\end{remark}

\subsubsection{Linearized collision operator\label{S2.2.3}}

Considering a deviation of a Maxwellian distribution $M=(M_{1},...,M_{s})$,
where the components are given by $M_{\alpha }=n_{\alpha }\left( \dfrac{%
m_{\alpha }}{2\pi }\right) ^{3/2}e^{-m_{\alpha }\left\vert \boldsymbol{\xi }%
\right\vert ^{2}/2}$, of the form $\left( \ref{s1}\right) $, results, by
insertion in the Boltzmann equation $\left( \ref{BE1}\right) $, in a system $%
\left( \ref{LBE}\right) $, where the components of the linearized collision
operator $\mathcal{L}=\left( \mathcal{L}_{1},...,\mathcal{L}_{s}\right) $
are given by%
\begin{eqnarray}
L_{\alpha }h &=&-\sum\limits_{\beta =1}^{s}M_{\alpha }^{-1/2}\left(
Q_{\alpha \beta }(M_{\alpha },M_{\beta }^{1/2}h_{\beta })+Q_{\alpha \beta
}(M_{\alpha }^{1/2}h_{\alpha },M_{\beta })\right)  \notag \\
&=&\sum\limits_{\beta =1}^{s}\int_{\left( \mathbb{R}^{3}\right) ^{3}}\left(
M_{\beta \ast }M_{\alpha }^{\prime }M_{\beta \ast }^{\prime }\right)
^{1/2}W_{\alpha \beta }(\boldsymbol{\xi },\boldsymbol{\xi }_{\ast
}\left\vert \boldsymbol{\xi }^{\prime },\boldsymbol{\xi }_{\ast }^{\prime
}\right. )\,  \notag \\
&&\times \left( \frac{h_{\alpha }}{M_{\alpha }^{1/2}}+\frac{h_{\beta \ast }}{%
M_{\beta \ast }^{1/2}}-\frac{h_{\alpha }^{\prime }}{\left( M_{\alpha
}^{\prime }\right) ^{1/2}}-\frac{h_{\beta \ast }^{\prime }}{\left( M_{\beta
\ast }^{\prime }\right) ^{1/2}}\right) \,d\boldsymbol{\xi }_{\ast }d%
\boldsymbol{\xi }^{\prime }d\boldsymbol{\xi }_{\ast }^{\prime }  \notag \\
&=&\upsilon _{\alpha }h_{\alpha }-K_{\alpha }\left( h\right) \text{, }
\label{dec5}
\end{eqnarray}%
where%
\begin{eqnarray}
\upsilon _{\alpha } &=&\sum\limits_{\beta =1}^{s}\int_{\left( \mathbb{R}%
^{3}\right) ^{3}}M_{\beta \ast }W_{\alpha \beta }(\boldsymbol{\xi },%
\boldsymbol{\xi }_{\ast }\left\vert \boldsymbol{\xi }^{\prime },\boldsymbol{%
\xi }_{\ast }^{\prime }\right. )\,d\boldsymbol{\xi }_{\ast }d\boldsymbol{\xi 
}^{\prime }d\boldsymbol{\xi }_{\ast }^{\prime },  \label{dec4} \\
K_{\alpha }\left( h\right) &=&\sum\limits_{\beta =1}^{s}\int_{\left( \mathbb{%
R}^{3}\right) ^{3}}\left( M_{\beta \ast }M_{\alpha }^{\prime }M_{\beta \ast
}^{\prime }\right) ^{1/2}W_{\alpha \beta }(\boldsymbol{\xi },\boldsymbol{\xi 
}_{\ast }\left\vert \boldsymbol{\xi }^{\prime },\boldsymbol{\xi }_{\ast
}^{\prime }\right. )\,  \notag \\
&&\times \left( \frac{h_{\alpha }^{\prime }}{\left( M_{\alpha }^{\prime
}\right) ^{1/2}}+\frac{h_{\beta \ast }^{\prime }}{\left( M_{\beta \ast
}^{\prime }\right) ^{1/2}}-\frac{h_{\beta \ast }}{M_{\beta \ast }^{1/2}}%
\right) \,d\boldsymbol{\xi }_{\ast }d\boldsymbol{\xi }^{\prime }d\boldsymbol{%
\xi }_{\ast }^{\prime }\text{,}  \notag
\end{eqnarray}%
while the components of the quadratic term $\Gamma =\left( \Gamma
_{1},...,\Gamma _{s}\right) $ are of the form $\left( \ref{nl1}\right) $.

The multiplication operator $\Lambda $ defined by 
\begin{equation*}
\Lambda (f)=\nu f\text{, where }\nu =\mathrm{diag}\left( \nu _{1},...,\nu
_{s}\right) ,
\end{equation*}%
is a closed, densely defined, self-adjoint operator on $\left( L^{2}\left( d%
\boldsymbol{\xi }\right) \right) ^{s}$. It is Fredholm, as well, if and only
if $\Lambda $ is coercive.

The following lemma follows immediately by Lemma $\ref{L0a}$.

\begin{lemma}
\label{L1a}For any $\left\{ \alpha ,\beta \right\} \subseteq \left\{
1,\ldots ,s\right\} $, the measure 
\begin{equation*}
d\widetilde{A}_{\alpha \beta }=\left( M_{\alpha }M_{\beta \ast }M_{\alpha
}^{\prime }M_{\beta \ast }^{\prime }\right) ^{1/2}dA_{\alpha \beta }
\end{equation*}%
is invariant under the (ordered) interchange $\left( \ref{tr1}\right) $ of
variables, while%
\begin{equation*}
d\widetilde{A}_{\alpha \beta }+d\widetilde{A}_{\beta \alpha }
\end{equation*}%
is invariant under the (ordered) interchange $\left( \ref{tr1a}\right) $ of
variables.
\end{lemma}

The weak form of the linearized collision operator $\mathcal{L}$ reads%
\begin{eqnarray*}
&&\left( \mathcal{L}h,g\right) \\
&=&\sum\limits_{\alpha ,\beta =1}^{s}\int_{\left( \mathbb{R}^{3}\right)
^{4}}\left( \frac{h_{\alpha }}{M_{\alpha }^{1/2}}+\frac{h_{\beta \ast }}{%
M_{\beta \ast }^{1/2}}-\frac{h_{\alpha }^{\prime }}{\left( M_{\alpha
}^{\prime }\right) ^{1/2}}-\frac{h_{\beta \ast }^{\prime }}{\left( M_{\beta
\ast }^{\prime }\right) ^{1/2}}\right) \frac{g_{\alpha }}{M_{\alpha }^{1/2}}%
\,d\widetilde{A}_{\alpha \beta } \\
&=&\sum\limits_{\alpha ,\beta =1}^{s}\int_{\left( \mathbb{R}^{3}\right)
^{4}}\left( \frac{h_{\alpha }}{M_{\alpha }^{1/2}}+\frac{h_{\beta \ast }}{%
M_{\beta \ast }^{1/2}}-\frac{h_{\alpha }^{\prime }}{\left( M_{\alpha
}^{\prime }\right) ^{1/2}}-\frac{h_{\beta \ast }^{\prime }}{\left( M_{\beta
\ast }^{\prime }\right) ^{1/2}}\right) \frac{g_{\beta \ast }}{M_{\beta \ast
}^{1/2}}\,d\widetilde{A}_{\alpha \beta } \\
&=&-\sum\limits_{\alpha ,\beta =1}^{s}\int_{\left( \mathbb{R}^{3}\right)
^{4}}\left( \frac{h_{\alpha }}{M_{\alpha }^{1/2}}+\frac{h_{\beta \ast }}{%
M_{\beta \ast }^{1/2}}-\frac{h_{\alpha }^{\prime }}{\left( M_{\alpha
}^{\prime }\right) ^{1/2}}-\frac{h_{\beta \ast }^{\prime }}{\left( M_{\beta
\ast }^{\prime }\right) ^{1/2}}\right) \frac{g_{\alpha }^{\prime }}{\left(
M_{\alpha }^{\prime }\right) ^{1/2}}\,d\widetilde{A}_{\alpha \beta } \\
&=&-\sum\limits_{\alpha ,\beta =1}^{s}\int_{\left( \mathbb{R}^{3}\right)
^{4}}\!\left( \!\frac{h_{\alpha }}{M_{\alpha }^{1/2}}+\frac{h_{\beta \ast }}{%
M_{\beta \ast }^{1/2}}-\frac{h_{\alpha }^{\prime }}{\left( M_{\alpha
}^{\prime }\right) ^{1/2}}-\frac{h_{\beta \ast }^{\prime }}{\left( M_{\beta
\ast }^{\prime }\right) ^{1/2}}\!\right) \!\frac{g_{\beta \ast }^{\prime }}{%
\left( M_{\beta \ast }^{\prime }\right) ^{1/2}}d\widetilde{A}_{\alpha \beta
}\!,
\end{eqnarray*}%
for any function $g=\left( g_{1},...,g_{s}\right) $, such that the first
integrals are defined for all $\left\{ \alpha ,\beta \right\} \subseteq
\left\{ 1,\ldots ,s\right\} $, while the following equalities are obtained
by applying Lemma $\ref{L1a}$.

We have the following lemma.

\begin{lemma}
\label{L2a}Let $g=\left( g_{1},...,g_{s}\right) $ be such that for any $%
\left\{ \alpha ,\beta \right\} \subseteq \left\{ 1,\ldots ,s\right\} $%
\begin{equation*}
\int_{\left( \mathbb{R}^{3}\right) ^{4}}\left( \frac{h_{\alpha }}{M_{\alpha
}^{1/2}}+\frac{h_{\beta \ast }}{M_{\beta \ast }^{1/2}}-\frac{h_{\alpha
}^{\prime }}{\left( M_{\alpha }^{\prime }\right) ^{1/2}}-\frac{h_{\beta \ast
}^{\prime }}{\left( M_{\beta \ast }^{\prime }\right) ^{1/2}}\right) \frac{%
g_{\alpha }}{M_{\alpha }^{1/2}}\,d\widetilde{A}_{\alpha \beta }
\end{equation*}%
is defined. Then%
\begin{eqnarray*}
\left( \mathcal{L}h,g\right) &=&\frac{1}{4}\sum\limits_{\alpha ,\beta
=1}^{s}\int_{\left( \mathbb{R}^{3}\right) ^{4}}\left( \frac{h_{\alpha }}{%
M_{\alpha }^{1/2}}+\frac{h_{\beta \ast }}{M_{\beta \ast }^{1/2}}-\frac{%
h_{\alpha }^{\prime }}{\left( M_{\alpha }^{\prime }\right) ^{1/2}}-\frac{%
h_{\beta \ast }^{\prime }}{\left( M_{\beta \ast }^{\prime }\right) ^{1/2}}%
\right) \\
&&\times \left( \frac{g_{\alpha }}{M_{\alpha }^{1/2}}+\frac{g_{\beta \ast }}{%
M_{\beta \ast }^{1/2}}-\frac{g_{\alpha }^{\prime }}{\left( M_{\alpha
}^{\prime }\right) ^{1/2}}-\frac{g_{\beta \ast }^{\prime }}{\left( M_{\beta
\ast }^{\prime }\right) ^{1/2}}\right) d\widetilde{A}_{\alpha \beta }.
\end{eqnarray*}
\end{lemma}

\begin{proposition}
The linearized collision operator is symmetric and nonnegative,%
\begin{equation*}
\left( \mathcal{L}h,g\right) =\left( h,\mathcal{L}g\right) \text{ and }%
\left( \mathcal{L}h,h\right) \geq 0\text{,}
\end{equation*}%
and\ the kernel of $\mathcal{L}$, $\ker \mathcal{L}$, is generated by%
\begin{equation*}
\left\{ \sqrt{M_{1}}e_{1},...,\sqrt{M_{s}}e_{s},\mathcal{M}^{1/2}\mathbf{m}%
v_{x},\mathcal{M}^{1/2}\mathbf{m}v_{y},\mathcal{M}^{1/2}\mathbf{m}v_{z},%
\mathcal{M}^{1/2}\mathbf{m}\left\vert \mathbf{v}\right\vert ^{2}\right\} ,
\end{equation*}%
where $\mathcal{M}=\mathrm{diag}\left( M_{1},...,M_{s}\right) $, $\mathbf{m}%
=\left( m_{1},...,m_{s}\right) $, and $\left\{ e_{1},...,e_{s}\right\} $ is
the standard basis of $\mathbb{R}^{s}$.
\end{proposition}

\begin{proof}
By Lemma $\ref{L2a}$, it is immediate that $\left( \mathcal{L}h,g\right)
=\left( h,\mathcal{L}g\right) $, and%
\begin{equation*}
\left( \mathcal{L}h,h\right) =\frac{1}{4}\sum\limits_{\alpha ,\beta
=1}^{s}\int_{\left( \mathbb{R}^{3}\right) ^{4}}\left( \frac{h_{\alpha }}{%
M_{\alpha }^{1/2}}+\frac{h_{\beta \ast }}{M_{\beta \ast }^{1/2}}-\frac{%
h_{\alpha }^{\prime }}{\left( M_{\alpha }^{\prime }\right) ^{1/2}}-\frac{%
h_{\beta \ast }^{\prime }}{\left( M_{\beta \ast }^{\prime }\right) ^{1/2}}%
\right) ^{2}d\widetilde{A}_{\alpha \beta }\geq 0.
\end{equation*}%
Furthermore, $h\in \ker \mathcal{L}$ if and only if $\left( \mathcal{L}%
h,h\right) =0$, which will be fulfilled\ if and only if for all $\left\{
\alpha ,\beta \right\} \subseteq \left\{ 1,\ldots ,s\right\} $%
\begin{equation*}
\left( \frac{h_{\alpha }}{M_{\alpha }^{1/2}}+\frac{h_{\beta \ast }}{M_{\beta
\ast }^{1/2}}-\frac{h_{\alpha }^{\prime }}{\left( M_{\alpha }^{\prime
}\right) ^{1/2}}-\frac{h_{\beta \ast }^{\prime }}{\left( M_{\beta \ast
}^{\prime }\right) ^{1/2}}\right) W_{\alpha \beta }(\boldsymbol{\xi },%
\boldsymbol{\xi }_{\ast }\left\vert \boldsymbol{\xi }^{\prime },\boldsymbol{%
\xi }_{\ast }^{\prime }\right. )=0\text{ a.e.,}
\end{equation*}%
i.e. if and only if $\mathcal{M}^{-1/2}h$ is a collision invariant. The last
part of the lemma now follows by Proposition $\ref{P2a}.$
\end{proof}

\begin{remark}
Note also that the quadratic term is orthogonal to the kernel of $\mathcal{L}
$, i.e. $\Gamma \left( h,h\right) \in \left( \ker \mathcal{L}\right) ^{\perp
_{\mathcal{\mathfrak{h}}^{(s)}}}$.
\end{remark}

\section{Main Results\label{S3}}

This section is devoted to the main results, concerning compact properties
in Theorems \ref{Thm1} and \ref{Thm3}, respectively, and bounds of collision
frequencies in Theorems \ref{Thm2} and \ref{Thm4}, respectively.

\subsection{Polyatomic molecules modeled by a discrete internal energy
variable\label{S3.1}}

Assume that for some positive number $\gamma $, such that $0<\gamma <1$,
there is a bound 
\begin{equation}
0\leq \sigma _{ij}^{kl}\left( \left\vert \mathbf{g}\right\vert ,\left\vert
\cos \theta \right\vert \right) \leq \frac{C}{\left\vert \mathbf{g}%
\right\vert ^{2}}\left( \Psi _{ij}^{kl}+\left( \Psi _{ij}^{kl}\right)
^{\gamma /2}\right) \text{, with }\Psi _{ij}^{kl}=\left\vert \mathbf{g}%
\right\vert \sqrt{\left\vert \mathbf{g}\right\vert ^{2}-\frac{4}{m}\Delta
I_{ij}^{kl}}\text{,}  \label{est1}
\end{equation}%
for $\left\vert \mathbf{g}\right\vert ^{2}>\frac{4}{m}\Delta I_{ij}^{kl}$ on
the scattering cross sections $\sigma _{ij}^{kl}$, $\left\{ i,j,k,l\right\}
\subseteq \left\{ 1,...,r\right\} \!$. Then the following result may be
obtained.

\begin{theorem}
\label{Thm1}Assume that the scattering cross sections $\sigma _{ij}^{kl}$, $%
\left\{ i,j,k,l\right\} \subseteq \left\{ 1,...,r\right\} $, satisfy the
bound $\left( \ref{est1}\right) $ for some positive number $\gamma $, such
that $0<\gamma <1$. Then the operator $K=\left( K_{1},...,K_{r}\right) $,
with the components $K_{i}$ given by $\left( \ref{dec1}\right) $ is a
self-adjoint compact operator on $\left( L^{2}\left( d\boldsymbol{\xi }%
\right) \right) ^{r}$.
\end{theorem}

Theorem \ref{Thm1} will be proven in Section $\ref{PT1}$.

\begin{corollary}
The linearized collision operator $\mathcal{L}$, with scattering cross
sections satisfying $\left( \ref{est1}\right) $, is a closed, densely
defined, self-adjoint operator on $\left( L^{2}\left( d\boldsymbol{\xi }%
\right) \right) ^{r}$.
\end{corollary}

\begin{proof}
By Theorem \ref{Thm1}, the linear operator $\mathcal{L}=\Lambda -K$, where
\linebreak $\Lambda (f)=\nu f$ for $\nu =\mathrm{diag}\left( \nu
_{1},...,\nu _{s}\right) $, is closed as the sum of a closed and a bounded
operator, and densely defined, since the domains of the linear operators $%
\mathcal{L}$ and $\Lambda $ are equal; $D(\mathcal{L})=D(\Lambda )$.
Furthermore, it is a self-adjoint operator, since the set of self-adjoint
operators is closed under addition of bounded self-adjoint operators, see
Theorem 4.3 of Chapter V in \cite{Kato}.
\end{proof}

Now consider the scattering cross sections - cf. hard sphere models - 
\begin{equation}
\sigma _{ij}^{kl}=C\dfrac{\sqrt{\left\vert \mathbf{g}\right\vert ^{2}-\frac{4%
}{m}\Delta I_{ij}^{kl}}}{\left\vert \mathbf{g}\right\vert \varphi
_{i}\varphi _{j}}\text{ if }\left\vert \mathbf{g}\right\vert ^{2}>\frac{4}{m}%
\Delta I_{ij}^{kl}\text{,}  \label{e1}
\end{equation}%
for some positive constant $C>0$ and all $\left\{ i,j,k,l\right\} \subseteq
\left\{ 1,...,r\right\} $.

In fact, it would be enough with the bounds 
\begin{equation}
C_{-}\dfrac{\sqrt{\left\vert \mathbf{g}\right\vert ^{2}-\frac{4}{m}\Delta
I_{ij}^{kl}}}{\left\vert \mathbf{g}\right\vert \varphi _{i}\varphi _{j}}\leq
\sigma _{ij}^{kl}\leq C_{+}\dfrac{\sqrt{\left\vert \mathbf{g}\right\vert
^{2}-\frac{4}{m}\Delta I_{ij}^{kl}}}{\left\vert \mathbf{g}\right\vert
\varphi _{i}\varphi _{j}}\text{ if }\left\vert \mathbf{g}\right\vert ^{2}>%
\frac{4}{m}\Delta I_{ij}^{kl}\text{,}  \label{ie1}
\end{equation}%
for some positive constants $C_{\pm }>0$, on the scattering cross sections.

\begin{theorem}
\label{Thm2} The linearized collision operator $\mathcal{L}$, with
scattering cross sections $\left( \ref{e1}\right) $ (or $\left( \ref{ie1}%
\right) $), can be split into a positive multiplication operator $\Lambda $,
defined by $\Lambda f=\nu f$, where $\nu =\nu (\left\vert \boldsymbol{\xi }%
\right\vert )=\mathrm{diag}\left( \nu _{1},...,\nu _{r}\right) $, minus a
compact operator $K$ on $\left( L^{2}\left( d\boldsymbol{\xi }\right)
\right) ^{r}$%
\begin{equation}
\mathcal{L}=\Lambda -K,  \label{dec3}
\end{equation}%
where there exist positive numbers $\nu _{-}$ and $\nu _{+}$, $0<\nu
_{-}<\nu _{+}$, such that for any $i\in \left\{ 1,...,r\right\} $ 
\begin{equation}
\nu _{-}\left( 1+\left\vert \boldsymbol{\xi }\right\vert \right) \leq \nu
_{i}(\left\vert \boldsymbol{\xi }\right\vert )\leq \nu _{+}\left(
1+\left\vert \boldsymbol{\xi }\right\vert \right) \text{ for all }%
\boldsymbol{\xi }\in \mathbb{R}^{3}\text{.}  \label{ine1}
\end{equation}
\end{theorem}

The decomposition $\left( \ref{dec3}\right) $ follows by the decomposition $%
\left( \ref{dec2}\right) ,\left( \ref{dec1}\right) $ and Theorem \ref{Thm1},
while the bounds $\left( \ref{ine1}\right) $ are proven in Section $\ref{PT2}
$.

\begin{corollary}
The linearized collision operator $\mathcal{L}$, with scattering cross
sections $\left( \ref{e1}\right) $ (or $\left( \ref{ie1}\right) $), is a
Fredholm operator, with domain%
\begin{equation*}
D(\mathcal{L})=\left( L^{2}\left( \left( 1+\left\vert \boldsymbol{\xi }%
\right\vert \right) d\boldsymbol{\xi }\right) \right) ^{r}\text{.}
\end{equation*}
\end{corollary}

\begin{proof}
By Theorem \ref{Thm2} the multiplication operator $\Lambda $ is coercive
and, hence, a Fredholm operator. The set of Fredholm operators is closed
under addition of compact operators, see Theorem 5.26 of Chapter IV in \cite%
{Kato} and its proof, so, by Theorem \ref{Thm2}, $\mathcal{L}$ is a Fredholm
operator.
\end{proof}

\begin{corollary}
For the linearized collision operator $\mathcal{L}$, with scattering cross
sections $\left( \ref{e1}\right) $ (or $\left( \ref{ie1}\right) $), there
exists a positive number $\lambda $, $0<\lambda <1$, such that 
\begin{equation*}
\left( h,\mathcal{L}h\right) \geq \lambda \left( h,\nu (\left\vert 
\boldsymbol{\xi }\right\vert )h\right) \geq \lambda \nu _{-}\left( h,\left(
1+\left\vert \boldsymbol{\xi }\right\vert \right) h\right)
\end{equation*}%
for any $h\in \left( L^{2}(\left( 1+\left\vert \boldsymbol{\xi }\right\vert
\right) d\boldsymbol{\xi })\right) ^{r}\cap \mathrm{Im}\mathcal{L}$.
\end{corollary}

\begin{proof}
Let $h\in \left( L^{2}(\left( 1+\left\vert \boldsymbol{\xi }\right\vert
\right) d\boldsymbol{\xi })\right) ^{r}\cap \left( \mathrm{ker}\mathcal{L}%
\right) ^{\perp }=\left( L^{2}(\left( 1+\left\vert \boldsymbol{\xi }%
\right\vert \right) d\boldsymbol{\xi })\right) ^{r}\cap \mathrm{Im}\mathcal{L%
}$. As a Fredholm operator, $\mathcal{L}$ is closed with a closed range, and
as a compact operator, $K$ is bounded, and so there are positive constants $%
\nu _{0}>0$ and $c_{K}>0$, such that%
\begin{equation*}
(h,Lh)\geq \nu _{0}(h,h)\text{ and }(h,Kh)\leq c_{K}(h,h).
\end{equation*}%
Let $\lambda =\dfrac{\nu _{0}}{\nu _{0}+c_{K}}$. Then%
\begin{eqnarray*}
(h,Lh) &=&(1-\lambda )(h,Lh)+\lambda (h,(\nu (|\boldsymbol{\xi }|)-K)h) \\
&\geq &(1-\lambda )\nu _{0}(h,h)+\lambda (h,\nu (|\boldsymbol{\xi }%
|)h)-\lambda c_{K}(h,h) \\
&=&(\nu _{0}-\lambda (\nu _{0}+c_{K}))(h,h)+\lambda (h,\nu (|\boldsymbol{\xi 
}|)h)=\lambda (h,\nu (|\boldsymbol{\xi }|)h)\text{.}
\end{eqnarray*}
\end{proof}

\subsection{Multicomponent mixtures\label{S3.2}}

Assume that for some positive number $\gamma $, such that $0<\gamma <1$,
there is a bound, cf. \cite{BGPS-13}, 
\begin{equation}
0\leq \sigma _{\alpha \beta }\left( \left\vert \mathbf{g}\right\vert ,\cos
\theta \right) \leq C\left( 1+\frac{1}{\left\vert \mathbf{g}\right\vert
^{2-\gamma }}\right) \text{, }0<\gamma <1\text{,}  \label{est2}
\end{equation}%
on the scattering cross sections $\sigma _{\alpha \beta }$, $\left\{ \alpha
,\beta \right\} \subseteq \left\{ 1,...,s\right\} $. Then the following
result may be obtained.

\begin{theorem}
\label{Thm3}Assume that the scattering cross sections $\sigma _{\alpha \beta
}$, $\left\{ \alpha ,\beta \right\} \subseteq \left\{ 1,...,s\right\} $,
satisfy the bound $\left( \ref{est2}\right) $ for some positive number $%
\gamma $, such that $0<\gamma <1$. Then the operator $K=\left(
K_{1},...,K_{s}\right) $, with the components $K_{\alpha }$ given by $\left( %
\ref{dec4}\right) $ is a self-adjoint compact operator on $\left(
L^{2}\left( d\boldsymbol{\xi }\right) \right) ^{s}$.
\end{theorem}

Theorem \ref{Thm3} will be proven in Section $\ref{PT3}$, but an alternative
proof can be found in \cite{BGPS-13}. Essential ideas in the proof are
inspired by \cite{BGPS-13}, while the approach differs.

\begin{corollary}
The linearized collision operator $\mathcal{L}$, with scattering cross
sections satisfying $\left( \ref{est2}\right) $, is a closed, densely
defined, self-adjoint operator on $\left( L^{2}\left( d\boldsymbol{\xi }%
\right) \right) ^{s}$.
\end{corollary}

Now consider a hard sphere model, i.e. such that 
\begin{equation}
\sigma _{\alpha \beta }=C_{\alpha \beta }  \label{e2}
\end{equation}%
for some positive constant $C_{\alpha \beta }>0$ for all $\left\{ \alpha
,\beta \right\} \subseteq \left\{ 1,...,s\right\} $.

In fact, it would be enough with the bounds 
\begin{equation}
C_{-}\leq \sigma _{\alpha \beta }\leq C_{+}  \label{ie2}
\end{equation}%
for some positive constants $C_{\pm }>0$ and all $\left\{ \alpha ,\beta
\right\} \subseteq \left\{ 1,...,s\right\} $, on the scattering cross
sections.

\begin{theorem}
\label{Thm4} \cite{BGPS-13} The linearized collision operator $\mathcal{L}$,
for a hard sphere model $\left( \ref{e2}\right) $ (or $\left( \ref{ie2}%
\right) $), can be split into a positive multiplication operator $\Lambda $,
where\linebreak\ $\Lambda f=\nu f$ for $\nu =\nu (\left\vert \boldsymbol{\xi 
}\right\vert )=\mathrm{diag}\left( \nu _{1},...,\nu _{s}\right) $, minus a
compact operator $K$ on $\left( L^{2}\left( d\boldsymbol{\xi }\right)
\right) ^{s}$%
\begin{equation}
\mathcal{L}=\Lambda -K,  \label{dec6}
\end{equation}%
where there exist positive numbers $\nu _{-}$ and $\nu _{+}$, $0<\nu
_{-}<\nu _{+}$, such that for all $\alpha \in \left\{ 1,...,s\right\} $%
\begin{equation}
\nu _{-}\left( 1+\left\vert \boldsymbol{\xi }\right\vert \right) \leq \nu
_{\alpha }(\left\vert \boldsymbol{\xi }\right\vert )\leq \nu _{+}\left(
1+\left\vert \boldsymbol{\xi }\right\vert \right) \text{ for all }%
\boldsymbol{\xi }\in \mathbb{R}^{3}\text{.}  \label{ine2}
\end{equation}
\end{theorem}

The decomposition $\left( \ref{dec6}\right) $ follows by the decomposition $%
\left( \ref{dec5}\right) ,\left( \ref{dec4}\right) $ and Theorem \ref{Thm3},
while the bounds $\left( \ref{ine2}\right) $ are proven in Section $\ref{PT4}
$.

\begin{corollary}
The linearized collision operator $\mathcal{L}$, for a hard-sphere model $%
\left( \ref{e2}\right) $ (or $\left( \ref{ie2}\right) $), is a Fredholm
operator, with domain%
\begin{equation*}
D(\mathcal{L})=\left( L^{2}\left( \left( 1+\left\vert \boldsymbol{\xi }%
\right\vert \right) d\boldsymbol{\xi }\right) \right) ^{s}\text{.}
\end{equation*}
\end{corollary}

\begin{corollary}
For the linearized collision operator $\mathcal{L}$, for a hard sphere model 
$\left( \ref{e2}\right) $ (or $\left( \ref{ie2}\right) $), there exists a
positive number $\lambda $, $0<\lambda <1$, such that \ 
\begin{equation*}
\left( h,\mathcal{L}h\right) \geq \lambda \left( h,\nu (\left\vert 
\boldsymbol{\xi }\right\vert )h\right) \geq \lambda \nu _{-}\left( h,\left(
1+\left\vert \boldsymbol{\xi }\right\vert \right) h\right)
\end{equation*}%
for any $h\in \left( L^{2}(\left( 1+\left\vert \boldsymbol{\xi }\right\vert
\right) d\boldsymbol{\xi })\right) ^{s}\cap \mathrm{Im}\mathcal{L}$.
\end{corollary}

\section{Compactness and Bounds on the Collision Frequency\label{S4}}

This section concerns the proofs of the main results presented in Theorem %
\ref{Thm1}-\ref{Thm4}. Note that throughout this section $C$ will denote a
generic positive constant.

To show the compactness properties we will apply the following result.

Denote, for any (non-zero) natural number $N$,%
\begin{equation*}
\mathfrak{h}_{N}:=\left\{ (\boldsymbol{\xi },\boldsymbol{\xi }_{\ast })\in
\left( \mathbb{R}^{3}\right) ^{2}:\left\vert \boldsymbol{\xi }-\boldsymbol{%
\xi }_{\ast }\right\vert \geq \frac{1}{N}\text{; }\left\vert \boldsymbol{\xi 
}\right\vert \leq N\right\}
\end{equation*}%
and%
\begin{equation*}
b^{(N)}=b^{(N)}(\boldsymbol{\xi },\boldsymbol{\xi }_{\ast }):=b(\boldsymbol{%
\xi },\boldsymbol{\xi }_{\ast })\mathbf{1}_{\mathfrak{h}_{N}}\text{.}
\end{equation*}%
Then we have the following lemma from \cite{Glassey}, that will be of
practical use for us to obtain compactness in this section .

\begin{lemma}
\label{LGD} (Glassey \cite[Lemma 3.5.1]{Glassey}, Drange \cite{Dr-75})

Assume that $b(\boldsymbol{\xi },\boldsymbol{\xi }_{\ast })\geq 0$ and let $%
Tf\left( \boldsymbol{\xi }\right) =\int_{\mathbb{R}^{3}}b(\boldsymbol{\xi },%
\boldsymbol{\xi }_{\ast })f\left( \boldsymbol{\xi }_{\ast }\right) \,d%
\boldsymbol{\xi }_{\ast }$.

Then $T$ is compact on $L^{2}\left( d\boldsymbol{\xi \,}\right) $ if

(i) $\int_{\mathbb{R}^{3}}b(\boldsymbol{\xi },\boldsymbol{\xi }_{\ast })\,d%
\boldsymbol{\xi }$ is bounded in $\boldsymbol{\xi }_{\ast }$;

(ii) $b^{(N)}\in L^{2}\left( d\boldsymbol{\xi \,}d\boldsymbol{\xi }_{\ast
}\right) $ for any (non-zero) natural number $N$;

(iii) $\underset{\boldsymbol{\xi }\in \mathbb{R}^{3}}{\sup }\int_{\mathbb{R}%
^{3}}b(\boldsymbol{\xi },\boldsymbol{\xi }_{\ast })-b^{(N)}(\boldsymbol{\xi }%
,\boldsymbol{\xi }_{\ast })\,d\boldsymbol{\xi }_{\ast }\rightarrow 0$ as $%
N\rightarrow \infty $.
\end{lemma}

Then we say that the kernel $b(\boldsymbol{\xi },\boldsymbol{\xi }_{\ast })$
is approximately Hilbert-Schmidt, while $T$ is an approximately
Hilbert-Schmidt operator. The reader is referred to Lemma 3.5.1 in \cite%
{Glassey} for a proof.

\subsection{Polyatomic molecules\label{S4.1}}

This section is devoted to the proofs of the compactness properties in
Theorem \ref{Thm1} and the bounds on the collision frequency in Theorem \ref%
{Thm2} of the linearized collision operator for polyatomic molecules modeled
with a discrete number of internal energies.

\subsubsection{\label{PT1} Compactness}

This section concerns the proof of Theorem \ref{Thm1}. Note that in the
proof the kernels are rewritten in such a way that $\boldsymbol{\xi }_{\ast
} $ - and not $\boldsymbol{\xi }^{\prime }$ and $\boldsymbol{\xi }_{\ast
}^{\prime }$ - always will be argument of the distribution functions. Then
there will be essentially two different types of kernels; either $%
\boldsymbol{\xi }_{\ast }$ is an argument in the loss term (as $\boldsymbol{%
\xi }$) or in the gain term (opposite to $\boldsymbol{\xi }$) of the
collision operator. The kernels of the terms from the loss part of the
collision operator will be shown\ to be Hilbert-Schmidt in a quite direct
way, while the kernels of the terms from the gain parts of the collision
operators will be shown to be approximately Hilbert-Schmidt in the sense of
Lemma \ref{LGD}.

\begin{proof}
For $i\in \left\{ 1,...,r\right\} $, rewrite expression $\left( \ref{dec1}%
\right) $ as 
\begin{eqnarray*}
K_{i}h &=&M_{i}^{-1/2}\sum\limits_{j,k,l=1}^{r}\int_{\left( \mathbb{R}%
^{3}\right) ^{3}}w(\boldsymbol{\xi },\boldsymbol{\xi }_{\ast
},I_{i},I_{j}\left\vert \boldsymbol{\xi }^{\prime },\boldsymbol{\xi }_{\ast
}^{\prime },I_{k},I_{l}\right. ) \\
&&\times \left( \frac{h_{k}^{\prime }}{\left( M_{k}^{\prime }\right) ^{1/2}}+%
\frac{h_{l\ast }^{\prime }}{\left( M_{l\ast }^{\prime }\right) ^{1/2}}-\frac{%
h_{j\ast }}{M_{j\ast }^{1/2}}\right) \,d\boldsymbol{\xi }_{\ast }d%
\boldsymbol{\xi }^{\prime }d\boldsymbol{\xi }_{\ast }^{\prime }\text{,}
\end{eqnarray*}%
with 
\begin{equation*}
w(\boldsymbol{\xi },\boldsymbol{\xi }_{\ast },I_{i},I_{j}\left\vert 
\boldsymbol{\xi }^{\prime },\boldsymbol{\xi }_{\ast }^{\prime
},I_{k},I_{l}\right. )=\left( \frac{M_{i}M_{j\ast }M_{k}^{\prime }M_{l\ast
}^{\prime }}{\varphi _{i}\varphi _{j}\varphi _{k}\varphi _{l}}\right)
^{1/2}W(\boldsymbol{\xi },\boldsymbol{\xi }_{\ast },I_{i},I_{j}\left\vert 
\boldsymbol{\xi }^{\prime },\boldsymbol{\xi }_{\ast }^{\prime
},I_{k},I_{l}\right. )\text{.}
\end{equation*}%
Due to relations $\left( \ref{rel1}\right) $, the relations%
\begin{eqnarray}
w(\boldsymbol{\xi },\boldsymbol{\xi }_{\ast },I_{i},I_{j}\left\vert 
\boldsymbol{\xi }^{\prime },\boldsymbol{\xi }_{\ast }^{\prime
},I_{k},I_{l}\right. ) &=&w(\boldsymbol{\xi }_{\ast },\boldsymbol{\xi }%
,I_{j},I_{i}\left\vert \boldsymbol{\xi }_{\ast }^{\prime },\boldsymbol{\xi }%
^{\prime },I_{l},I_{k}\right. )  \notag \\
w(\boldsymbol{\xi },\boldsymbol{\xi }_{\ast },I_{i},I_{j}\left\vert 
\boldsymbol{\xi }^{\prime },\boldsymbol{\xi }_{\ast }^{\prime
},I_{k},I_{l}\right. ) &=&w(\boldsymbol{\xi }^{\prime },\boldsymbol{\xi }%
_{\ast }^{\prime },I_{k},I_{l}\left\vert \boldsymbol{\xi },\boldsymbol{\xi }%
_{\ast },I_{i},I_{j}\right. )  \notag \\
w(\boldsymbol{\xi },\boldsymbol{\xi }_{\ast },I_{i},I_{j}\left\vert 
\boldsymbol{\xi }^{\prime },\boldsymbol{\xi }_{\ast }^{\prime
},I_{k},I_{l}\right. ) &=&w(\boldsymbol{\xi },\boldsymbol{\xi }_{\ast }%
\boldsymbol{,}I_{i},I_{j}\left\vert \boldsymbol{\xi }_{\ast }^{\prime },%
\boldsymbol{\xi }^{\prime },I_{l},I_{k}\right. )  \label{rel2}
\end{eqnarray}%
are satisfied. Hence, by first renaming $\left\{ \boldsymbol{\xi }_{\ast
},I_{j}\right\} \leftrightarrows \left\{ \boldsymbol{\xi }^{\prime
},I_{k}\right\} $ and then renaming $\left\{ \boldsymbol{\xi }_{\ast
},I_{j}\right\} \leftrightarrows \left\{ \boldsymbol{\xi }_{\ast }^{\prime
},I_{l}\right\} $, followed by applying the last relation in $\left( \ref%
{rel2}\right) $, 
\begin{eqnarray*}
&&\sum\limits_{j,k,l=1}^{r}\int_{\left( \mathbb{R}^{3}\right) ^{3}}w(%
\boldsymbol{\xi },\boldsymbol{\xi }_{\ast },I_{i},I_{j}\left\vert 
\boldsymbol{\xi }^{\prime },\boldsymbol{\xi }_{\ast }^{\prime
},I_{k},I_{l}\right. )\,\frac{h_{l\ast }^{\prime }}{\left( M_{l\ast
}^{\prime }\right) ^{1/2}}\,d\boldsymbol{\xi }_{\ast }d\boldsymbol{\xi }%
^{\prime }d\boldsymbol{\xi }_{\ast }^{\prime } \\
&=&\sum\limits_{j,k,l=1}^{r}\int_{\left( \mathbb{R}^{3}\right) ^{3}}w(%
\boldsymbol{\xi },\boldsymbol{\xi }^{\prime },I_{i},I_{k}\left\vert 
\boldsymbol{\xi }_{\ast },\boldsymbol{\xi }_{\ast }^{\prime
},I_{j},I_{l}\right. )\,\frac{h_{l\ast }^{\prime }}{\left( M_{l\ast
}^{\prime }\right) ^{1/2}}\,d\boldsymbol{\xi }_{\ast }d\boldsymbol{\xi }%
^{\prime }d\boldsymbol{\xi }_{\ast }^{\prime } \\
&=&\sum\limits_{j,k,l=1}^{r}\int_{\left( \mathbb{R}^{3}\right) ^{3}}w(%
\boldsymbol{\xi },\boldsymbol{\xi }^{\prime },I_{i},I_{k}\left\vert 
\boldsymbol{\xi }_{\ast }^{\prime },\boldsymbol{\xi }_{\ast
},I_{l},I_{j}\right. )\,\frac{h_{j\ast }}{M_{j\ast }^{1/2}}\,d\boldsymbol{%
\xi }_{\ast }d\boldsymbol{\xi }^{\prime }d\boldsymbol{\xi }_{\ast }^{\prime }
\\
&=&\sum\limits_{j,k,l=1}^{r}\int_{\left( \mathbb{R}^{3}\right) ^{3}}w(%
\boldsymbol{\xi },\boldsymbol{\xi }^{\prime },I_{i},I_{k}\left\vert 
\boldsymbol{\xi }_{\ast },\boldsymbol{\xi }_{\ast }^{\prime
},I_{j},I_{l}\right. )\,\frac{h_{j\ast }}{M_{j\ast }^{1/2}}\,d\boldsymbol{%
\xi }_{\ast }d\boldsymbol{\xi }^{\prime }d\boldsymbol{\xi }_{\ast }^{\prime }%
\text{.}
\end{eqnarray*}%
Moreover, by renaming $\left\{ \boldsymbol{\xi }_{\ast },I_{j}\right\}
\leftrightarrows \left\{ \boldsymbol{\xi }^{\prime },I_{k}\right\} $, 
\begin{eqnarray*}
&&\sum\limits_{j,k,l=1}^{r}\int_{\left( \mathbb{R}^{3}\right) ^{3}}w(%
\boldsymbol{\xi },\boldsymbol{\xi }_{\ast },I_{i},I_{j}\left\vert 
\boldsymbol{\xi }^{\prime },\boldsymbol{\xi }_{\ast }^{\prime
},I_{k},I_{l}\right. )\,\frac{h_{k}^{\prime }}{\left( M_{k}^{\prime }\right)
^{1/2}}\,d\boldsymbol{\xi }_{\ast }d\boldsymbol{\xi }^{\prime }d\boldsymbol{%
\xi }_{\ast }^{\prime } \\
&=&\sum\limits_{j,k,l=1}^{r}\int_{\left( \mathbb{R}^{3}\right) ^{3}}w(%
\boldsymbol{\xi },\boldsymbol{\xi }^{\prime },I_{i},I_{k}\left\vert 
\boldsymbol{\xi }_{\ast },\boldsymbol{\xi }_{\ast }^{\prime
},I_{j},I_{l}\right. )\,\frac{h_{j\ast }}{M_{j\ast }^{1/2}}\,d\boldsymbol{%
\xi }_{\ast }d\boldsymbol{\xi }^{\prime }d\boldsymbol{\xi }_{\ast }^{\prime }%
\text{.}
\end{eqnarray*}%
It follows that%
\begin{eqnarray}
K_{i}\left( h\right) &=&\sum\limits_{j=1}^{r}\int_{\mathbb{R}^{3}}k_{ij}(%
\boldsymbol{\xi },\boldsymbol{\xi }_{\ast })\,h_{j\ast }\,d\boldsymbol{\xi }%
_{\ast }\text{, where }  \notag \\
k_{ij}(\boldsymbol{\xi },\boldsymbol{\xi }_{\ast }) &=&k_{ij2}(\boldsymbol{%
\xi },\boldsymbol{\xi }_{\ast })-k_{ij1}(\boldsymbol{\xi },\boldsymbol{\xi }%
_{\ast })\text{, with}  \notag \\
k_{ij1}(\boldsymbol{\xi },\boldsymbol{\xi }_{\ast }) &=&\left( M_{i}M_{j\ast
}\right) ^{-1/2}\sum\limits_{k,l=1}^{r}\int_{\left( \mathbb{R}^{3}\right)
^{2}}w(\boldsymbol{\xi },\boldsymbol{\xi }_{\ast },I_{i},I_{j}\left\vert 
\boldsymbol{\xi }^{\prime },\boldsymbol{\xi }_{\ast }^{\prime
},I_{k},I_{l}\right. )\,d\boldsymbol{\xi }^{\prime }d\boldsymbol{\xi }_{\ast
}^{\prime }\text{ and}  \notag \\
k_{ij2}(\boldsymbol{\xi },\boldsymbol{\xi }_{\ast }) &=&2\left(
M_{i}M_{j\ast }\right) ^{-1/2}\sum\limits_{k,l=1}^{r}\int_{\left( \mathbb{R}%
^{3}\right) ^{2}}w(\boldsymbol{\xi },\boldsymbol{\xi }^{\prime
},I_{i},I_{k}\left\vert \boldsymbol{\xi }_{\ast },\boldsymbol{\xi }_{\ast
}^{\prime },I_{j},I_{l}\right. )\,d\boldsymbol{\xi }^{\prime }d\boldsymbol{%
\xi }_{\ast }^{\prime }\text{.}  \label{k1}
\end{eqnarray}%
Note that%
\begin{equation*}
k_{ij}(\boldsymbol{\xi },\boldsymbol{\xi }_{\ast })=k_{ji2}(\boldsymbol{\xi }%
_{\ast },\boldsymbol{\xi })-k_{ji1}(\boldsymbol{\xi }_{\ast },\boldsymbol{%
\xi })=k_{ji}(\boldsymbol{\xi }_{\ast },\boldsymbol{\xi }),
\end{equation*}%
since, by applying the first and the last relation in $\left( \ref{rel2}%
\right) $, 
\begin{eqnarray}
k_{ij1}(\boldsymbol{\xi },\boldsymbol{\xi }_{\ast }) &=&\left( M_{i}M_{j\ast
}\right) ^{-1/2}\sum\limits_{k,l=1}^{r}\int_{\left( \mathbb{R}^{3}\right)
^{2}}w(\boldsymbol{\xi },\boldsymbol{\xi }_{\ast },I_{i},I_{j}\left\vert 
\boldsymbol{\xi }^{\prime },\boldsymbol{\xi }_{\ast }^{\prime
},I_{k},I_{l}\right. )\,d\boldsymbol{\xi }^{\prime }d\boldsymbol{\xi }_{\ast
}^{\prime }  \notag \\
&=&\left( M_{i}M_{j\ast }\right) ^{-1/2}\sum\limits_{k,l=1}^{r}\int_{\left( 
\mathbb{R}^{3}\right) ^{2}}w(\boldsymbol{\xi }_{\ast },\boldsymbol{\xi ,}%
I_{j},I_{i}\left\vert \boldsymbol{\xi }_{\ast }^{\prime },\boldsymbol{\xi }%
^{\prime },I_{l},I_{k}\right. )\,d\boldsymbol{\xi }^{\prime }d\boldsymbol{%
\xi }_{\ast }^{\prime }  \notag \\
&=&\left( M_{i}M_{j\ast }\right) ^{-1/2}\sum\limits_{k,l=1}^{r}\int_{\left( 
\mathbb{R}^{3}\right) ^{2}}w(\boldsymbol{\xi }_{\ast },\boldsymbol{\xi ,}%
I_{j},I_{i}\left\vert \boldsymbol{\xi }^{\prime },\boldsymbol{\xi }_{\ast
}^{\prime },I_{k},I_{l}\right. )\,d\boldsymbol{\xi }^{\prime }d\boldsymbol{%
\xi }_{\ast }^{\prime }  \notag \\
&=&k_{ji1}(\boldsymbol{\xi }_{\ast },\boldsymbol{\xi })  \label{sa1}
\end{eqnarray}%
and, by applying the second relation in $\left( \ref{rel2}\right) $ and
renaming $\left\{ \boldsymbol{\xi }^{\prime },I_{k}\right\} \leftrightarrows
\left\{ \boldsymbol{\xi }_{\ast }^{\prime },I_{l}\right\} $, 
\begin{eqnarray}
k_{ij2}(\boldsymbol{\xi },\boldsymbol{\xi }_{\ast }) &=&2\left(
M_{i}M_{j\ast }\right) ^{-1/2}\sum\limits_{k,l=1}^{r}\int_{\left( \mathbb{R}%
^{3}\right) ^{2}}w(\boldsymbol{\xi },\boldsymbol{\xi }^{\prime
},I_{i},I_{k}\left\vert \boldsymbol{\xi }_{\ast },\boldsymbol{\xi }_{\ast
}^{\prime },I_{j},I_{l}\right. )\,d\boldsymbol{\xi }^{\prime }d\boldsymbol{%
\xi }_{\ast }^{\prime }  \notag \\
&=&2\left( M_{i}M_{j\ast }\right) ^{-1/2}\sum\limits_{k,l=1}^{r}\int_{\left( 
\mathbb{R}^{3}\right) ^{2}}w(\boldsymbol{\xi }_{\ast },\boldsymbol{\xi }%
_{\ast }^{\prime },I_{j},I_{l}\left\vert \boldsymbol{\xi },\boldsymbol{\xi }%
^{\prime },I_{i},I_{k}\right. )\,d\boldsymbol{\xi }^{\prime }d\boldsymbol{%
\xi }_{\ast }^{\prime }  \notag \\
&=&2\left( M_{i}M_{j\ast }\right) ^{-1/2}\sum\limits_{k,l=1}^{r}\int_{\left( 
\mathbb{R}^{3}\right) ^{2}}w(\boldsymbol{\xi }_{\ast },\boldsymbol{\xi }%
^{\prime },I_{j},I_{k}\left\vert \boldsymbol{\xi },\boldsymbol{\xi }_{\ast
}^{\prime },I_{i},I_{l}\right. )\,d\boldsymbol{\xi }^{\prime }d\boldsymbol{%
\xi }_{\ast }^{\prime }  \notag \\
&=&k_{ji2}(\boldsymbol{\xi }_{\ast },\boldsymbol{\xi })\text{.}  \label{sa2}
\end{eqnarray}

\begin{figure}[h]
\centering
\includegraphics[width=0.6\textwidth]{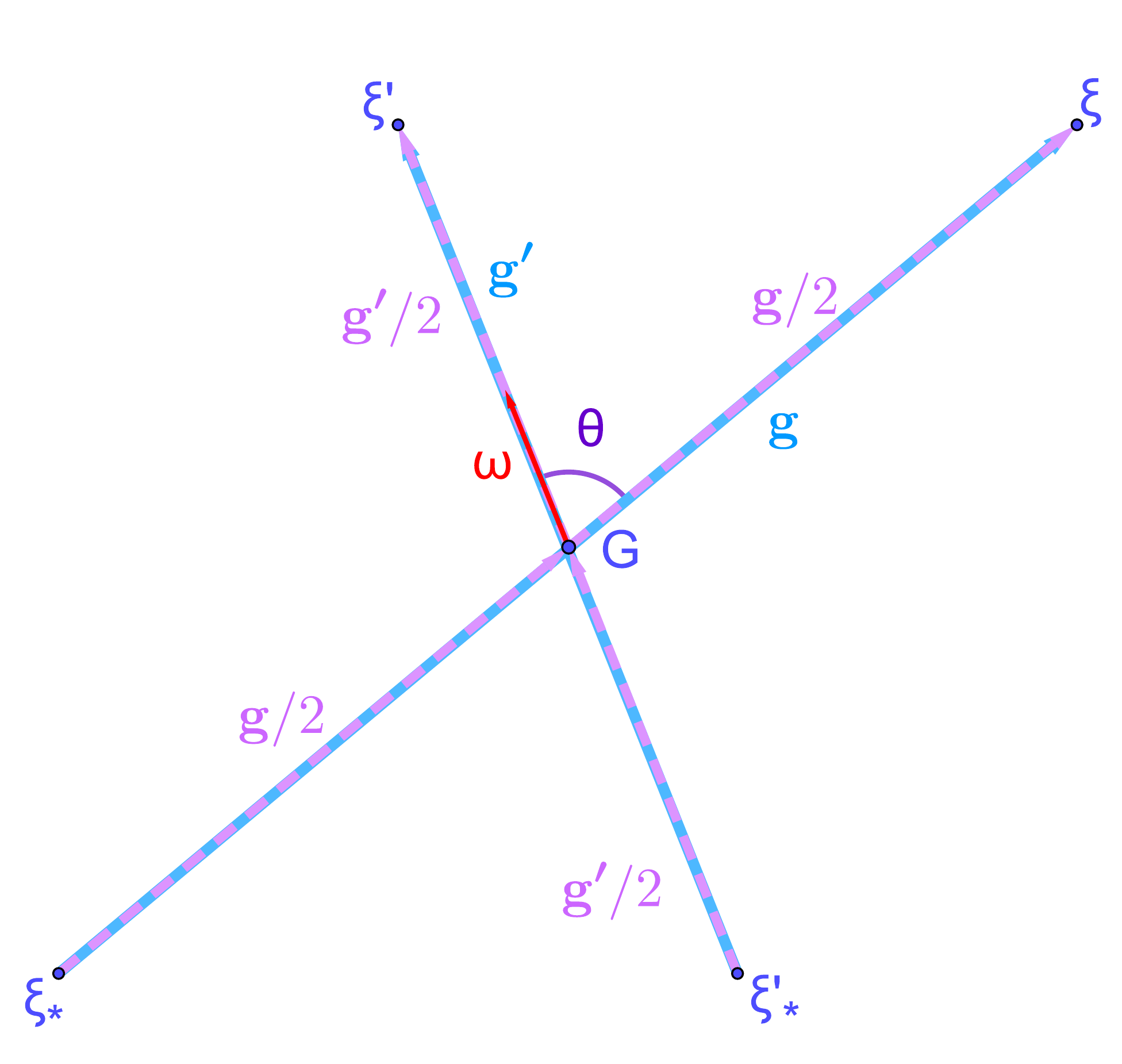}
\caption{Typical collision of $K_{ij}^{(1)}$. Classical representation of an
inelastic collision.}
\label{fig1}
\end{figure}

We now continue by proving the compactness for the two different types of
collision kernel separately. \ 

\textbf{I. Compactness of }$K_{ij}^{\left( 1\right) }=\int_{\mathbb{R}%
^{3}}k_{ij1}(\boldsymbol{\xi },\boldsymbol{\xi }_{\ast })\,h_{j\ast }\,d%
\boldsymbol{\xi }_{\ast }$.

By a change of variables $\left\{ \boldsymbol{\xi }^{\prime },\boldsymbol{%
\xi }_{\ast }^{\prime }\right\} \rightarrow \left\{ \mathbf{g}^{\prime }=%
\boldsymbol{\xi }^{\prime }-\boldsymbol{\xi }_{\ast }^{\prime },\mathbf{G}%
^{\prime }=\dfrac{\boldsymbol{\xi }^{\prime }+\boldsymbol{\xi }_{\ast
}^{\prime }}{2}\right\} $, cf. Figure $\ref{fig1}$, noting that $\left( \ref%
{df1}\right) $, and using relation $\left( \ref{M1}\right) $, expression $%
\left( \ref{k1}\right) $ of $k_{ij1}$ may be transformed to the following
form 
\begin{eqnarray*}
&&k_{ij1}(\boldsymbol{\xi },\boldsymbol{\xi }_{\ast }) \\
&=&\sum\limits_{k,l=1}^{r}\int\limits_{\mathbb{R}^{3}\times \mathbb{R}%
_{+}\times \mathbb{S}^{2}}\!\!\!\!\left( \frac{M_{k}^{\prime }M_{l\ast
}^{\prime }}{\varphi _{i}\varphi _{j}\varphi _{k}\varphi _{l}}\right)
^{1/2}\!\!\left\vert \mathbf{g}^{\prime }\right\vert ^{2}W(\boldsymbol{\xi },%
\boldsymbol{\xi }_{\ast },I_{i},I_{j}\left\vert \boldsymbol{\xi }^{\prime },%
\boldsymbol{\xi }_{\ast }^{\prime },I_{k},I_{l}\right. )d\mathbf{G}^{\prime
}d\left\vert \mathbf{g}^{\prime }\right\vert d\boldsymbol{\omega } \\
&=&\left( M_{i}M_{j\ast }\right) ^{1/2}\left\vert \mathbf{g}\right\vert
\sum\limits_{k,l=1}^{r}\int_{\mathbb{S}^{2}}\sigma _{ij}^{kl}\left(
\left\vert \mathbf{g}\right\vert ,\cos \theta \right) \mathbf{1}%
_{m\left\vert \mathbf{g}\right\vert >4\Delta I_{ij}^{kl}}\,d\boldsymbol{%
\omega }\text{.}
\end{eqnarray*}%
By assumption $\left( \ref{est1}\right) $ and 
\begin{equation*}
m\frac{\left\vert \boldsymbol{\xi }\right\vert ^{2}}{2}+m\frac{\left\vert 
\boldsymbol{\xi }_{\ast }\right\vert ^{2}}{2}+I_{i}+I_{j}=m\left\vert 
\mathbf{G}\right\vert ^{2}+m\frac{\left\vert \mathbf{g}\right\vert ^{2}}{4}%
+I_{i}+I_{j}=m\left\vert \mathbf{G}\right\vert ^{2}+E_{ij}\text{,}
\end{equation*}%
the bound%
\begin{eqnarray}
&&k_{ij1}^{2}(\boldsymbol{\xi },\boldsymbol{\xi }_{\ast })  \notag \\
&\leq &\frac{C}{\left\vert \mathbf{g}\right\vert ^{2}}M_{i}M_{j\ast }\left(
\int_{\mathbb{S}^{2}}\,d\boldsymbol{\omega }\right) ^{2}\left(
\sum\limits_{k,l=1}^{r}\left( \Psi _{ij}^{kl}+\left( \Psi _{ij}^{kl}\right)
^{\gamma /2}\right) \mathbf{1}_{m\left\vert \mathbf{g}\right\vert >4\Delta
I_{ij}^{kl}}\right) ^{2}\text{\textbf{\ }}  \notag \\
&\leq &\frac{C}{\left\vert \mathbf{g}\right\vert ^{2}}e^{-m\left\vert 
\mathbf{G}\right\vert ^{2}-E_{ij}}\!\!\left( \sum\limits_{k,l=1}^{r}\left(
1+\left\vert \mathbf{g}\right\vert ^{2}\right) \!\!\right) ^{2}\!\!=C\frac{%
\left( 1+\left\vert \mathbf{g}\right\vert ^{2}\right) ^{2}}{\left\vert 
\mathbf{g}\right\vert ^{2}}e^{-m\left\vert \mathbf{G}\right\vert
^{2}-E_{ij}}\!\text{.}  \label{b1}
\end{eqnarray}%
may be obtained. Then, by applying the bound $\left( \ref{b1}\right) $ and
first changing variables of integration $\left\{ \boldsymbol{\xi },%
\boldsymbol{\xi }_{\ast }\right\} \rightarrow \left\{ \mathbf{g},\mathbf{G}%
\right\} $, with unitary Jacobian, and then to spherical coordinates,%
\begin{eqnarray*}
\int_{\left( \mathbb{R}^{3}\right) ^{2}}k_{ij1}^{2}(\boldsymbol{\xi },%
\boldsymbol{\xi }_{\ast })d\boldsymbol{\xi }d\boldsymbol{\xi }_{\ast } &\leq
&C\int_{\left( \mathbb{R}^{3}\right) ^{2}}C\frac{\left( 1+\left\vert \mathbf{%
g}\right\vert ^{2}\right) ^{2}}{\left\vert \mathbf{g}\right\vert ^{2}}%
e^{-m\left\vert \mathbf{G}\right\vert ^{2}-E_{ij}}d\mathbf{g}d\mathbf{G} \\
&\leq &C\int_{0}^{\infty }R^{2}e^{-mR^{2}}dR\int_{0}^{\infty
}e^{-ms^{2}/4}\left( 1+s^{2}\right) ^{2}ds=C\text{.}
\end{eqnarray*}%
Hence,%
\begin{equation*}
K_{ij}^{\left( 1\right) }=\int_{\mathbb{R}^{3}}k_{ij1}(\boldsymbol{\xi },%
\boldsymbol{\xi }_{\ast })\,h_{j\ast }\,d\boldsymbol{\xi }_{\ast }
\end{equation*}%
are Hilbert-Schmidt integral operators and as such compact on $L^{2}\left( d%
\boldsymbol{\xi }\right) $, see e.g. Theorem 7.83 in \cite{RenardyRogers},
for all $\left\{ i,j\right\} \subseteq \left\{ 1,...,r\right\} $.

\begin{figure}[h]
\centering
\includegraphics[width=0.6\textwidth]{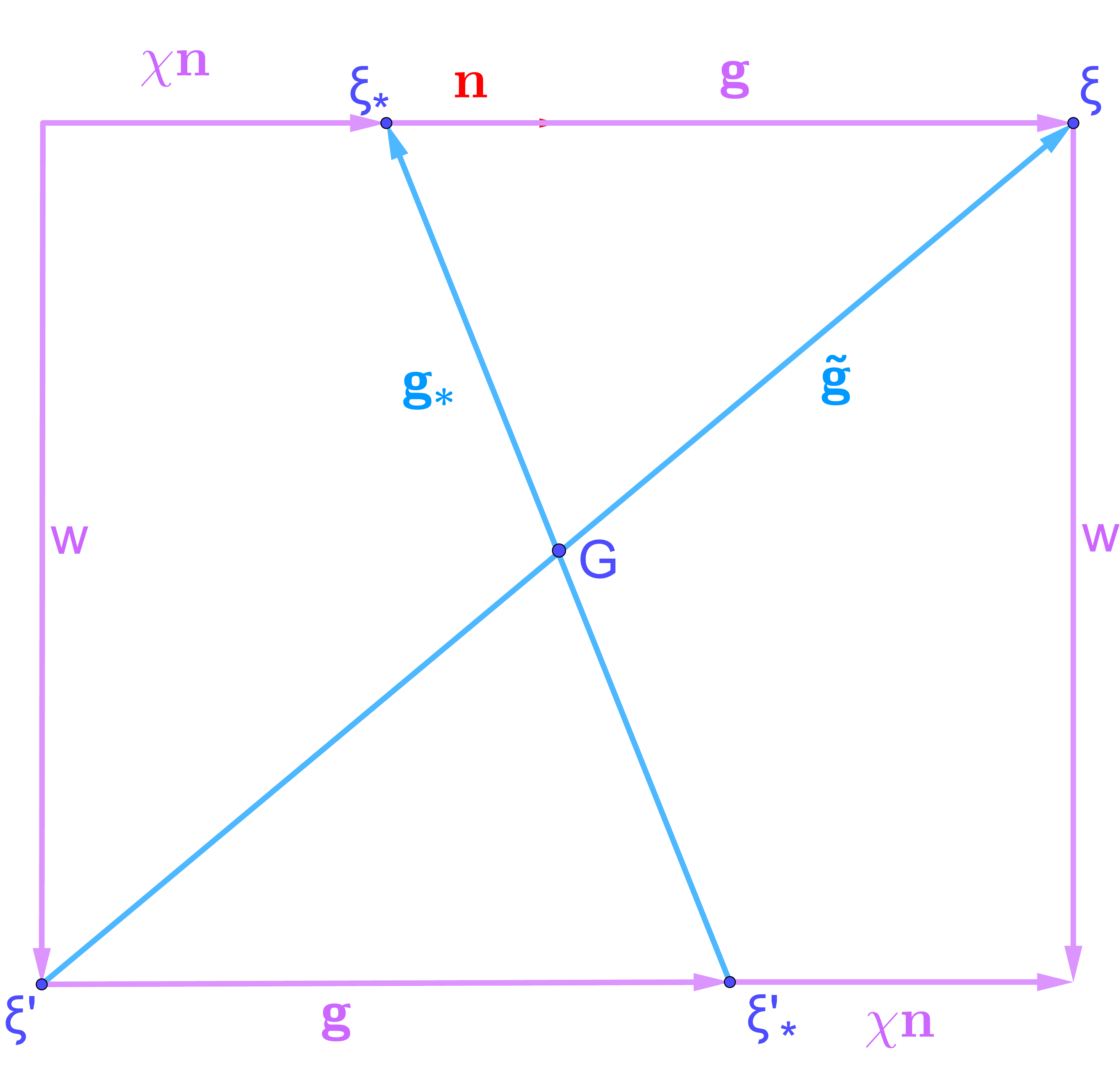}
\caption{Typical collision of $K_{ij}^{(2)}$.}
\label{fig2}
\end{figure}

\textbf{II. Compactness of }$K_{ij}^{\left( 2\right) }=\int_{\mathbb{R}%
^{3}}k_{ij2}(\boldsymbol{\xi },\boldsymbol{\xi }_{\ast })\,h_{j\ast }\,d%
\boldsymbol{\xi }_{\ast }$.

Noting that, see Figure $\ref{fig2}$,%
\begin{eqnarray*}
&&W(\boldsymbol{\xi },\boldsymbol{\xi }^{\prime },I_{i},I_{k}\left\vert 
\boldsymbol{\xi }_{\ast },\boldsymbol{\xi }_{\ast }^{\prime
},I_{j},I_{l}\right. ) \\
&=&4m\varphi _{i}\varphi _{k}\sigma _{ik}^{jl}\frac{\left\vert \widetilde{%
\mathbf{g}}\right\vert }{\left\vert \mathbf{g}_{\ast }\right\vert }\delta
_{3}\left( \boldsymbol{\xi }^{\prime }-\boldsymbol{\xi }_{\ast }^{\prime }+%
\boldsymbol{\xi }-\boldsymbol{\xi }_{\ast }\right) \\
&&\times \delta _{1}\left( \frac{m}{2}\left( \left\vert \boldsymbol{\xi }%
\right\vert ^{2}-\left\vert \boldsymbol{\xi }_{\ast }\right\vert
^{2}+\left\vert \boldsymbol{\xi }^{\prime }\right\vert ^{2}-\left\vert 
\boldsymbol{\xi }_{\ast }^{\prime }\right\vert ^{2}\right) -\Delta
I_{ik}^{jl}\right) \\
&=&4m\varphi _{i}\varphi _{k}\sigma _{ik}^{jl}\frac{\left\vert \widetilde{%
\mathbf{g}}\right\vert }{\left\vert \mathbf{g}_{\ast }\right\vert }\delta
_{3}\left( \mathbf{g}^{\prime }+\mathbf{g}\right) \delta _{1}\left( m\left( 
\boldsymbol{\xi }_{\ast }-\boldsymbol{\xi }^{\prime }\right) \cdot \mathbf{g}%
-\Delta I_{ik}^{jl}\right) \\
&=&4\varphi _{i}\varphi _{k}\sigma _{ik}^{jl}\frac{\left\vert \widetilde{%
\mathbf{g}}\right\vert }{\left\vert \mathbf{g}_{\ast }\right\vert \left\vert 
\mathbf{g}\right\vert }\delta _{3}\left( \mathbf{g}^{\prime }+\mathbf{g}%
\right) \delta _{1}\left( \chi -\frac{\Delta I_{ik}^{jl}}{m\left\vert 
\mathbf{g}\right\vert }\right) \text{, with }\chi =\left( \boldsymbol{\xi }%
_{\ast }-\boldsymbol{\xi }^{\prime }\right) \cdot \frac{\mathbf{g}}{%
\left\vert \mathbf{g}\right\vert }\text{,} \\
\mathbf{g} &=&\boldsymbol{\xi }-\boldsymbol{\xi }\mathbf{_{\ast }}\!\text{, }%
\mathbf{g}^{\prime }=\boldsymbol{\xi }^{\prime }-\boldsymbol{\xi }_{\ast
}^{\prime }\!\text{, }\widetilde{\mathbf{g}}=\boldsymbol{\xi }^{\prime }-%
\boldsymbol{\xi }\text{, }\mathbf{g}_{\ast }=\boldsymbol{\xi }_{\ast
}^{\prime }-\boldsymbol{\xi }_{\ast }\!\text{, }\Delta
I_{ik}^{jl}=I_{j}+I_{l}-I_{i}-I_{k}\text{,}
\end{eqnarray*}%
by the change of variables $\left\{ \boldsymbol{\xi }^{\prime },\boldsymbol{%
\xi }_{\ast }^{\prime }\right\} \rightarrow \left\{ \mathbf{g}^{\prime }=%
\boldsymbol{\xi }^{\prime }-\boldsymbol{\xi }_{\ast }^{\prime },~\mathbf{h}=%
\boldsymbol{\xi }^{\prime }-\boldsymbol{\xi }_{\ast }\right\} $, where 
\begin{equation*}
d\boldsymbol{\xi }^{\prime }d\boldsymbol{\xi }_{\ast }^{\prime }=d\mathbf{g}%
^{\prime }d\mathbf{h}=d\mathbf{g}^{\prime }d\chi d\mathbf{w}\text{, with }%
\mathbf{w}=\boldsymbol{\xi }^{\prime }-\boldsymbol{\xi }_{\ast }+\chi 
\mathbf{n}\text{ and }\mathbf{n}=\frac{\mathbf{g}}{\left\vert \mathbf{g}%
\right\vert }\text{,}
\end{equation*}%
the expression $\left( \ref{k1}\right) $ of $k_{ij2}$ may be transformed to%
\begin{eqnarray*}
&&k_{ij2}(\boldsymbol{\xi },\boldsymbol{\xi }_{\ast }) \\
&=&\sum\limits_{k,l=1}^{r}\int_{\left( \mathbb{R}^{3}\right) ^{2}}2\left( 
\frac{M_{k}^{\prime }M_{l\ast }^{\prime }}{\varphi _{i}\varphi _{j}\varphi
_{k}\varphi _{l}}\right) ^{1/2}W(\boldsymbol{\xi },\boldsymbol{\xi }^{\prime
},I_{i},I_{k}\left\vert \boldsymbol{\xi }_{\ast },\boldsymbol{\xi }_{\ast
}^{\prime },I_{j},I_{l}\right. )\,d\mathbf{g}^{\prime }d\mathbf{h} \\
&=&\frac{8\varphi _{i}^{1/2}}{\varphi _{j}^{1/2}}\sum\limits_{k,l=1}^{r}\int%
\limits_{\left( \mathbb{R}^{3}\right) ^{\perp _{\mathbf{n}}}}\!\!\frac{%
\varphi _{k}\left\vert \widetilde{\mathbf{g}}\right\vert }{\left\vert 
\mathbf{g}_{\ast }\right\vert \left\vert \mathbf{g}\right\vert }\mathbf{1}%
_{m\left\vert \widetilde{\mathbf{g}}\right\vert ^{2}>4\Delta
I_{ik}^{jl}}\left( \frac{M_{k}^{\prime }M_{l\ast }^{\prime }}{\varphi
_{k}\varphi _{l}}\right) ^{1/2}\!\!\!\sigma _{ik}^{jl}\left( \left\vert 
\widetilde{\mathbf{g}}\right\vert ,\frac{\widetilde{\mathbf{g}}\cdot \mathbf{%
g}_{\ast }}{\left\vert \widetilde{\mathbf{g}}\right\vert \left\vert \mathbf{g%
}_{\ast }\right\vert }\right) d\mathbf{w}\text{.\textbf{\ }}
\end{eqnarray*}%
Here, see Figure $\ref{fig2}$,%
\begin{equation*}
\left\{ 
\begin{array}{c}
\boldsymbol{\xi }^{\prime }=\boldsymbol{\xi }_{\ast }+\mathbf{w}-\chi 
\mathbf{n} \\ 
\boldsymbol{\xi }_{\ast }^{\prime }=\boldsymbol{\xi }+\mathbf{w}-\chi 
\mathbf{n}%
\end{array}%
\right. \text{, with }\chi =\chi _{ik}^{jl}=\chi _{ik}^{jl}\left( \left\vert 
\mathbf{g}\right\vert \right) =\frac{\Delta I_{ik}^{jl}}{m\left\vert \mathbf{%
g}\right\vert }\text{,}
\end{equation*}%
implying that%
\begin{eqnarray*}
\frac{\left\vert \boldsymbol{\xi }^{\prime }\right\vert ^{2}}{2}+\frac{%
\left\vert \boldsymbol{\xi }_{\ast }^{\prime }\right\vert ^{2}}{2}
&=&\left\vert \frac{\boldsymbol{\xi +\xi }_{\ast }}{2}-\chi _{ik}^{jl}%
\mathbf{n}+\mathbf{w}\right\vert ^{2}+\frac{\left\vert \boldsymbol{\xi }-%
\boldsymbol{\xi }_{\ast }\right\vert ^{2}}{4} \\
&=&\left\vert \frac{\left( \boldsymbol{\xi +\xi }_{\ast }\right) _{\perp _{%
\boldsymbol{n}}}}{2}+\mathbf{w}\right\vert ^{2}+\left( \frac{\left( 
\boldsymbol{\xi +\xi }_{\ast }\right) _{\mathbf{n}}}{2}-\chi
_{ik}^{jl}\right) ^{2}+\frac{\left\vert \mathbf{g}\right\vert ^{2}}{4} \\
&=&\left\vert \frac{\left( \boldsymbol{\xi +\xi }_{\ast }\right) _{\perp _{%
\boldsymbol{n}}}}{2}+\mathbf{w}\right\vert ^{2}\!+\frac{\left( \left\vert 
\boldsymbol{\xi }_{\ast }\right\vert ^{2}-\left\vert \boldsymbol{\xi }%
\right\vert ^{2}+2\chi _{ik}^{jl}\left\vert \boldsymbol{\xi }-\boldsymbol{%
\xi }_{\ast }\right\vert \right) ^{2}}{4\left\vert \boldsymbol{\xi }-%
\boldsymbol{\xi }_{\ast }\right\vert ^{2}}+\frac{\left\vert \mathbf{g}%
\right\vert ^{2}}{4}\!\text{,}
\end{eqnarray*}%
where%
\begin{eqnarray}
\left( \boldsymbol{\xi +\xi }_{\ast }\right) _{\mathbf{n}} &=&\left( 
\boldsymbol{\xi +\xi }_{\ast }\right) \cdot \mathbf{n}=\frac{\left\vert 
\boldsymbol{\xi }\right\vert ^{2}-\left\vert \boldsymbol{\xi }_{\ast
}\right\vert ^{2}}{\left\vert \boldsymbol{\xi }-\boldsymbol{\xi }_{\ast
}\right\vert }\text{ and}  \notag \\
\ \left( \boldsymbol{\xi +\xi }_{\ast }\right) _{\perp _{\boldsymbol{n}}} &=&%
\boldsymbol{\xi +\xi }_{\ast }-\left( \boldsymbol{\xi +\xi }_{\ast }\right)
_{\mathbf{n}}\mathbf{n}.  \label{o1}
\end{eqnarray}%
Hence, by assumption $\left( \ref{est1}\right) $, 
\begin{eqnarray}
&&k_{ij2}^{2}(\boldsymbol{\xi },\boldsymbol{\xi }_{\ast })  \notag \\
&\leq &\frac{C}{\left\vert \mathbf{g}\right\vert ^{2}}\exp \left( -\frac{m}{4%
}\left\vert \mathbf{g}\right\vert ^{2}\right) \left( \sum\limits_{k,l=1}^{r}%
\frac{1}{e^{\left( I_{k}+I_{l}\right) /2}}\exp \left( -m\frac{\left(
\left\vert \boldsymbol{\xi }_{\ast }\right\vert ^{2}-\left\vert \boldsymbol{%
\xi }\right\vert ^{2}+2\chi _{ik}^{jl}\left\vert \mathbf{g}\right\vert
\right) ^{2}}{8\left\vert \mathbf{g}\right\vert ^{2}}\right) \right.  \notag
\\
&&\times \!\!\!\!\!\left. \int\limits_{\left( \mathbb{R}^{3}\right) ^{\perp
_{\mathbf{n}}}}\!\!\!\!\!\mathbf{1}_{m\left\vert \widetilde{\mathbf{g}}%
\right\vert ^{2}>4\Delta I_{ik}^{jl}}\!\left( \!1+\frac{1}{\left( \widetilde{%
\Psi }_{ik}^{jl}\right) ^{1-\gamma /2}}\!\right) \!\exp \!\!\left( \!-\frac{m%
}{2}\left\vert \frac{\left( \boldsymbol{\xi +\xi }_{\ast }\right) _{\perp _{%
\boldsymbol{n}}}}{2}+\mathbf{w}\right\vert ^{2}\right) \!d\mathbf{w}%
\!\!\right) ^{2}  \notag \\
&\leq &\frac{C}{\left\vert \mathbf{g}\right\vert ^{2}}\exp \left( -\frac{m}{4%
}\left\vert \mathbf{g}\right\vert ^{2}\right) \left(
\sum\limits_{k,l=1}^{r}\exp \left( -m\frac{\left( \left\vert \boldsymbol{\xi 
}_{\ast }\right\vert ^{2}-\left\vert \boldsymbol{\xi }\right\vert ^{2}+2\chi
_{ik}^{jl}\left\vert \mathbf{g}\right\vert \right) ^{2}}{8\left\vert \mathbf{%
g}\right\vert ^{2}}\right) \right) ^{2}  \notag \\
&=&\frac{C}{\left\vert \mathbf{g}\right\vert ^{2}}\left(
\sum\limits_{k,l=1}^{r}\exp \left( -\frac{m}{8}\left( \left\vert \mathbf{g}%
\right\vert +2\left\vert \boldsymbol{\xi }\right\vert \cos \varphi +2\chi
_{ik}^{jl}\right) ^{2}-\frac{m}{8}\left\vert \mathbf{g}\right\vert
^{2}\right) \right) ^{2}  \notag \\
&\leq &\frac{C}{\left\vert \mathbf{g}\right\vert ^{2}}\sum%
\limits_{k,l=1}^{r}\exp \left( -m\left( \dfrac{\left\vert \mathbf{g}%
\right\vert }{2}+\left\vert \boldsymbol{\xi }\right\vert \cos \varphi +\chi
_{ik}^{jl}\right) ^{2}-\frac{m}{4}\left\vert \mathbf{g}\right\vert
^{2}\right) \text{, with}  \notag \\
&&\cos \varphi =\mathbf{n}\cdot \frac{\boldsymbol{\xi }}{\left\vert 
\boldsymbol{\xi }\right\vert }\text{, }\widetilde{\Psi }_{ik}^{jl}=\left%
\vert \widetilde{\mathbf{g}}\right\vert \left\vert \mathbf{g}_{\ast
}\right\vert \text{, and }\left\vert \mathbf{g}_{\ast }\right\vert
^{2}=\left\vert \widetilde{\mathbf{g}}\right\vert ^{2}-\frac{4}{m}\Delta
I_{ik}^{jl}\text{,}  \label{b2}
\end{eqnarray}%
since, 
\begin{eqnarray*}
&&\int_{\left( \mathbb{R}^{3}\right) ^{\perp _{\mathbf{n}}}}\!\!\mathbf{1}%
_{m\left\vert \widetilde{\mathbf{g}}\right\vert ^{2}>4\Delta
I_{ik}^{jl}}\!\left( \!1+\frac{1}{\left( \widetilde{\Psi }_{ik}^{jl}\right)
^{1-\gamma /2}}\!\right) \!\exp \!\!\left( \!-\frac{m}{2}\left\vert \frac{%
\left( \boldsymbol{\xi +\xi }_{\ast }\right) _{\perp _{\boldsymbol{n}}}}{2}+%
\mathbf{w}\right\vert ^{2}\right) \!d\mathbf{w} \\
&\leq &\int_{\left\vert \mathbf{w}\right\vert \leq 1}1+\left\vert \mathbf{w}%
\right\vert ^{\gamma -2}\,d\mathbf{w}+2\int_{\left\vert \mathbf{w}%
\right\vert \geq 1}\exp \left( -\frac{m}{2}\left\vert \frac{\left( 
\boldsymbol{\xi +\xi }_{\ast }\right) _{\perp _{\boldsymbol{n}}}}{2}+\mathbf{%
w}\right\vert ^{2}\right) \,d\mathbf{w} \\
&\leq &\int_{\left\vert \mathbf{w}\right\vert \leq 1}1+\left\vert \mathbf{w}%
\right\vert ^{\gamma -2}\,d\mathbf{w}+2\int_{\left( \mathbb{R}^{3}\right)
^{\perp _{\mathbf{n}}}}e^{-m\left\vert \widetilde{\mathbf{w}}\right\vert
^{2}/2}\,d\widetilde{\mathbf{w}} \\
&=&2\pi \left( \int_{0}^{1}R+R^{\gamma -1}\,dR+\int_{0}^{\infty
}Re^{-mR^{2}/2}\,dR\right) =C\text{.}
\end{eqnarray*}

To show that $k_{ij2}(\boldsymbol{\xi },\boldsymbol{\xi }_{\ast })\mathbf{1}%
_{\mathfrak{h}_{N}}\in L^{2}\left( d\boldsymbol{\xi \,}d\boldsymbol{\xi }%
_{\ast }\right) $ for any (non-zero) natural number $N$, separate the
integration domain\ of the integral of $k_{ij2}^{2}(\boldsymbol{\xi },%
\boldsymbol{\xi }_{\ast })$ over $\left( \mathbb{R}^{3}\right) ^{2}$ in two
separate domains $\left\{ \left( \mathbb{R}^{3}\right) ^{2}\text{; }%
\left\vert \mathbf{g}\right\vert \geq \left\vert \boldsymbol{\xi }%
\right\vert \right\} $ and $\left\{ \left( \mathbb{R}^{3}\right) ^{2}\text{; 
}\left\vert \mathbf{g}\right\vert \leq \left\vert \boldsymbol{\xi }%
\right\vert \right\} $. The integral of $k_{ij2}^{2}$ over the domain $%
\left\{ \left( \mathbb{R}^{3}\right) ^{2}\text{; }\left\vert \mathbf{g}%
\right\vert \geq \left\vert \boldsymbol{\xi }\right\vert \right\} $ will be
bounded, since 
\begin{eqnarray*}
\int_{\left\vert \mathbf{g}\right\vert \geq \left\vert \boldsymbol{\xi }%
\right\vert }k_{ij2}^{2}(\boldsymbol{\xi },\boldsymbol{\xi }_{\ast })\,d%
\boldsymbol{\xi }d\boldsymbol{\xi }_{\ast } &\leq &\int_{\left\vert \mathbf{g%
}\right\vert \geq \left\vert \boldsymbol{\xi }\right\vert }\frac{C}{%
\left\vert \mathbf{g}\right\vert ^{2}}e^{-m\left\vert \mathbf{g}\right\vert
^{2}/4}d\mathbf{g\,}d\boldsymbol{\xi } \\
&=&C\int_{0}^{\infty }\int_{\eta }^{\infty }e^{-R^{2}/4}\eta ^{2}dR\mathbf{\,%
}d\eta \\
&\leq &C\int_{0}^{\infty }e^{-R^{2}/8}dR\int_{0}^{\infty }e^{-\eta
^{2}/8}\eta ^{2}d\eta =C.
\end{eqnarray*}%
As for the second domain, consider the truncated domains $\left\{ \!\left( 
\mathbb{R}^{3}\right) ^{2}\!\!\text{;}\left\vert \mathbf{g}\right\vert \leq
\left\vert \boldsymbol{\xi }\right\vert \leq N\!\right\} $ for (non-zero)
natural numbers $N$. Then 
\begin{eqnarray*}
&&\int_{\left\vert \mathbf{g}\right\vert \leq \left\vert \boldsymbol{\xi }%
\right\vert \leq N}k_{ij2}^{2}(\boldsymbol{\xi },\boldsymbol{\xi }_{\ast
})\,d\boldsymbol{\xi }d\boldsymbol{\xi }_{\ast } \\
&\leq &\int_{\left\vert \mathbf{g}\right\vert \leq \left\vert \boldsymbol{%
\xi }\right\vert \leq N}\!\frac{C}{\left\vert \mathbf{g}\right\vert ^{2}}%
\sum\limits_{k,l=1}^{r}\exp \!\left( \!-m\left( \!\dfrac{\left\vert \mathbf{g%
}\right\vert }{2}+\left\vert \boldsymbol{\xi }\right\vert \cos \varphi +\chi
_{ik}^{jl}\left( \left\vert \mathbf{g}\right\vert \right) \!\right) ^{2}-%
\frac{m}{4}\left\vert \mathbf{g}\right\vert ^{2}\!\right) d\boldsymbol{\xi }d%
\mathbf{g} \\
&=&C\!\!\sum\limits_{k,l=1}^{r}\int_{0}^{N}\!\!\!\int_{0}^{\eta
}\!\!\int_{0}^{\pi }\!\!\eta ^{2}\exp \!\!\left( \!-m\!\left( \!\dfrac{R}{2}%
+\eta \cos \varphi +\chi _{ik}^{jl}\left( R\right) \!\!\right) ^{2}\!\!-%
\frac{m}{4}R^{2}\!\!\right) \!\sin \varphi \,d\varphi dRd\eta \\
&=&C\sum\limits_{k,l=1}^{r}\int_{0}^{N}\int_{0}^{\eta }\int_{R+2\chi
_{ik}^{jl}\left( R\right) -2\eta }^{R+2\chi _{ik}^{jl}\left( R\right) +2\eta
}\eta e^{-m\zeta ^{2}/4}e^{-mR^{2}/4}d\zeta \mathbf{\,}dR\mathbf{\,}d\eta \\
&\leq &C\int_{0}^{N}\eta \,d\eta \int_{0}^{\infty
}e^{-mR^{2}/4}dR\int_{-\infty }^{\infty }e^{-m\zeta ^{2}/4}d\zeta =CN^{2}%
\text{.}
\end{eqnarray*}

Furthermore, the integral of $k_{ij2}(\boldsymbol{\xi },\boldsymbol{\xi }%
_{\ast })$ with respect to $\boldsymbol{\xi }$ over $\mathbb{R}^{3}$ is
bounded in $\boldsymbol{\xi }_{\ast }$. Indeed, directly by the bound $%
\left( \ref{b2}\right) $ on $k_{ij2}^{2}$%
\begin{equation}
0\leq k_{ij2}(\boldsymbol{\xi },\boldsymbol{\xi }_{\ast })\leq \frac{C}{%
\left\vert \mathbf{g}\right\vert }\sum\limits_{k,l=1}^{r}\exp \left( -\frac{m%
}{8}\left( \left\vert \mathbf{g}\right\vert +2\left\vert \boldsymbol{\xi }%
\right\vert \cos \varphi +2\chi _{ik}^{jl}\right) ^{2}-\frac{m}{8}\left\vert 
\mathbf{g}\right\vert ^{2}\right) \text{.}  \label{b7}
\end{equation}%
Then the following bound on the integral of $k_{ij2}$ with respect to $%
\boldsymbol{\xi }_{\ast }$ over the domain $\left\{ \mathbb{R}^{3}\text{; }%
\left\vert \mathbf{g}\right\vert \geq \left\vert \boldsymbol{\xi }%
\right\vert \right\} $ can be obtained for $\left\vert \boldsymbol{\xi }%
\right\vert \neq 0$ 
\begin{eqnarray*}
\int_{\left\vert \mathbf{g}\right\vert \geq \left\vert \boldsymbol{\xi }%
\right\vert }k_{ij2}(\boldsymbol{\xi },\boldsymbol{\xi }_{\ast })\,d%
\boldsymbol{\xi }_{\ast } &\leq &\frac{C}{\left\vert \boldsymbol{\xi }%
\right\vert }\int_{\left\vert \mathbf{g}\right\vert \geq \left\vert 
\boldsymbol{\xi }\right\vert }e^{-m\left\vert \mathbf{g}\right\vert ^{2}/8}d%
\mathbf{g} \\
&=&\frac{C}{\left\vert \boldsymbol{\xi }\right\vert }\int_{\left\vert 
\boldsymbol{\xi }\right\vert }^{\infty }R^{2}e^{-mR^{2}/8}dR \\
&\leq &\frac{C}{\left\vert \boldsymbol{\xi }\right\vert }\int_{0}^{\infty
}R^{2}e^{-mR^{2}/8}\,dR=\frac{C}{\left\vert \boldsymbol{\xi }\right\vert }%
\text{,}
\end{eqnarray*}%
as well as, over the domain $\left\{ \mathbb{R}^{3}\text{; }\left\vert 
\mathbf{g}\right\vert \leq \left\vert \boldsymbol{\xi }\right\vert \right\} $
\begin{eqnarray*}
&&\int_{\left\vert \mathbf{g}\right\vert \leq \left\vert \boldsymbol{\xi }%
\right\vert }k_{ij2}(\boldsymbol{\xi },\boldsymbol{\xi }_{\ast })\,d%
\boldsymbol{\xi }_{\ast } \\
&\leq &\int_{\left\vert \mathbf{g}\right\vert \leq \left\vert \boldsymbol{%
\xi }\right\vert }\frac{C}{\left\vert \mathbf{g}\right\vert }%
\sum\limits_{k,l=1}^{r}\exp \left( -\frac{m}{8}\left( \left\vert \mathbf{g}%
\right\vert +2\left\vert \boldsymbol{\xi }\right\vert \cos \varphi +2\chi
_{ik}^{jl}\left( \left\vert \mathbf{g}\right\vert \right) \right) ^{2}-\frac{%
m}{8}\left\vert \mathbf{g}\right\vert ^{2}\right) d\mathbf{g} \\
&=&C\!\sum\limits_{k,l=1}^{r}\int_{0}^{\left\vert \boldsymbol{\xi }%
\right\vert }\!\!\!\int_{0}^{\pi }\!\!R\exp \!\!\left( \!-\frac{m}{8}%
\!\left( \!R+2\left\vert \boldsymbol{\xi }_{\ast }\right\vert \cos \varphi
+2\chi _{ik}^{jl}\left( R\right) \!\!\right) ^{2}\!-\frac{m}{8}%
R^{2}\!\right) \!\sin \varphi \,d\varphi dR \\
&=&\frac{C}{\left\vert \boldsymbol{\xi }\right\vert }\sum\limits_{k,l=1}^{r}%
\int_{0}^{\left\vert \boldsymbol{\xi }\right\vert }\int_{R+2\chi
_{ik}^{jl}\left( R\right) -2\left\vert \boldsymbol{\xi }\right\vert
}^{R+2\chi _{ik}^{jl}\left( R\right) +2\left\vert \boldsymbol{\xi }%
\right\vert }Re^{-m\eta ^{2}/8}e^{-mR^{2}/8}d\eta dR \\
&\leq &\frac{C}{\left\vert \boldsymbol{\xi }\right\vert }\int_{0}^{\infty
}Re^{-mR^{2}/8}\,dR\int_{-\infty }^{\infty }e^{-m\eta ^{2}/8}d\eta =\frac{C}{%
\left\vert \boldsymbol{\xi }\right\vert }\text{.}
\end{eqnarray*}%
However, due to the symmetry $k_{ij2}(\boldsymbol{\xi },\boldsymbol{\xi }%
_{\ast })=k_{ji2}(\boldsymbol{\xi }_{\ast },\boldsymbol{\xi })$ $\left( \ref%
{sa2}\right) $, also 
\begin{equation*}
\int_{\mathbb{R}^{3}}k_{ij2}(\boldsymbol{\xi },\boldsymbol{\xi }_{\ast })\,d%
\boldsymbol{\xi }\leq \frac{C}{\left\vert \boldsymbol{\xi }_{\ast
}\right\vert }\text{.}
\end{equation*}%
Therefore, if $\left\vert \boldsymbol{\xi }_{\ast }\right\vert \geq 1$, then%
\begin{equation*}
\int_{\mathbb{R}^{3}}k_{ij2}(\boldsymbol{\xi },\boldsymbol{\xi }_{\ast })\,d%
\boldsymbol{\xi }\leq \frac{C}{\left\vert \boldsymbol{\xi }_{\ast
}\right\vert }\leq C\text{.}
\end{equation*}%
Otherwise, if $\left\vert \boldsymbol{\xi }_{\ast }\right\vert \leq 1$,
then, by the bound $\left( \ref{b7}\right) $,%
\begin{eqnarray*}
&&\int_{\mathbb{R}^{3}}k_{ij2}(\boldsymbol{\xi },\boldsymbol{\xi }_{\ast
})\,d\boldsymbol{\xi =}\int_{\mathbb{R}^{3}}k_{ji2}(\boldsymbol{\xi }_{\ast
},\boldsymbol{\xi })\,d\boldsymbol{\xi } \\
&\leq &\int_{\mathbb{R}^{3}}\frac{C}{\left\vert \mathbf{g}\right\vert }%
\sum\limits_{k,l=1}^{r}\exp \left( -\frac{m}{8}\left( \left\vert \mathbf{g}%
\right\vert +2\left\vert \boldsymbol{\xi }_{\ast }\right\vert \cos \varphi
+2\chi _{ik}^{jl}\left( \left\vert \mathbf{g}\right\vert \right) \right)
^{2}-\frac{m}{8}\left\vert \mathbf{g}\right\vert ^{2}\right) \,d\mathbf{g} \\
&=&C\!\sum\limits_{k,l=1}^{r}\int_{0}^{\infty }\!\!\int_{0}^{\pi }\!\!R\exp
\!\!\left( \!-\frac{m}{8}\!\left( \!R+2\left\vert \boldsymbol{\xi }_{\ast
}\right\vert \cos \varphi +2\chi _{ik}^{jl}\left( R\right) \!\!\right)
^{2}\!\!-\frac{m}{8}R^{2}\!\right) \!\sin \varphi \,d\varphi dR \\
&\leq &\sum\limits_{k,l=1}^{r}\frac{C}{\left\vert \boldsymbol{\xi }_{\ast
}\right\vert }\int_{0}^{\infty }Re^{-mR^{2}/8}dR\int_{R+2\chi
_{ik}^{jl}\left( R\right) -2\left\vert \boldsymbol{\xi }_{\ast }\right\vert
}^{R+2\chi _{ik}^{jl}\left( R\right) +2\left\vert \boldsymbol{\xi }_{\ast
}\right\vert }\,d\eta =C\text{.}
\end{eqnarray*}%
Furthermore, 
\begin{eqnarray*}
&&\sup_{\boldsymbol{\xi }\in \mathbb{R}^{3}}\int_{\mathbb{R}^{3}}k_{ij2}(%
\boldsymbol{\xi },\boldsymbol{\xi }_{\ast })-k_{ij2}(\boldsymbol{\xi },%
\boldsymbol{\xi }_{\ast })\mathbf{1}_{\mathfrak{h}_{N}}\,d\boldsymbol{\xi }%
_{\ast } \\
&\leq &\sup_{\boldsymbol{\xi }\in \mathbb{R}^{3}}\int_{\left\vert \mathbf{g}%
\right\vert \leq \frac{1}{N}}k_{ij2}(\boldsymbol{\xi },\boldsymbol{\xi }%
_{\ast })\,d\boldsymbol{\xi }_{\ast }+\sup_{\left\vert \boldsymbol{\xi }%
\right\vert \geq N}\int_{\mathbb{R}^{3}}k_{ij2}(\boldsymbol{\xi },%
\boldsymbol{\xi }_{\ast })\,d\boldsymbol{\xi }_{\ast } \\
&\leq &\int_{\left\vert \mathbf{g}\right\vert \leq \frac{1}{N}}\frac{C}{%
\left\vert \mathbf{g}\right\vert }\,d\mathbf{g}+\frac{C}{N}\leq C\left(
\int_{0}^{\frac{1}{N}}R\,dR+\frac{1}{N}\right) \\
&=&C\left( \frac{1}{N^{2}}+\frac{1}{N}\right) \rightarrow 0\text{ as }%
N\rightarrow \infty \text{.}
\end{eqnarray*}%
Hence, by Lemma \ref{LGD} the operators 
\begin{equation*}
K_{ij}^{(2)}=\int_{\mathbb{R}^{3}}k_{ij2}(\boldsymbol{\xi },\boldsymbol{\xi }%
_{\ast })\,h_{j\ast }\,d\boldsymbol{\xi }_{\ast }
\end{equation*}%
are compact on $L^{2}\left( d\boldsymbol{\xi }\right) $ for all $\left\{
i,j\right\} \subseteq \left\{ 1,...,r\right\} $.

Concluding, the operator 
\begin{equation*}
K=(K_{1},...,K_{r})=\sum%
\limits_{j=1}^{r}(K_{1j}^{(2)},...,K_{rj}^{(2)})-(K_{1j}^{(1)},...,K_{rj}^{(1)})
\end{equation*}%
is a compact self-adjoint operator on $\left( L^{2}\left( d\boldsymbol{\xi }%
\right) \right) ^{r}$. The self-adjointness is due to the symmetry relations 
$\left( \ref{sa1}\right) ,\left( \ref{sa2}\right) $, cf. \cite[p.198]%
{Yoshida-65}.
\end{proof}

\subsubsection{\label{PT2} Bounds on the collision frequency}

This section concerns the proof of Theorem \ref{Thm2}.

\begin{proof}
Under assumption $\left( \ref{e1}\right) $ each collision frequency $%
\upsilon _{1},...,\upsilon _{r}$ equals%
\begin{eqnarray*}
\upsilon _{i} &=&\sum\limits_{j,k,l=1}^{r}\int_{\left( \mathbb{R}^{3}\right)
^{3}}\frac{M_{j\ast }}{\varphi _{i}\varphi _{j}}W(\boldsymbol{\xi },%
\boldsymbol{\xi }_{\ast },I_{i},I_{j}\left\vert \boldsymbol{\xi }^{\prime },%
\boldsymbol{\xi }_{\ast }^{\prime },I_{k},I_{l}\right. )\,d\boldsymbol{\xi }%
_{\ast }d\boldsymbol{\xi }^{\prime }d\boldsymbol{\xi }_{\ast }^{\prime } \\
&=&\sum\limits_{j,k,l=1}^{r}\int_{\left( \mathbb{R}^{3}\right) ^{3}}M_{j\ast
}\sigma _{ij}^{kl}\frac{\left\vert \mathbf{g}\right\vert }{\left\vert 
\mathbf{g}^{\prime }\right\vert ^{2}}\delta _{3}\left( \mathbf{G}-\mathbf{G}%
^{\prime }\right) \mathbf{1}_{m\left\vert \mathbf{g}\right\vert >4\Delta
I_{ij}^{kl}} \\
&&\times \delta _{1}\left( \sqrt{\left\vert \mathbf{g}\right\vert ^{2}-\frac{%
4}{m}\Delta I_{ij}^{kl}}-\left\vert \mathbf{g}^{\prime }\right\vert \right)
\,d\boldsymbol{\xi }_{\ast }d\mathbf{G}^{\prime }d\mathbf{g}^{\prime } \\
&=&\frac{C}{\varphi _{i}}\sum\limits_{j,k,l=1}^{r}e^{-I_{j}}\int_{\mathbb{R}%
^{3}}e^{-m\left\vert \boldsymbol{\xi }_{\ast }\right\vert ^{2}/2}\sigma
_{ij}^{kl}\left\vert \mathbf{g}\right\vert \mathbf{1}_{m\left\vert \mathbf{g}%
\right\vert >4\Delta I_{ij}^{kl}}\,d\boldsymbol{\xi }_{\ast }\int_{\mathbb{S}%
^{2}}\,d\boldsymbol{\omega } \\
&=&\frac{C}{\varphi _{i}}\sum\limits_{j,k,l=1}^{r}e^{-I_{j}}\int_{\mathbb{R}%
^{3}}e^{-m\left\vert \boldsymbol{\xi }_{\ast }\right\vert ^{2}/2}\sqrt{%
\left\vert \mathbf{g}\right\vert ^{2}-\frac{4}{m}\Delta I_{ij}^{kl}}\mathbf{1%
}_{m\left\vert \mathbf{g}\right\vert >4\Delta I_{ij}^{kl}}\,d\boldsymbol{\xi 
}_{\ast }\text{.}
\end{eqnarray*}%
Given $i\in \left\{ 1,...,r\right\} $, there are $\left\{ j,k,l\right\}
\subset \left\{ 1,...,r\right\} $, such that $\Delta I_{ij}^{kl}\leq 0$.
Assuming that $\Delta I_{ij}^{kl}\leq 0$ for some fixed $\left\{
j,k,l\right\} \subset \left\{ 1,...,r\right\} $, implies the inequality%
\begin{eqnarray*}
\upsilon _{i} &\geq &\frac{Ce^{-I_{j}}}{\varphi _{i}}\int_{\mathbb{R}%
^{3}}e^{-m\left\vert \boldsymbol{\xi }_{\ast }\right\vert ^{2}/2}\sqrt{%
\left\vert \mathbf{g}\right\vert ^{2}-\frac{4}{m}\Delta I_{ij}^{kl}}\,d%
\boldsymbol{\xi }_{\ast } \\
&\geq &C\int_{\mathbb{R}^{3}}\left\vert \mathbf{g}\right\vert
e^{-m\left\vert \boldsymbol{\xi }_{\ast }\right\vert ^{2}/2}\,d\boldsymbol{%
\xi }_{\ast }\text{.}
\end{eqnarray*}

Now consider two different cases separately: $\left\vert \boldsymbol{\xi }%
\right\vert \leq 1$ and $\left\vert \boldsymbol{\xi }\right\vert \geq 1$.

Firstly, if $\left\vert \boldsymbol{\xi }\right\vert \leq 1$, then%
\begin{eqnarray*}
\upsilon _{i} &\geq &C\int_{\mathbb{R}^{3}}\left( \left\vert \left\vert 
\boldsymbol{\xi }\right\vert -\left\vert \boldsymbol{\xi }_{\ast
}\right\vert \right\vert \right) e^{-m\left\vert \boldsymbol{\xi }_{\ast
}\right\vert ^{2}/2}\,d\boldsymbol{\xi }_{\ast } \\
&\geq &C\int_{\left\vert \boldsymbol{\xi }_{\ast }\right\vert \geq 2}\left(
\left\vert \boldsymbol{\xi }_{\ast }\right\vert -\left\vert \boldsymbol{\xi }%
\right\vert \right) e^{-m\left\vert \boldsymbol{\xi }_{\ast }\right\vert
^{2}/2}\,d\boldsymbol{\xi }_{\ast } \\
&\geq &C\int_{\left\vert \boldsymbol{\xi }_{\ast }\right\vert \geq 2}\frac{1%
}{2}\left\vert \boldsymbol{\xi }_{\ast }\right\vert e^{-m\left\vert 
\boldsymbol{\xi }_{\ast }\right\vert ^{2}/2}\,d\boldsymbol{\xi }_{\ast } \\
&=&C\int_{1}^{\infty }R^{3}e^{-2mR^{2}}\,dR\int_{\mathbb{S}^{2}}\,d%
\boldsymbol{\omega }=C \\
&\geq &C\left( 1+\left\vert \boldsymbol{\xi }\right\vert \right) \text{.}
\end{eqnarray*}%
Secondly, if $\left\vert \boldsymbol{\xi }\right\vert \geq 1$, then%
\begin{eqnarray*}
\upsilon _{i} &\geq &C\int_{\mathbb{R}^{3}}\left( \left\vert \left\vert 
\boldsymbol{\xi }\right\vert -\left\vert \boldsymbol{\xi }_{\ast
}\right\vert \right\vert \right) e^{-m\left\vert \boldsymbol{\xi }_{\ast
}\right\vert ^{2}/2}\,d\boldsymbol{\xi }_{\ast } \\
&\geq &C\int_{\left\vert \boldsymbol{\xi }_{\ast }\right\vert \leq
1/2}\left( \left\vert \boldsymbol{\xi }\right\vert -\left\vert \boldsymbol{%
\xi }_{\ast }\right\vert \right) e^{-m\left\vert \boldsymbol{\xi }_{\ast
}\right\vert ^{2}/2}\,d\boldsymbol{\xi }_{\ast } \\
&\geq &C\frac{1}{2}\left\vert \boldsymbol{\xi }\right\vert
e^{-m/8}\int_{\left\vert \boldsymbol{\xi }_{\ast }\right\vert \leq 1/2}\,d%
\boldsymbol{\xi }_{\ast }=C\left\vert \boldsymbol{\xi }\right\vert \\
&\geq &C\left( 1+\left\vert \boldsymbol{\xi }\right\vert \right) \text{.}
\end{eqnarray*}%
Hence, there is a positive constant $\upsilon _{-}>0$, such that $\upsilon
_{i}\geq \upsilon _{-}\left( 1+\left\vert \boldsymbol{\xi }\right\vert
\right) $ for all $i\in \left\{ 1,...,r\right\} $ and $\boldsymbol{\xi }\in 
\mathbb{R}^{3}$.

On the other hand,%
\begin{eqnarray*}
\upsilon _{i} &\leq &C\sum\limits_{j,k,l=1}^{r}\int_{\mathbb{R}^{3}}\sqrt{%
\left\vert \mathbf{g}\right\vert ^{2}-\frac{4}{m}\Delta I_{ij}^{kl}}%
\,e^{-m\left\vert \boldsymbol{\xi }_{\ast }\right\vert ^{2}/2}\mathbf{1}%
_{m\left\vert \mathbf{g}\right\vert >4\Delta I_{ij}^{kl}}\,d\boldsymbol{\xi }%
_{\ast } \\
&\leq &C\sum\limits_{j,k,l=1}^{r}\int_{\mathbb{R}^{3}}\left( 1+\left\vert 
\boldsymbol{\xi }\right\vert +\left\vert \boldsymbol{\xi }_{\ast
}\right\vert \right) e^{-m\left\vert \boldsymbol{\xi }_{\ast }\right\vert
^{2}/2}\,d\boldsymbol{\xi }_{\ast } \\
&\leq &C\left( \left\vert \boldsymbol{\xi }\right\vert \int_{\mathbb{R}%
^{3}}e^{-m\left\vert \boldsymbol{\xi }_{\ast }\right\vert ^{2}/2}d%
\boldsymbol{\xi }_{\ast }+\int_{\mathbb{R}^{3}}e^{-m\left\vert \boldsymbol{%
\xi }_{\ast }\right\vert ^{2}/2}\left( 1+\left\vert \boldsymbol{\xi }_{\ast
}\right\vert \right) \,d\boldsymbol{\xi }_{\ast }\right) \\
&\leq &C\left( 1+\left\vert \boldsymbol{\xi }\right\vert \right) \text{.}
\end{eqnarray*}%
Hence, there is a positive constant $\upsilon _{+}>0$, such that $\upsilon
_{i}\leq \upsilon _{+}\left( 1+\left\vert \boldsymbol{\xi }\right\vert
\right) $ for all $i\in \left\{ 1,...,r\right\} $ and $\boldsymbol{\xi }\in 
\mathbb{R}^{3}$.
\end{proof}

\subsection{Mixtures\label{S4.2}}

This section is devoted to the proofs of the compactness properties in
Theorem \ref{Thm3} and the bounds on the collision frequency in Theorem \ref%
{Thm4} of the linearized collision operator for (monatomic) multicomponent
mixtures.

\subsubsection{\label{PT3}Compactness}

This section concerns the proof of Theorem \ref{Thm3}. Note that in the
proof the kernels are rewritten in such a way that $\boldsymbol{\xi }_{\ast
} $ - and not $\boldsymbol{\xi }^{\prime }$ and $\boldsymbol{\xi }_{\ast
}^{\prime }$ - always will be argument of the distribution functions. As for
single species, either $\boldsymbol{\xi }_{\ast }$ is an argument in the
loss term (as $\boldsymbol{\xi }$) or in the gain term (opposite to $%
\boldsymbol{\xi }$) of the collision operator. However,\ in the latter case,
unlike for single species, for mixtures we have to differ between two
different cases; either $\boldsymbol{\xi }_{\ast }$ is associated to the
same species as $\boldsymbol{\xi }$, or not. The kernels of the terms from
the loss part of the collision operator will be shown\ to be Hilbert-Schmidt
in a quite direct way. The kernels of - some of - the terms - for which $%
\boldsymbol{\xi }_{\ast }$ is associated to the same species as $\boldsymbol{%
\xi }$ - from the gain parts of the collision operators will be shown to be
approximately Hilbert-Schmidt in the sense of Lemma \ref{LGD}. By applying
the following lemma, Lemma \ref{L3}, (for disparate masses) by Boudin et al
in \cite{BGPS-13}, it will be shown that the kernels of the remaining terms
- i.e. for which $\boldsymbol{\xi }_{\ast }$ is associated to the opposite
species to $\boldsymbol{\xi }$ - from the gain parts of the collision
operators, are Hilbert-Schmidt.

\begin{lemma}
\label{L3} \cite{BGPS-13} Assume that $m_{\alpha }\neq m_{\beta }$, 
\begin{equation}
\left\{ 
\begin{array}{l}
\boldsymbol{\xi }^{\prime }=\boldsymbol{\xi }_{\ast }-\dfrac{m_{\alpha }}{%
m_{\beta }}\left\vert \boldsymbol{\xi }-\boldsymbol{\xi }_{\ast }^{\prime
}\right\vert \boldsymbol{\eta } \\ 
\boldsymbol{\xi }_{\ast }^{\prime }=\boldsymbol{\xi }-\left\vert \boldsymbol{%
\xi }-\boldsymbol{\xi }_{\ast }^{\prime }\right\vert \boldsymbol{\eta }%
\end{array}%
\right. \text{, where }\boldsymbol{\eta }\in \mathbb{S}^{2}\text{,}
\label{vrel2}
\end{equation}%
and%
\begin{equation}
m_{\alpha }\left\vert \boldsymbol{\xi }\right\vert ^{2}+m_{\beta }\left\vert 
\boldsymbol{\xi }^{\prime }\right\vert ^{2}=m_{\alpha }\left\vert 
\boldsymbol{\xi }_{\ast }^{\prime }\right\vert ^{2}+m_{\beta }\left\vert 
\boldsymbol{\xi }_{\ast }\right\vert ^{2}\text{.}  \label{vrel3}
\end{equation}%
Then there exists a positive number $\rho $, $0<\rho <1$, such that%
\begin{equation*}
m_{\beta }\left\vert \boldsymbol{\xi }^{\prime }\right\vert ^{2}+m_{\alpha
}\left\vert \boldsymbol{\xi }_{\ast }^{\prime }\right\vert ^{2}\geq \rho
\left( m_{\alpha }\left\vert \boldsymbol{\xi }\right\vert ^{2}+m_{\beta
}\left\vert \boldsymbol{\xi }_{\ast }\right\vert ^{2}\right) \text{.}
\end{equation*}
\end{lemma}

An alternative - maybe more basic, in the sense that only very basic
calculations are used - to the proof of Lemma \ref{L3} in \cite{BGPS-13} is
accounted for in the appendix. The proof is constructive, in the way that an
explicit value of such a number $\rho $, namely%
\begin{equation*}
\rho =1-\frac{2}{1+\sqrt{1+\dfrac{\left( m_{\alpha }-m_{\beta }\right) ^{2}}{%
4m_{\alpha }m_{\beta }}}}\text{,}
\end{equation*}%
is produced in the proof.

Now we turn to the proof of Theorem \ref{Thm3}.

\begin{proof}
By first renaming $\left\{ \boldsymbol{\xi }_{\ast }\right\}
\leftrightarrows \left\{ \boldsymbol{\xi }^{\prime }\right\} $ and then $%
\left\{ \boldsymbol{\xi }_{\ast }\right\} \leftrightarrows \left\{ 
\boldsymbol{\xi }_{\ast }^{\prime }\right\} $ 
\begin{eqnarray*}
&&\int_{\left( \mathbb{R}^{3}\right) ^{3}}\left( M_{\beta \ast }M_{\alpha
}^{\prime }\right) ^{1/2}W_{\alpha \beta }(\boldsymbol{\xi },\boldsymbol{\xi 
}_{\ast }\left\vert \boldsymbol{\xi }^{\prime },\boldsymbol{\xi }_{\ast
}^{\prime }\right. )\,\,h_{\beta \ast }^{\prime }\,d\boldsymbol{\xi }_{\ast
}d\boldsymbol{\xi }^{\prime }d\boldsymbol{\xi }_{\ast }^{\prime } \\
&=&\int_{\left( \mathbb{R}^{3}\right) ^{3}}\left( M_{\beta }^{\prime
}M_{\alpha \ast }\right) ^{1/2}W_{\alpha \beta }(\boldsymbol{\xi },%
\boldsymbol{\xi }^{\prime }\left\vert \boldsymbol{\xi }_{\ast },\boldsymbol{%
\xi }_{\ast }^{\prime }\right. )\,h_{\beta \ast }^{\prime }\,d\boldsymbol{%
\xi }_{\ast }d\boldsymbol{\xi }^{\prime }d\boldsymbol{\xi }_{\ast }^{\prime }
\\
&=&\int_{\left( \mathbb{R}^{3}\right) ^{3}}\left( M_{\beta }^{\prime
}M_{\alpha \ast }^{\prime }\right) ^{1/2}W_{\alpha \beta }(\boldsymbol{\xi },%
\boldsymbol{\xi }^{\prime }\left\vert \boldsymbol{\xi }_{\ast }^{\prime },%
\boldsymbol{\xi }_{\ast }\right. )\,h_{\beta \ast }\,d\boldsymbol{\xi }%
_{\ast }d\boldsymbol{\xi }^{\prime }d\boldsymbol{\xi }_{\ast }^{\prime }%
\text{.}
\end{eqnarray*}%
Moreover, by renaming $\left\{ \boldsymbol{\xi }_{\ast }\right\}
\leftrightarrows \left\{ \boldsymbol{\xi }^{\prime }\right\} $, 
\begin{eqnarray*}
&&\int_{\left( \mathbb{R}^{3}\right) ^{3}}\left( M_{\beta \ast }M_{\beta
\ast }^{\prime }\right) ^{1/2}W_{\alpha \beta }(\boldsymbol{\xi },%
\boldsymbol{\xi }_{\ast }\left\vert \boldsymbol{\xi }^{\prime },\boldsymbol{%
\xi }_{\ast }^{\prime }\right. )\,h_{\alpha }^{\prime }\,d\boldsymbol{\xi }%
_{\ast }d\boldsymbol{\xi }^{\prime }d\boldsymbol{\xi }_{\ast }^{\prime } \\
&=&\int_{\left( \mathbb{R}^{3}\right) ^{3}}\left( M_{\beta }^{\prime
}M_{\beta \ast }^{\prime }\right) ^{1/2}W_{\alpha \beta }(\boldsymbol{\xi },%
\boldsymbol{\xi }^{\prime }\left\vert \boldsymbol{\xi }_{\ast },\boldsymbol{%
\xi }_{\ast }^{\prime }\right. )\,h_{\alpha \ast }\,d\boldsymbol{\xi }_{\ast
}d\boldsymbol{\xi }^{\prime }d\boldsymbol{\xi }_{\ast }^{\prime }\text{.}
\end{eqnarray*}%
It follows that%
\begin{eqnarray}
K_{\alpha }\left( h\right) &=&\sum\limits_{\beta =1}^{s}\int_{\mathbb{R}%
^{3}}k_{\alpha \beta }\left( \boldsymbol{\xi },\boldsymbol{\xi }_{\ast
}\right) \,h_{\ast }\,d\boldsymbol{\xi }_{\ast }\text{, where }  \notag \\
k_{\alpha \beta } &=&k_{\alpha \beta }^{\left( \alpha \right) }\mathbf{1}%
_{\alpha }+k_{\alpha \beta }^{\left( \beta \right) }\mathbf{1}_{\beta
}=k_{\alpha \beta }^{\left( \alpha \right) }\mathbf{1}_{\alpha }+\left(
k_{\alpha \beta 2}^{\left( \beta \right) }-k_{\alpha \beta 1}^{\left( \beta
\right) }\right) \mathbf{1}_{\beta }\text{; \ }  \notag \\
k_{\alpha \beta }h_{\ast } &=&k_{\alpha \beta }^{\left( \alpha \right)
}h_{\alpha \ast }+k_{\alpha \beta }^{\left( \beta \right) }h_{\beta \ast
}=k_{\alpha \beta }^{\left( \alpha \right) }h_{\alpha \ast }-k_{\alpha \beta
1}^{\left( \beta \right) }h_{\beta \ast }+k_{\alpha \beta 2}^{\left( \beta
\right) }h_{\beta \ast }\text{, with}  \notag \\
k_{\alpha \beta }^{\left( \alpha \right) }(\boldsymbol{\xi },\boldsymbol{\xi 
}_{\ast }) &=&\int_{\left( \mathbb{R}^{3}\right) ^{2}}\left( M_{\beta
}^{\prime }M_{\beta \ast }^{\prime }\right) ^{1/2}W_{\alpha \beta }(%
\boldsymbol{\xi },\boldsymbol{\xi }^{\prime }\left\vert \boldsymbol{\xi }%
_{\ast },\boldsymbol{\xi }_{\ast }^{\prime }\right. )\,d\boldsymbol{\xi }%
^{\prime }d\boldsymbol{\xi }_{\ast }^{\prime }\text{,}  \notag \\
k_{\alpha \beta 1}^{\left( \beta \right) }(\boldsymbol{\xi },\boldsymbol{\xi 
}_{\ast }) &=&\int_{\left( \mathbb{R}^{3}\right) ^{2}}\left( M_{\alpha
}^{\prime }M_{\beta \ast }^{\prime }\right) ^{1/2}W_{\alpha \beta }(%
\boldsymbol{\xi },\boldsymbol{\xi }_{\ast }\left\vert \boldsymbol{\xi }%
^{\prime },\boldsymbol{\xi }_{\ast }^{\prime }\right. )\,d\boldsymbol{\xi }%
^{\prime }d\boldsymbol{\xi }_{\ast }^{\prime }\text{, and}  \notag \\
k_{\alpha \beta 2}^{\left( \beta \right) }(\boldsymbol{\xi },\boldsymbol{\xi 
}_{\ast }) &=&\int_{\left( \mathbb{R}^{3}\right) ^{2}}\left( M_{\alpha \ast
}^{\prime }M_{\beta }^{\prime }\right) ^{1/2}W_{\alpha \beta }\left( 
\boldsymbol{\xi },\boldsymbol{\xi }^{\prime }\left\vert \boldsymbol{\xi }%
_{\ast }^{\prime },\boldsymbol{\xi }_{\ast }\right. \right) \,d\boldsymbol{%
\xi }^{\prime }d\boldsymbol{\xi }_{\ast }^{\prime }\text{.}  \label{k2}
\end{eqnarray}

Note that, by applying the second relation in $\left( \ref{rel3}\right) $
and renaming $\left\{ \boldsymbol{\xi }^{\prime }\right\} \leftrightarrows
\left\{ \boldsymbol{\xi }_{\ast }^{\prime }\right\} $,%
\begin{eqnarray}
k_{\alpha \beta }^{\left( \alpha \right) }(\boldsymbol{\xi },\boldsymbol{\xi 
}_{\ast }) &=&\int_{\left( \mathbb{R}^{3}\right) ^{2}}\left( M_{\beta
}^{\prime }M_{\beta \ast }^{\prime }\right) ^{1/2}W_{\alpha \beta }(%
\boldsymbol{\xi }_{\ast },\boldsymbol{\xi }_{\ast }^{\prime }\left\vert 
\boldsymbol{\xi },\boldsymbol{\xi }^{\prime }\right. )\,d\boldsymbol{\xi }%
^{\prime }d\boldsymbol{\xi }_{\ast }^{\prime }  \notag \\
&=&\int_{\left( \mathbb{R}^{3}\right) ^{2}}\left( M_{\beta }^{\prime
}M_{\beta \ast }^{\prime }\right) ^{1/2}W_{\alpha \beta }(\boldsymbol{\xi }%
_{\ast },\boldsymbol{\xi }^{\prime }\left\vert \boldsymbol{\xi },\boldsymbol{%
\xi }_{\ast }^{\prime }\right. )\,d\boldsymbol{\xi }^{\prime }d\boldsymbol{%
\xi }_{\ast }^{\prime }  \notag \\
&=&k_{\alpha \beta }^{\left( \alpha \right) }(\boldsymbol{\xi }_{\ast },%
\boldsymbol{\xi })\text{.}  \label{sa3}
\end{eqnarray}%
Moreover,%
\begin{equation}
k_{\alpha \beta }^{\left( \beta \right) }(\boldsymbol{\xi },\boldsymbol{\xi }%
_{\ast })=k_{\beta \alpha 1}^{\left( \alpha \right) }(\boldsymbol{\xi }%
_{\ast },\boldsymbol{\xi })-k_{\beta \alpha 2}^{\left( \alpha \right) }(%
\boldsymbol{\xi }_{\ast },\boldsymbol{\xi })=k_{\beta \alpha }^{\left(
\alpha \right) }(\boldsymbol{\xi }_{\ast },\boldsymbol{\xi }),  \label{sa4}
\end{equation}%
since, by applying the first relation in $\left( \ref{rel3}\right) $ and
renaming $\left\{ \boldsymbol{\xi }^{\prime }\right\} \leftrightarrows
\left\{ \boldsymbol{\xi }_{\ast }^{\prime }\right\} $, 
\begin{eqnarray*}
k_{\alpha \beta 1}^{\left( \beta \right) }(\boldsymbol{\xi },\boldsymbol{\xi 
}_{\ast }) &=&\int_{\left( \mathbb{R}^{3}\right) ^{2}}\left( M_{\alpha
}^{\prime }M_{\beta \ast }^{\prime }\right) ^{1/2}W_{\beta \alpha }(%
\boldsymbol{\xi }_{\ast },\boldsymbol{\xi }\left\vert \boldsymbol{\xi }%
_{\ast }^{\prime },\boldsymbol{\xi }^{\prime }\right. )\,d\boldsymbol{\xi }%
^{\prime }d\boldsymbol{\xi }_{\ast }^{\prime } \\
&=&\int_{\left( \mathbb{R}^{3}\right) ^{2}}\left( M_{\alpha \ast }^{\prime
}M_{\beta }^{\prime }\right) ^{1/2}W_{\beta \alpha }(\boldsymbol{\xi }_{\ast
},\boldsymbol{\xi }\left\vert \boldsymbol{\xi }^{\prime },\boldsymbol{\xi }%
_{\ast }^{\prime }\right. )\,d\boldsymbol{\xi }^{\prime }d\boldsymbol{\xi }%
_{\ast }^{\prime } \\
&=&k_{\beta \alpha 1}^{\left( \alpha \right) }(\boldsymbol{\xi }_{\ast },%
\boldsymbol{\xi })\text{,}
\end{eqnarray*}%
while, by applying the two first relations in $\left( \ref{rel3}\right) $
and renaming $\left\{ \boldsymbol{\xi }^{\prime }\right\} \leftrightarrows
\left\{ \boldsymbol{\xi }_{\ast }^{\prime }\right\} $,%
\begin{eqnarray*}
k_{\alpha \beta 2}^{\left( \beta \right) }(\boldsymbol{\xi },\boldsymbol{\xi 
}_{\ast }) &=&\int_{\left( \mathbb{R}^{3}\right) ^{2}}\left( M_{\alpha \ast
}^{\prime }M_{\beta }^{\prime }\right) ^{1/2}W_{\beta \alpha }\left( 
\boldsymbol{\xi }^{\prime },\boldsymbol{\xi }\left\vert \boldsymbol{\xi }%
_{\ast },\boldsymbol{\xi }_{\ast }^{\prime }\right. \right) \,d\boldsymbol{%
\xi }^{\prime }d\boldsymbol{\xi }_{\ast }^{\prime } \\
&=&\int_{\left( \mathbb{R}^{3}\right) ^{2}}\left( M_{\alpha \ast }^{\prime
}M_{\beta }^{\prime }\right) ^{1/2}W_{\beta \alpha }\left( \boldsymbol{\xi }%
_{\ast },\boldsymbol{\xi }_{\ast }^{\prime }\left\vert \boldsymbol{\xi }%
^{\prime },\boldsymbol{\xi }\right. \right) \,d\boldsymbol{\xi }^{\prime }d%
\boldsymbol{\xi }_{\ast }^{\prime } \\
&=&\int_{\left( \mathbb{R}^{3}\right) ^{2}}\left( M_{\alpha }^{\prime
}M_{\beta \ast }^{\prime }\right) ^{1/2}W_{\beta \alpha }\left( \boldsymbol{%
\xi }_{\ast },\boldsymbol{\xi }^{\prime }\left\vert \boldsymbol{\xi }_{\ast
}^{\prime },\boldsymbol{\xi }\right. \right) \,d\boldsymbol{\xi }^{\prime }d%
\boldsymbol{\xi }_{\ast }^{\prime } \\
&=&k_{\beta \alpha 2}^{\left( \alpha \right) }(\boldsymbol{\xi }_{\ast },%
\boldsymbol{\xi })\text{.}
\end{eqnarray*}

\begin{figure}[h]
\centering
\includegraphics[width=0.6\textwidth]{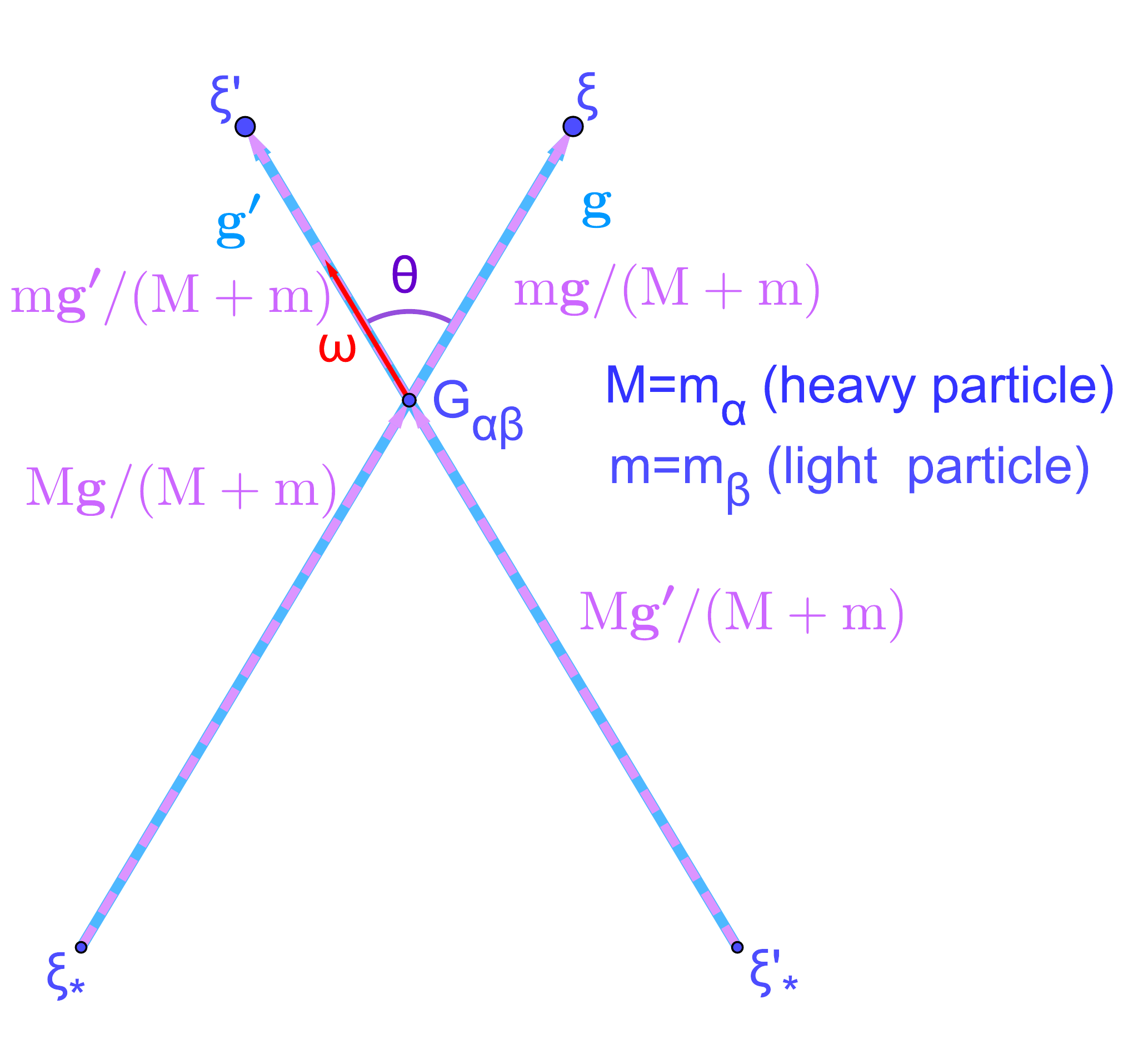}
\caption{Typical collision of $K_{\protect\alpha \protect\beta }^{(1)}$.
Classical representation of a collision between particles of different
species.}
\label{fig3}
\end{figure}

We now continue by proving the compactness for the three different types of
collision kernel separately. \ Note that, by applying the last relation in $%
\left( \ref{rel3}\right) $, $k_{\alpha \beta 2}^{\left( \beta \right) }(%
\boldsymbol{\xi },\boldsymbol{\xi }_{\ast })=k_{\alpha \beta }^{\left(
\alpha \right) }(\boldsymbol{\xi },\boldsymbol{\xi }_{\ast })$ if $\alpha
=\beta $, and we will remain with only two cases - the first two below. Even
if $m_{\alpha }=m_{\beta }$, the kernels $k_{\alpha \beta }^{\left( \alpha
\right) }(\boldsymbol{\xi },\boldsymbol{\xi }_{\ast })$ and $k_{\alpha \beta
2}^{\left( \beta \right) }(\boldsymbol{\xi },\boldsymbol{\xi }_{\ast })$ are
structurally equal, why we (in principle) remains with (first) two cases
(the second one twice).

\textbf{I. Compactness of }$K_{\alpha \beta }^{(1)}=\int_{\mathbb{R}%
^{3}}k_{\alpha \beta 1}^{\left( \beta \right) }(\boldsymbol{\xi },%
\boldsymbol{\xi }_{\ast })\,h_{\beta \ast }\,d\boldsymbol{\xi }_{\ast }$.

By the change of variables $\left\{ \boldsymbol{\xi }^{\prime },\boldsymbol{%
\xi }_{\ast }^{\prime }\right\} \rightarrow \left\{ \mathbf{g}^{\prime }=%
\boldsymbol{\xi }^{\prime }-\boldsymbol{\xi }_{\ast }^{\prime },\mathbf{G}%
_{\alpha \beta }^{\prime }=\dfrac{m_{\alpha }\boldsymbol{\xi }^{\prime
}+m_{\beta }\boldsymbol{\xi }_{\ast }^{\prime }}{m_{\alpha }+m_{\beta }}%
\right\} $, cf. Figure $\ref{fig3}$, noting that $\left( \ref{df2}\right) $,
and using relation $\left( \ref{M2}\right) $, expression $\left( \ref{k2}%
\right) $ of $k_{\alpha \beta 1}^{\left( \beta \right) }$ may be transformed
to%
\begin{eqnarray*}
k_{\alpha \beta 1}^{\left( \beta \right) }(\boldsymbol{\xi },\boldsymbol{\xi 
}_{\ast }) &=&\left( M_{\alpha }M_{\beta \ast }\right) ^{1/2}\int_{\left( 
\mathbb{R}^{3}\right) ^{2}}\frac{\sigma _{\alpha \beta }}{\left\vert \mathbf{%
g}\right\vert }\delta _{3}\left( \mathbf{G}_{\alpha \beta }-\mathbf{G}%
_{\alpha \beta }^{\prime }\right) \delta _{1}\left( \left\vert \mathbf{g}%
\right\vert -\left\vert \mathbf{g}^{\prime }\right\vert \right) \,d\mathbf{G}%
_{\alpha \beta }^{\prime }d\mathbf{g}^{\prime } \\
&=&\left( M_{\alpha }M_{\beta \ast }\right) ^{1/2}\left\vert \mathbf{g}%
\right\vert \int_{\mathbb{S}^{2}}\sigma _{\alpha \beta }\left( \left\vert 
\mathbf{g}\right\vert ,\cos \theta \right) \,d\boldsymbol{\omega }\text{,
with }\cos \theta =\boldsymbol{\omega }\cdot \frac{\mathbf{g}}{\left\vert 
\mathbf{g}\right\vert }\text{.}
\end{eqnarray*}%
By assumption $\left( \ref{est2}\right) $ and\textbf{\ } 
\begin{equation*}
m_{\alpha }\left\vert \boldsymbol{\xi }\right\vert ^{2}+m_{\beta }\left\vert 
\boldsymbol{\xi }_{\ast }\right\vert ^{2}=\left( m_{\alpha }+m_{\beta
}\right) \left\vert \mathbf{G}_{\alpha \beta }\right\vert ^{2}+\mu _{\alpha
\beta }\left\vert \mathbf{g}\right\vert ^{2}\text{, with }\mu _{\alpha \beta
}=\frac{m_{\alpha }m_{\beta }}{m_{\alpha }+m_{\beta }}\text{,}
\end{equation*}%
the bound%
\begin{eqnarray}
\left( k_{\alpha \beta 1}^{\left( \beta \right) }(\boldsymbol{\xi },%
\boldsymbol{\xi }_{\ast })\right) ^{2} &\leq &CM_{\alpha }M_{\beta \ast
}\left( 1+\frac{1}{\left\vert \mathbf{g}\right\vert ^{2-\gamma }}\right)
^{2}\left\vert \mathbf{g}\right\vert ^{2}\int_{\mathbb{S}^{2}}\,d\boldsymbol{%
\omega }  \notag \\
&=&Ce^{-\left( m_{\alpha }+m_{\beta }\right) \left\vert \mathbf{G}_{\alpha
\beta }\right\vert ^{2}/2-\mu _{\alpha \beta }\left\vert \mathbf{g}%
\right\vert ^{2}/2}\left( \left\vert \mathbf{g}\right\vert +\frac{1}{%
\left\vert \mathbf{g}\right\vert ^{1-\gamma }}\right) ^{2}  \label{b3}
\end{eqnarray}%
may be obtained. Then, by applying the bound $\left( \ref{b3}\right) $ and
first changing variables of integration $\left\{ \boldsymbol{\xi },%
\boldsymbol{\xi }_{\ast }\right\} \rightarrow \left\{ \mathbf{g},\mathbf{G}%
_{\alpha \beta }\right\} $, with unitary Jacobian, and then to spherical
coordinates,%
\begin{eqnarray*}
&&\int_{\left( \mathbb{R}^{3}\right) ^{2}}\left( k_{\alpha \beta 1}^{\left(
\beta \right) }(\boldsymbol{\xi },\boldsymbol{\xi }_{\ast })\right) ^{2}d%
\boldsymbol{\xi }d\boldsymbol{\xi }_{\ast } \\
&\leq &C\int_{\left( \mathbb{R}^{3}\right) ^{2}}e^{-\left( m_{\alpha
}+m_{\beta }\right) \left\vert \mathbf{G}_{\alpha \beta }\right\vert
^{2}/2-\mu _{\alpha \beta }\left\vert \mathbf{g}\right\vert ^{2}/2}\left(
\left\vert \mathbf{g}\right\vert +\frac{1}{\left\vert \mathbf{g}\right\vert
^{1-\gamma }}\right) ^{2}d\mathbf{g}\boldsymbol{\,}d\mathbf{G}_{\alpha \beta
} \\
&\leq &C\int_{0}^{\infty }R^{2}e^{-R^{2}}dR\int_{0}^{\infty }\left(
r^{2}+r^{\gamma }\right) ^{2}e^{-r^{2}/4}dr \\
&\leq &C\int_{0}^{\infty }\left( 1+r^{4}\right) e^{-r^{2}/4}dr=C\text{.}
\end{eqnarray*}%
Hence,%
\begin{equation*}
K_{\alpha \beta }^{(1)}=\int_{\mathbb{R}^{3}}k_{\alpha \beta 1}^{\left(
\beta \right) }(\boldsymbol{\xi },\boldsymbol{\xi }_{\ast })\,h_{\beta \ast
}\,d\boldsymbol{\xi }_{\ast }
\end{equation*}%
are Hilbert-Schmidt integral operators and as such continuous and compact on 
$L^{2}\left( d\boldsymbol{\xi }\right) $, see e.g. Theorem 7.83 in \cite%
{RenardyRogers}, for all $\left\{ \alpha ,\beta \right\} \subseteq \left\{
1,...,s\right\} $.

\begin{figure}[h]
\centering
\includegraphics[width=0.6\textwidth]{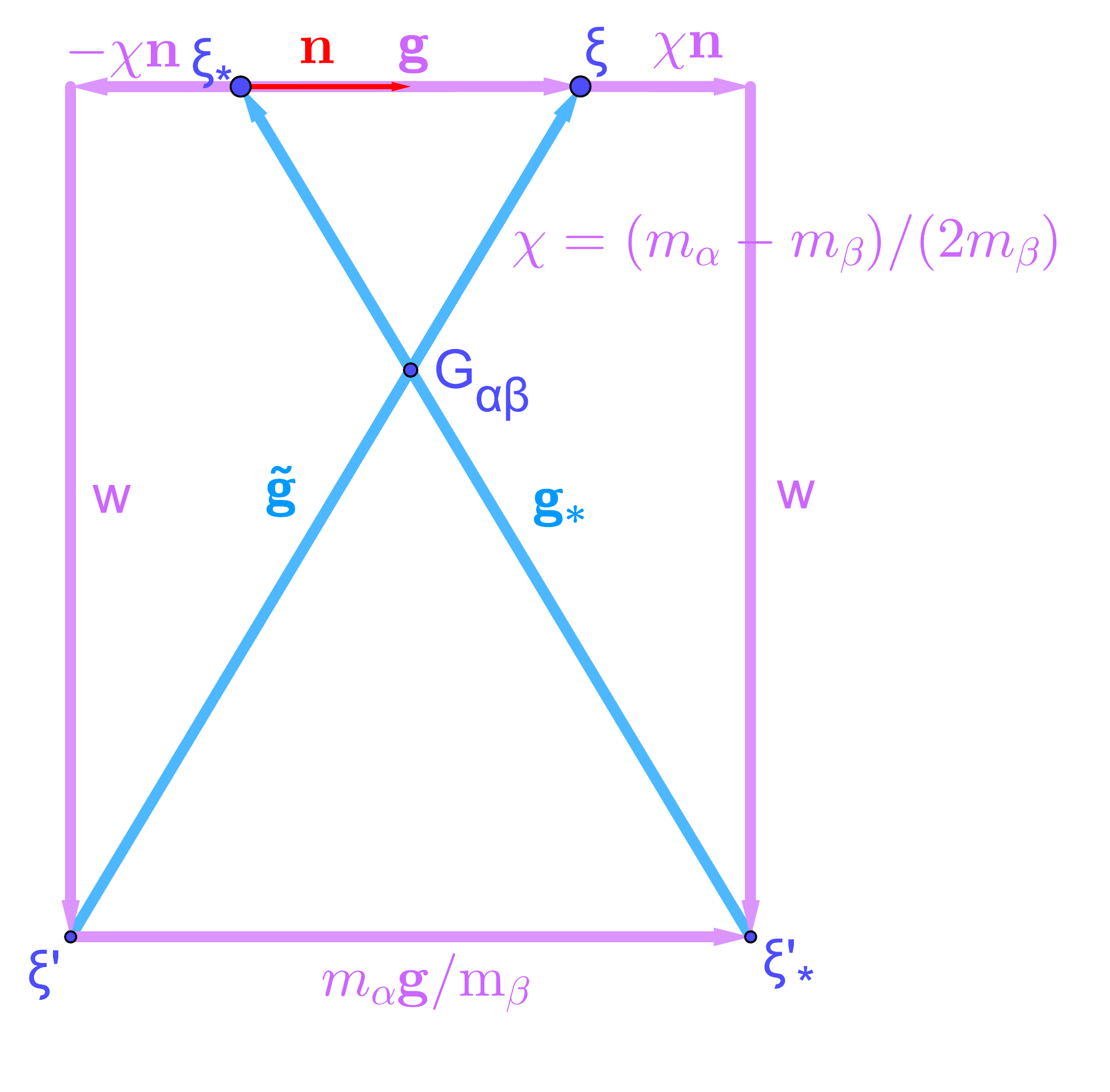}
\caption{Typical collision of $K_{\protect\alpha \protect\beta }^{(3)}$.}
\label{fig4}
\end{figure}

\textbf{II. Compactness of }$K_{\alpha \beta }^{(3)}=\int_{\mathbb{R}%
^{3}}k_{\alpha \beta }^{\left( \alpha \right) }(\boldsymbol{\xi },%
\boldsymbol{\xi }_{\ast })\,h_{\alpha \ast }\,d\boldsymbol{\xi }_{\ast }$.

Noting that, cf. Figure \ref{fig4},%
\begin{eqnarray*}
W_{\alpha \beta }(\boldsymbol{\xi },\boldsymbol{\xi }^{\prime }\left\vert 
\boldsymbol{\xi }_{\ast },\boldsymbol{\xi }_{\ast }^{\prime }\right. )
&=&2\left( m_{\alpha }+m_{\beta }\right) ^{2}m_{\alpha }m_{\beta }\sigma
_{\alpha \beta }\delta _{3}\left( m_{\alpha }\mathbf{g}+m_{\beta }\mathbf{g}%
^{\prime }\right) \\
&&\times \delta _{1}\left( 2\left\vert \mathbf{g}\right\vert m_{\alpha
}\left( \chi -\frac{m_{\alpha }-m_{\beta }}{2m_{\beta }}\left\vert \mathbf{g}%
\right\vert \right) \right) \\
&=&\frac{\left( m_{\alpha }+m_{\beta }\right) ^{2}}{\left\vert \mathbf{g}%
\right\vert m_{\beta }^{2}}\sigma _{\alpha \beta }\,\delta _{3}\!\left( 
\frac{m_{\alpha }}{m_{\beta }}\mathbf{g}+\mathbf{g}^{\prime }\right) \delta
_{1}\!\left( \chi -\frac{m_{\alpha }-m_{\beta }}{2m_{\beta }}\left\vert 
\mathbf{g}\right\vert \!\right) \text{,}
\end{eqnarray*}%
where $\mathbf{g}=\boldsymbol{\xi }-\boldsymbol{\xi }_{\ast }$, $\mathbf{g}%
^{\prime }=\boldsymbol{\xi }^{\prime }-\boldsymbol{\xi }_{\ast }^{\prime }$, 
$\chi =\left( \boldsymbol{\xi }_{\ast }^{\prime }-\boldsymbol{\xi }\right)
\cdot \mathbf{n}$, and $\mathbf{n}=\dfrac{\boldsymbol{\xi }-\boldsymbol{\xi }%
_{\ast }}{\left\vert \boldsymbol{\xi }-\boldsymbol{\xi }_{\ast }\right\vert }
$, by a change of variables $\left\{ \boldsymbol{\xi }^{\prime },\boldsymbol{%
\xi }_{\ast }^{\prime }\right\} \rightarrow \left\{ \mathbf{g}^{\prime }=%
\boldsymbol{\xi }^{\prime }-\boldsymbol{\xi }_{\ast }^{\prime },~\widetilde{%
\mathbf{h}}=\boldsymbol{\xi }_{\ast }^{\prime }-\boldsymbol{\xi }\right\} $,
with%
\begin{equation*}
d\boldsymbol{\xi }^{\prime }d\boldsymbol{\xi }_{\ast }^{\prime }=d\mathbf{g}%
^{\prime }d\widetilde{\mathbf{h}}=d\mathbf{g}^{\prime }d\chi d\mathbf{w}%
\text{, with }\mathbf{w}=\boldsymbol{\xi }_{\ast }^{\prime }-\boldsymbol{\xi 
}-\chi \mathbf{n}\text{,}
\end{equation*}%
the expression $\left( \ref{k2}\right) $ of $k_{\alpha \beta }^{\left(
\alpha \right) }$ may be rewritten in the following way%
\begin{eqnarray*}
k_{\alpha \beta }^{\left( \alpha \right) }(\boldsymbol{\xi },\boldsymbol{\xi 
}_{\ast }) &=&\int_{\left( \mathbb{R}^{3}\right) ^{2}}\left( M_{\beta
}^{\prime }M_{\beta \ast }^{\prime }\right) ^{1/2}W_{\alpha \beta }(%
\boldsymbol{\xi },\boldsymbol{\xi }^{\prime }\left\vert \boldsymbol{\xi }%
_{\ast },\boldsymbol{\xi }_{\ast }^{\prime }\right. )\,d\mathbf{g}^{\prime }d%
\widetilde{\mathbf{h}} \\
&=&\int_{\left( \mathbb{R}^{3}\right) ^{\perp _{\mathbf{n}}}}\frac{\left(
m_{\alpha }+m_{\beta }\right) ^{2}}{m_{\beta }^{2}}\frac{\left( M_{\beta
}^{\prime }M_{\beta \ast }^{\prime }\right) ^{1/2}}{\left\vert \mathbf{g}%
\right\vert }\sigma _{\alpha \beta }\,\left( \left\vert \widetilde{\mathbf{g}%
}\right\vert ,\frac{\widetilde{\mathbf{g}}\cdot \mathbf{g}_{\ast }}{%
\left\vert \widetilde{\mathbf{g}}\right\vert \left\vert \mathbf{g}_{\ast
}\right\vert }\right) d\mathbf{w} \\
\text{with }\widetilde{\mathbf{g}} &=&\boldsymbol{\xi }-\boldsymbol{\xi }%
^{\prime }\text{, and }\mathbf{g}_{\ast }=\boldsymbol{\xi }_{\ast }-%
\boldsymbol{\xi }_{\ast }^{\prime }\text{.}
\end{eqnarray*}%
Here, see Figure $\ref{fig4}$,%
\begin{equation*}
\left\{ 
\begin{array}{l}
\boldsymbol{\xi }^{\prime }=\boldsymbol{\xi }_{\ast }+\mathbf{w}-\chi 
\mathbf{n} \\ 
\boldsymbol{\xi }_{\ast }^{\prime }=\boldsymbol{\xi }+\mathbf{w}+\chi 
\mathbf{n}%
\end{array}%
\right. \text{, with }\chi =\frac{m_{\alpha }-m_{\beta }}{2m_{\beta }}%
\left\vert \boldsymbol{\xi }-\boldsymbol{\xi }_{\ast }\right\vert \text{,}
\end{equation*}%
implying that, reminding the notations $\left( \ref{o1}\right) $, 
\begin{eqnarray*}
\frac{\left\vert \boldsymbol{\xi }^{\prime }\right\vert ^{2}}{2}+\frac{%
\left\vert \boldsymbol{\xi }_{\ast }^{\prime }\right\vert ^{2}}{2}
&=&\left\vert \frac{\boldsymbol{\xi +\xi }_{\ast }}{2}+\mathbf{w}\right\vert
^{2}+\frac{\left( \left\vert \boldsymbol{\xi }-\boldsymbol{\xi }_{\ast
}\right\vert +2\chi \right) ^{2}}{4} \\
&=&\left\vert \frac{\left( \boldsymbol{\xi +\xi }_{\ast }\right) _{\perp _{%
\boldsymbol{n}}}}{2}+\mathbf{w}\right\vert ^{2}+\left( \frac{\left( 
\boldsymbol{\xi +\xi }_{\ast }\right) _{\mathbf{n}}}{2}\right) ^{2}+\frac{%
m_{\alpha }^{2}}{4m_{\beta }^{2}}\left\vert \boldsymbol{\xi }-\boldsymbol{%
\xi }_{\ast }\right\vert ^{2} \\
&=&\left\vert \frac{\left( \boldsymbol{\xi +\xi }_{\ast }\right) _{\perp _{%
\boldsymbol{n}}}}{2}+\mathbf{w}\right\vert ^{2}+\frac{\left( \left\vert 
\boldsymbol{\xi }_{\ast }\right\vert ^{2}-\left\vert \boldsymbol{\xi }%
\right\vert ^{2}\right) ^{2}}{4\left\vert \boldsymbol{\xi }-\boldsymbol{\xi }%
_{\ast }\right\vert ^{2}}+\frac{m_{\alpha }^{2}}{4m_{\beta }^{2}}\left\vert 
\boldsymbol{\xi }-\boldsymbol{\xi }_{\ast }\right\vert ^{2}.
\end{eqnarray*}%
Hence, by assumption $\left( \ref{est2}\right) $%
\begin{eqnarray}
&&\left( k_{\alpha \beta }^{\left( \alpha \right) }(\boldsymbol{\xi },%
\boldsymbol{\xi }_{\ast })\right) ^{2}  \notag \\
&\leq &\frac{C}{\left\vert \mathbf{g}\right\vert ^{2}}\exp \left( -\frac{%
m_{\beta }}{4}\frac{\left( \left\vert \boldsymbol{\xi }_{\ast }\right\vert
^{2}-\left\vert \boldsymbol{\xi }\right\vert ^{2}\right) ^{2}}{\left\vert 
\boldsymbol{\xi }-\boldsymbol{\xi }_{\ast }\right\vert ^{2}}-\frac{m_{\alpha
}^{2}}{4m_{\beta }}\left\vert \boldsymbol{\xi }-\boldsymbol{\xi }_{\ast
}\right\vert ^{2}\right)  \notag \\
&&\times \left( \int_{\left( \mathbb{R}^{3}\right) ^{\perp _{\mathbf{n}%
}}}\left( 1+\frac{1}{\left\vert \widetilde{\mathbf{g}}\right\vert ^{2-\gamma
}}\right) \exp \left( -\frac{m_{\beta }}{2}\left\vert \frac{\left( 
\boldsymbol{\xi +\xi }_{\ast }\right) _{\perp _{\boldsymbol{n}}}}{2}+\mathbf{%
w}\right\vert ^{2}\right) \,d\mathbf{w}\right) ^{2}  \notag \\
&\leq &\frac{C}{\left\vert \mathbf{g}\right\vert ^{2}}\exp \left( -\frac{%
m_{\beta }}{4}\frac{\left( \left\vert \boldsymbol{\xi }_{\ast }\right\vert
^{2}-\left\vert \boldsymbol{\xi }\right\vert ^{2}\right) ^{2}}{\left\vert 
\boldsymbol{\xi }-\boldsymbol{\xi }_{\ast }\right\vert ^{2}}-\frac{m_{\alpha
}^{2}}{4m_{\beta }}\left\vert \boldsymbol{\xi }-\boldsymbol{\xi }_{\ast
}\right\vert ^{2}\right)  \notag \\
&=&\frac{C}{\left\vert \mathbf{g}\right\vert ^{2}}\exp \left( -\frac{%
m_{\beta }}{4}\frac{\left( \left\vert \mathbf{g}\right\vert ^{2}+2\mathbf{g}%
\cdot \boldsymbol{\xi }\right) ^{2}}{\left\vert \mathbf{g}\right\vert ^{2}}-%
\frac{m_{\alpha }^{2}}{4m_{\beta }}\left\vert \mathbf{g}\right\vert
^{2}\right)  \notag \\
&=&\frac{C}{\left\vert \mathbf{g}\right\vert ^{2}}\exp \left( -m_{\beta
}\left( \dfrac{\left\vert \mathbf{g}\right\vert }{2}+\left\vert \boldsymbol{%
\xi }\right\vert \cos \varphi \right) ^{2}-\frac{m_{\alpha }^{2}}{4m_{\beta }%
}\left\vert \mathbf{g}\right\vert ^{2}\right) \text{, }\cos \varphi =\frac{%
\mathbf{g}\cdot \boldsymbol{\xi }}{\left\vert \mathbf{g}\right\vert
\left\vert \boldsymbol{\xi }\right\vert }\text{,}  \label{b4}
\end{eqnarray}%
since,%
\begin{eqnarray*}
&&\int_{\left( \mathbb{R}^{3}\right) ^{\perp _{\mathbf{n}}}}\left( 1+\frac{1%
}{\left\vert \widetilde{\mathbf{g}}\right\vert ^{2-\gamma }}\right) \exp
\left( -\frac{m_{\beta }}{2}\left\vert \frac{\left( \boldsymbol{\xi +\xi }%
_{\ast }\right) _{\perp _{\boldsymbol{n}}}}{2}+\mathbf{w}\right\vert
^{2}\right) \,d\mathbf{w} \\
&\leq &\int_{\left\vert \mathbf{w}\right\vert \leq 1}1+\left\vert \mathbf{w}%
\right\vert ^{\gamma -2}\,d\mathbf{w}+2\int_{\left\vert \mathbf{w}%
\right\vert \geq 1}\exp \left( -\frac{m_{\beta }}{2}\left\vert \frac{\left( 
\boldsymbol{\xi +\xi }_{\ast }\right) _{\perp _{\boldsymbol{n}}}}{2}+\mathbf{%
w}\right\vert ^{2}\right) \,d\mathbf{w} \\
&\leq &\int_{\left\vert \mathbf{w}\right\vert \leq 1}1+\left\vert \mathbf{w}%
\right\vert ^{\gamma -2}\,d\mathbf{w}+2\int_{\left( \mathbb{R}^{3}\right)
^{\perp _{\mathbf{n}}}}e^{-m_{\beta }\left\vert \widetilde{\mathbf{w}}%
\right\vert ^{2}/2}\,d\widetilde{\mathbf{w}} \\
&=&2\pi \left( \int_{0}^{1}R+R^{\gamma -1}\,dr+2\int_{0}^{\infty
}Re^{-m_{\beta }R^{2}/2}\,dr\right) =C\text{.}
\end{eqnarray*}

To show that $k_{\alpha \beta }^{\left( \alpha \right) }(\boldsymbol{\xi },%
\boldsymbol{\xi }_{\ast })\mathbf{1}_{\mathfrak{h}_{N}}\in L^{2}\left( d%
\boldsymbol{\xi \,}d\boldsymbol{\xi }_{\ast }\right) $ for any (non-zero)
natural number $N$, separate the integration domain\ of the integral of $%
\left( k_{\alpha \beta }^{\left( \alpha \right) }(\boldsymbol{\xi },%
\boldsymbol{\xi }_{\ast })\right) ^{2}$ over $\left( \mathbb{R}^{3}\right)
^{2}$ in two separate domains $\left\{ \left( \mathbb{R}^{3}\right) ^{2}%
\text{; }\left\vert \mathbf{g}\right\vert \geq \left\vert \boldsymbol{\xi }%
\right\vert \right\} $ and $\left\{ \left( \mathbb{R}^{3}\right) ^{2}\text{; 
}\left\vert \mathbf{g}\right\vert \leq \left\vert \boldsymbol{\xi }%
\right\vert \right\} $. The integral of $\left( k_{\alpha \beta }^{\left(
\alpha \right) }\right) ^{2}$ over the domain $\left\{ \left( \mathbb{R}%
^{3}\right) ^{2}\text{; }\left\vert \mathbf{g}\right\vert \geq \left\vert 
\boldsymbol{\xi }\right\vert \right\} $ will be bounded, since

\begin{eqnarray*}
&&\int_{\left\vert \mathbf{g}\right\vert \geq \left\vert \boldsymbol{\xi }%
\right\vert }\left( k_{\alpha \beta }^{\left( \alpha \right) }(\boldsymbol{%
\xi },\boldsymbol{\xi }_{\ast })\right) ^{2}\,d\boldsymbol{\xi \,}d%
\boldsymbol{\xi }_{\ast } \\
&\leq &\int_{\left\vert \mathbf{g}\right\vert \geq \left\vert \boldsymbol{%
\xi }\right\vert }\frac{C}{\left\vert \mathbf{g}\right\vert ^{2}}%
e^{-m_{\alpha }^{2}\left\vert \mathbf{g}\right\vert ^{2}/\left( 4m_{\beta
}\right) }d\mathbf{g}\boldsymbol{\,}d\boldsymbol{\xi } \\
&=&C\int_{0}^{\infty }\int_{\eta }^{\infty }e^{-m_{\alpha }^{2}R^{2}/\left(
4m_{\beta }\right) }\eta ^{2}dR\boldsymbol{\,}d\eta  \\
&\leq &C\int_{0}^{\infty }e^{-m_{\alpha }^{2}R^{2}/\left( 8m_{\beta }\right)
}dR\int_{0}^{\infty }e^{-m_{\alpha }^{2}\eta ^{2}/\left( 4m_{\beta }\right)
}\eta ^{2}d\eta =C\text{.}
\end{eqnarray*}%
As for the second domain, consider the truncated domains $\left\{ \!\left( 
\mathbb{R}^{3}\right) ^{2}\!\!\text{;}\left\vert \mathbf{g}\right\vert \leq
\left\vert \boldsymbol{\xi }\right\vert \leq N\!\right\} $ for (non-zero)
natural numbers $N$. Then%
\begin{eqnarray*}
&&\int_{\left\vert \mathbf{g}\right\vert \leq \left\vert \boldsymbol{\xi }%
\right\vert \leq N}\left( k_{\alpha \beta }^{\left( \alpha \right) }(%
\boldsymbol{\xi },\boldsymbol{\xi }_{\ast })\right) ^{2}\,d\boldsymbol{\xi }d%
\boldsymbol{\xi }_{\ast } \\
&\leq &\int_{\left\vert \mathbf{g}\right\vert \leq \left\vert \boldsymbol{%
\xi }\right\vert \leq N}\frac{C}{\left\vert \mathbf{g}\right\vert ^{2}}\exp
\left( -m_{\beta }\left( \dfrac{\left\vert \mathbf{g}\right\vert }{2}%
+\left\vert \boldsymbol{\xi }\right\vert \cos \varphi \right) ^{2}-\frac{%
m_{\alpha }^{2}}{4m_{\beta }}\left\vert \mathbf{g}\right\vert ^{2}\right) \,d%
\boldsymbol{\xi \,\,}d\mathbf{g} \\
&=&C\int_{0}^{N}\int_{0}^{\zeta }\int_{0}^{\pi }\zeta ^{2}\exp \left(
-m_{\beta }\left( \dfrac{R}{2}+\zeta \cos \varphi \right) ^{2}-\frac{%
m_{\alpha }^{2}}{4m_{\beta }}R^{2}\right) \sin \varphi \,d\varphi 
\boldsymbol{\,}dR\boldsymbol{\,}d\zeta  \\
&=&C\int_{0}^{N}\int_{0}^{\zeta }\int_{R-2\zeta }^{R+2\zeta }\zeta
e^{-m_{\beta }\eta ^{2}/4}e^{-m_{\alpha }^{2}R^{2}/\left( 4m_{\beta }\right)
}d\eta \boldsymbol{\,}dR\boldsymbol{\,}d\zeta  \\
&\leq &C\int_{0}^{N}\zeta \,d\zeta \int_{0}^{\infty }e^{-m_{\alpha
}^{2}R^{2}/\left( 4m_{\beta }\right) }dR\int_{-\infty }^{\infty
}e^{-m_{\beta }\eta ^{2}/4}d\eta =CN^{2}\text{.}
\end{eqnarray*}%
Furthermore, the integral of $k_{\alpha \beta }^{\left( \alpha \right) }(%
\boldsymbol{\xi },\boldsymbol{\xi }_{\ast })$ with respect to $\boldsymbol{%
\xi }$ over $\mathbb{R}^{3}$ is bounded in $\boldsymbol{\xi }_{\ast }$.
Indeed, directly by the bound $\left( \ref{b4}\right) $ on $\left( k_{\alpha
\beta }^{\left( \alpha \right) }\right) ^{2}$,%
\begin{equation}
0\leq k_{\alpha \beta }^{\left( \alpha \right) }(\boldsymbol{\xi },%
\boldsymbol{\xi }_{\ast })\leq \frac{C}{\left\vert \mathbf{g}\right\vert }%
\exp \left( -\frac{m_{\beta }}{8}\left( \left\vert \mathbf{g}\right\vert
+2\left\vert \boldsymbol{\xi }\right\vert \cos \varphi \right) ^{2}-\frac{%
m_{\alpha }^{2}}{8m_{\beta }}\left\vert \mathbf{g}\right\vert ^{2}\right) 
\text{.}  \label{b8}
\end{equation}%
Then the following bound on the integral of $k_{\alpha \beta }^{\left(
\alpha \right) }$ with respect to $\boldsymbol{\xi }_{\ast }$ over the
domain $\left\{ \mathbb{R}^{3}\text{; }\left\vert \mathbf{g}\right\vert \geq
\left\vert \boldsymbol{\xi }\right\vert \right\} $ can be obtained for $%
\left\vert \boldsymbol{\xi }\right\vert \neq 0$ 
\begin{eqnarray*}
\int_{\left\vert \mathbf{g}\right\vert \geq \left\vert \boldsymbol{\xi }%
\right\vert }k_{\alpha \beta }^{\left( \alpha \right) }(\boldsymbol{\xi },%
\boldsymbol{\xi }_{\ast })\,d\boldsymbol{\xi }_{\ast } &\leq &\frac{C}{%
\left\vert \boldsymbol{\xi }\right\vert }\int_{\left\vert \mathbf{g}%
\right\vert \geq \left\vert \boldsymbol{\xi }\right\vert }e^{-m_{\alpha
}^{2}\left\vert \mathbf{g}\right\vert ^{2}/\left( 8m_{\beta }\right) }d%
\mathbf{g} \\
&=&\frac{C}{\left\vert \boldsymbol{\xi }\right\vert }\int_{\left\vert 
\boldsymbol{\xi }\right\vert }^{\infty }R^{2}e^{-m_{\alpha }^{2}R^{2}/\left(
8m_{\beta }\right) }dR \\
&\leq &\frac{C}{\left\vert \boldsymbol{\xi }\right\vert }\int_{0}^{\infty
}R^{2}e^{-m_{\alpha }^{2}R^{2}/\left( 8m_{\beta }\right) }\,dR=\frac{C}{%
\left\vert \boldsymbol{\xi }\right\vert }\text{,}
\end{eqnarray*}%
as well as, over the domain $\left\{ \mathbb{R}^{3}\text{; }\left\vert 
\mathbf{g}\right\vert \leq \left\vert \boldsymbol{\xi }\right\vert \right\} $
\begin{eqnarray*}
&&\int_{\left\vert \mathbf{g}\right\vert \leq \left\vert \boldsymbol{\xi }%
\right\vert }k_{\alpha \beta }^{\left( \alpha \right) }(\boldsymbol{\xi },%
\boldsymbol{\xi }_{\ast })\,d\boldsymbol{\xi }_{\ast } \\
&\leq &\int_{\left\vert \mathbf{g}\right\vert \leq \left\vert \boldsymbol{%
\xi }\right\vert }\frac{C}{\left\vert \mathbf{g}\right\vert }\exp \left( -%
\frac{m_{\beta }}{8}\left( \left\vert \mathbf{g}\right\vert +2\left\vert 
\boldsymbol{\xi }\right\vert \cos \varphi \right) ^{2}-\frac{m_{\alpha }^{2}%
}{8m_{\beta }}\left\vert \mathbf{g}\right\vert ^{2}\right) \,d\mathbf{g} \\
&=&C\int_{0}^{\left\vert \boldsymbol{\xi }\right\vert }\int_{0}^{\pi }R\exp
\left( -\frac{m_{\beta }}{8}\left( R+2\left\vert \boldsymbol{\xi }%
\right\vert \cos \varphi \right) ^{2}-\frac{m_{\alpha }^{2}}{8m_{\beta }}%
R^{2}\right) \sin \varphi \,d\varphi \boldsymbol{\,}dR \\
&=&\frac{C}{\left\vert \boldsymbol{\xi }\right\vert }\int_{0}^{\left\vert 
\boldsymbol{\xi }\right\vert }\int_{R-2\left\vert \boldsymbol{\xi }%
\right\vert }^{R+2\left\vert \boldsymbol{\xi }\right\vert }Re^{-m_{\beta
}\eta ^{2}/8}e^{-m_{\alpha }^{2}R^{2}/\left( 8m_{\beta }\right) }d\eta 
\boldsymbol{\,}dR \\
&\leq &\frac{C}{\left\vert \boldsymbol{\xi }\right\vert }\int_{0}^{\infty
}Re^{-m_{\alpha }^{2}R^{2}/\left( 8m_{\beta }\right) }\,dR\int_{-\infty
}^{\infty }e^{-m_{\beta }\eta ^{2}/8}d\eta =\frac{C}{\left\vert \boldsymbol{%
\xi }\right\vert }\text{.}
\end{eqnarray*}%
However, due to the symmetry $k_{\alpha \beta }^{\left( \alpha \right) }(%
\boldsymbol{\xi },\boldsymbol{\xi }_{\ast })=k_{\alpha \beta }^{\left(
\alpha \right) }(\boldsymbol{\xi }_{\ast },\boldsymbol{\xi })$ $\left( \ref%
{sa3}\right) $, also 
\begin{equation*}
\int_{\mathbb{R}^{3}}k_{\alpha \beta }^{\left( \alpha \right) }(\boldsymbol{%
\xi },\boldsymbol{\xi }_{\ast })\,d\boldsymbol{\xi }\leq \frac{C}{\left\vert 
\boldsymbol{\xi }_{\ast }\right\vert }\text{.}
\end{equation*}%
Therefore, if $\left\vert \boldsymbol{\xi }_{\ast }\right\vert \geq 1$, then%
\begin{equation*}
\int_{\mathbb{R}^{3}}k_{\alpha \beta }^{\left( \alpha \right) }(\boldsymbol{%
\xi },\boldsymbol{\xi }_{\ast })\,d\boldsymbol{\xi }\leq \frac{C}{\left\vert 
\boldsymbol{\xi }_{\ast }\right\vert }\leq C\text{.}
\end{equation*}%
Otherwise, if $\left\vert \boldsymbol{\xi }_{\ast }\right\vert \leq 1$,
then, by the bound $\left( \ref{b8}\right) $,%
\begin{eqnarray*}
&&\int_{\mathbb{R}^{3}}k_{\alpha \beta }^{\left( \alpha \right) }(%
\boldsymbol{\xi },\boldsymbol{\xi }_{\ast })\,d\boldsymbol{\xi =}\int_{%
\mathbb{R}^{3}}k_{\alpha \beta }^{\left( \alpha \right) }(\boldsymbol{\xi }%
_{\ast },\boldsymbol{\xi })\,d\boldsymbol{\xi } \\
&\leq &\int_{\mathbb{R}^{3}}\frac{C}{\left\vert \mathbf{g}\right\vert }\exp
\left( -\frac{m_{\beta }}{8}\left( \left\vert \mathbf{g}\right\vert
+2\left\vert \boldsymbol{\xi }_{\ast }\right\vert \cos \varphi \right) ^{2}-%
\frac{m_{\alpha }^{2}}{8m_{\beta }}\left\vert \mathbf{g}\right\vert
^{2}\right) \,d\mathbf{g} \\
&=&C\int_{0}^{\infty }\int_{0}^{\pi }R\exp \left( -\frac{m_{\beta }}{8}%
\left( R+2\left\vert \boldsymbol{\xi }_{\ast }\right\vert \cos \varphi
\right) ^{2}-\frac{m_{\alpha }^{2}}{8m_{\beta }}R^{2}\right) \sin \varphi
\,d\varphi \boldsymbol{\,}dR \\
&\leq &\frac{C}{\left\vert \boldsymbol{\xi }_{\ast }\right\vert }%
\int_{0}^{\infty }Re^{-m_{\alpha }^{2}R^{2}/\left( 8m_{\beta }\right)
}dR\int_{R-2\left\vert \boldsymbol{\xi }_{\ast }\right\vert }^{R+2\left\vert 
\boldsymbol{\xi }_{\ast }\right\vert }\,d\eta =C\text{.}
\end{eqnarray*}%
Furthermore, 
\begin{eqnarray*}
&&\sup_{\boldsymbol{\xi }\in \mathbb{R}^{3}}\int_{\mathbb{R}^{3}}k_{\alpha
\beta }^{\left( \alpha \right) }(\boldsymbol{\xi },\boldsymbol{\xi }_{\ast
})-k_{\alpha \beta }^{\left( \alpha \right) }(\boldsymbol{\xi },\boldsymbol{%
\xi }_{\ast })\mathbf{1}_{\mathfrak{h}_{N}}\,d\boldsymbol{\xi }_{\ast } \\
&\leq &\sup_{\boldsymbol{\xi }\in \mathbb{R}^{3}}\int_{\left\vert \mathbf{g}%
\right\vert \leq \frac{1}{N}}k_{\alpha \beta }^{\left( \alpha \right) }(%
\boldsymbol{\xi },\boldsymbol{\xi }_{\ast })\,d\boldsymbol{\xi }_{\ast
}+\sup_{\left\vert \boldsymbol{\xi }\right\vert \geq N}\int_{\mathbb{R}%
^{3}}k_{\alpha \beta }^{\left( \alpha \right) }(\boldsymbol{\xi },%
\boldsymbol{\xi }_{\ast })\,d\boldsymbol{\xi }_{\ast } \\
&\leq &\int_{\left\vert \mathbf{g}\right\vert \leq \frac{1}{N}}\frac{C}{%
\left\vert \mathbf{g}\right\vert }\,d\mathbf{g}+\frac{C}{N}\leq C\left(
\int_{0}^{\frac{1}{N}}R\,dR+\frac{1}{N}\right)  \\
&=&C\left( \frac{1}{N^{2}}+\frac{1}{N}\right) \rightarrow 0\text{ as }%
N\rightarrow \infty \text{.}
\end{eqnarray*}%
Hence, by Lemma \ref{LGD} the operators 
\begin{equation*}
K_{\alpha \beta }^{(3)}=\int_{\mathbb{R}^{3}}k_{\alpha \beta }^{\left(
\alpha \right) }(\boldsymbol{\xi },\boldsymbol{\xi }_{\ast })\,h_{\alpha
\ast }\,d\boldsymbol{\xi }_{\ast }
\end{equation*}%
are compact on $L^{2}\left( d\boldsymbol{\xi }\right) $ for all $\left\{
\alpha ,\beta \right\} \subseteq \left\{ 1,...,s\right\} $.

\begin{figure}[h]
\centering
\includegraphics[width=0.45\textwidth]{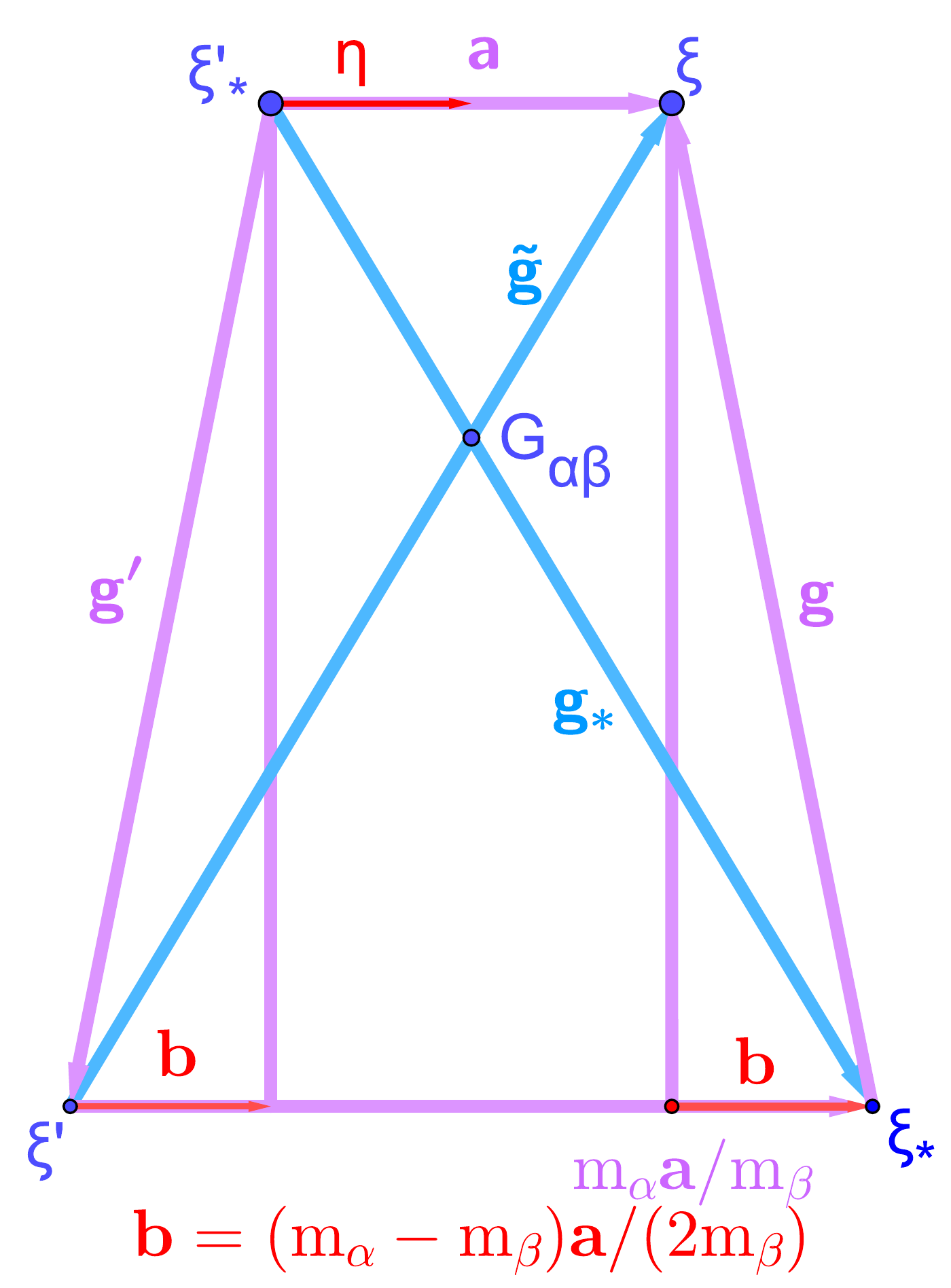}
\caption{Typical collision of $K_{\protect\alpha \protect\beta }^{(2)}$.}
\label{fig5}
\end{figure}

\textbf{III. Compactness of }$K_{\alpha \beta }^{(2)}=\int_{\mathbb{R}%
^{3}}k_{\alpha \beta 2}^{\left( \beta \right) }(\boldsymbol{\xi },%
\boldsymbol{\xi }_{\ast })\,h_{\beta \ast }\,d\boldsymbol{\xi }_{\ast }$.

Assume that $m_{\alpha }\neq m_{\beta }$. Noting that, cf. Figure \ref{fig5},%
\begin{eqnarray*}
W_{\alpha \beta }\left( \boldsymbol{\xi },\boldsymbol{\xi }^{\prime
}\left\vert \boldsymbol{\xi }_{\ast }^{\prime },\boldsymbol{\xi }_{\ast
}\right. \right) &=&2\left( m_{\alpha }+m_{\beta }\right) ^{2}m_{\alpha
}m_{\beta }\sigma _{\alpha \beta }\delta _{1}\left( \frac{m_{\alpha
}m_{\beta }}{m_{\alpha }-m_{\beta }}\left( \left\vert \mathbf{g}\right\vert
^{2}-\left\vert \mathbf{g}^{\prime }\right\vert ^{2}\right) \right) \\
&&\times \delta _{3}\left( \left( m_{\alpha }-m_{\beta }\right) \left( 
\mathbf{g}_{\alpha \beta }-\mathbf{g}_{\alpha \beta }^{\prime }\right)
\right) \\
&=&\frac{\left( m_{\alpha }+m_{\beta }\right) ^{2}}{\left( m_{\alpha
}-m_{\beta }\right) ^{2}}\frac{\sigma _{\alpha \beta }}{\left\vert \mathbf{g}%
\right\vert }\delta _{1}\left( \left\vert \mathbf{g}\right\vert -\left\vert 
\mathbf{g}^{\prime }\right\vert \right) \delta _{3}\left( \mathbf{g}_{\alpha
\beta }-\mathbf{g}_{\alpha \beta }^{\prime }\right) \text{, with} \\
&&\mathbf{g}_{\alpha \beta }=\dfrac{m_{\alpha }\boldsymbol{\xi }-m_{\beta }%
\boldsymbol{\xi }_{\ast }}{m_{\alpha }-m_{\beta }}\text{ and }\mathbf{g}%
_{\alpha \beta }^{\prime }=\dfrac{m_{\alpha }\boldsymbol{\xi }_{\ast
}^{\prime }-m_{\beta }\boldsymbol{\xi }^{\prime }}{m_{\alpha }-m_{\beta }}%
\text{,}
\end{eqnarray*}%
by a change of variables $\left\{ \boldsymbol{\xi }^{\prime },\boldsymbol{%
\xi }_{\ast }^{\prime }\right\} \rightarrow \left\{ \!\mathbf{g}^{\prime }=%
\boldsymbol{\xi }^{\prime }-\boldsymbol{\xi }_{\ast }^{\prime },\mathbf{g}%
_{\alpha \beta }^{\prime }=\dfrac{m_{\alpha }\boldsymbol{\xi }_{\ast
}^{\prime }-m_{\beta }\boldsymbol{\xi }^{\prime }}{m_{\alpha }-m_{\beta }}%
\!\right\} $, where%
\begin{equation*}
d\boldsymbol{\xi }^{\prime }d\boldsymbol{\xi }_{\ast }^{\prime }=d\mathbf{g}%
^{\prime }d\mathbf{g}_{\alpha \beta }^{\prime }=\left\vert \mathbf{g}%
^{\prime }\right\vert ^{2}d\left\vert \mathbf{g}^{\prime }\right\vert d%
\mathbf{g}_{\alpha \beta }^{\prime }d\boldsymbol{\omega }\text{, with\ }%
\boldsymbol{\omega }=\frac{\mathbf{g}^{\prime }}{\left\vert \mathbf{g}%
^{\prime }\right\vert }\text{,}
\end{equation*}%
the expression $\left( \ref{k2}\right) $ of $k_{\alpha \beta 2}^{\left(
\beta \right) }$ may be transformed to%
\begin{eqnarray*}
k_{\alpha \beta 2}^{\left( \beta \right) }(\boldsymbol{\xi },\boldsymbol{\xi 
}_{\ast }) &=&\int_{\left( \mathbb{R}^{3}\right) ^{2}}\left( M_{\beta
}^{\prime }M_{\alpha \ast }^{\prime }\right) ^{1/2}W_{\alpha \beta }\left( 
\boldsymbol{\xi },\boldsymbol{\xi }^{\prime }\left\vert \boldsymbol{\xi }%
_{\ast }^{\prime },\boldsymbol{\xi }_{\ast }\right. \right) \,d\mathbf{g}%
^{\prime }d\mathbf{g}_{\alpha \beta }^{\prime } \\
&=&\frac{\left( m_{\alpha }+m_{\beta }\right) ^{2}}{\left( m_{\alpha
}-m_{\beta }\right) ^{2}}\int_{\mathbb{S}^{2}}\left( M_{\beta }^{\prime
}M_{\alpha \ast }^{\prime }\right) ^{1/2}\left\vert \mathbf{g}\right\vert
\sigma _{\alpha \beta }\,\left( \left\vert \widetilde{\mathbf{g}}\right\vert
,-\frac{\widetilde{\mathbf{g}}\cdot \mathbf{g}_{\ast }}{\left\vert 
\widetilde{\mathbf{g}}\right\vert \left\vert \mathbf{g}_{\ast }\right\vert }%
\right) \,d\boldsymbol{\omega }\text{,} \\
\text{with} &&\mathbf{g}=\boldsymbol{\xi }-\boldsymbol{\xi }_{\ast }\text{, }%
\mathbf{g}^{\prime }=\boldsymbol{\xi }^{\prime }-\boldsymbol{\xi }_{\ast
}^{\prime }\text{, }\widetilde{\mathbf{g}}=\boldsymbol{\xi }-\boldsymbol{\xi 
}^{\prime }\text{ and }\mathbf{g}_{\ast }=\boldsymbol{\xi }_{\ast }-%
\boldsymbol{\xi }_{\ast }^{\prime }\text{.}
\end{eqnarray*}%
Here, see Figure \ref{fig5},%
\begin{equation*}
\left\{ 
\begin{array}{l}
\boldsymbol{\xi }^{\prime }=\boldsymbol{\xi }_{\ast }-\dfrac{m_{\alpha }}{%
m_{\beta }}\left\vert \boldsymbol{\xi }-\boldsymbol{\xi }_{\ast }^{\prime
}\right\vert \boldsymbol{\eta } \\ 
\boldsymbol{\xi }_{\ast }^{\prime }=\boldsymbol{\xi }-\left\vert \boldsymbol{%
\xi }-\boldsymbol{\xi }_{\ast }^{\prime }\right\vert \boldsymbol{\eta }%
\end{array}%
\right. \text{, }\boldsymbol{\eta }=\frac{\boldsymbol{\xi }-\boldsymbol{\xi }%
_{\ast }^{\prime }}{\left\vert \boldsymbol{\xi }-\boldsymbol{\xi }_{\ast
}^{\prime }\right\vert }\in \mathbb{S}^{2}\text{.}
\end{equation*}%
Then, by Lemma $\ref{L2}$, since relation $\left( \ref{vrel3}\right) $
follows by energy conservation, 
\begin{equation*}
m_{\beta }\left\vert \boldsymbol{\xi }^{\prime }\right\vert ^{2}+m_{\alpha
}\left\vert \boldsymbol{\xi }_{\ast }^{\prime }\right\vert ^{2}\geq \rho
\left( m_{\alpha }\left\vert \boldsymbol{\xi }\right\vert ^{2}+m_{\beta
}\left\vert \boldsymbol{\xi }_{\ast }\right\vert ^{2}\right) ,
\end{equation*}%
for some positive number $\rho $, $0<\rho <1$, and hence, by assumption $%
\left( \ref{est2}\right) $, noticing that $\left\vert \mathbf{g}\right\vert
\leq \left\vert \widetilde{\mathbf{g}}\right\vert $, cf. Figure \ref{fig5},
the bound%
\begin{eqnarray}
\left( k_{\alpha \beta 2}^{\left( \beta \right) }(\boldsymbol{\xi },%
\boldsymbol{\xi }_{\ast })\right) ^{2} &\leq &C\left( M_{\alpha }M_{\beta
\ast }\right) ^{\rho }\left( \int_{\mathbb{S}^{2}}\left( 1+\frac{1}{%
\left\vert \widetilde{\mathbf{g}}\right\vert ^{2-\gamma }}\right) \left\vert 
\mathbf{g}\right\vert \,d\boldsymbol{\omega }\right) ^{2}  \notag \\
&\leq &CM_{\alpha }^{\rho }M_{\beta \ast }^{\rho }\left( 1+\frac{1}{%
\left\vert \mathbf{g}\right\vert ^{2-\gamma }}\right) ^{2}\left\vert \mathbf{%
g}\right\vert ^{2}\left( \int_{\mathbb{S}^{2}}\,d\boldsymbol{\omega }\right)
^{2}  \notag \\
&=&CM_{\alpha }^{\rho }M_{\beta \ast }^{\rho }\left( \left\vert \mathbf{g}%
\right\vert +\frac{1}{\left\vert \mathbf{g}\right\vert ^{1-\gamma }}\right)
^{2}  \label{b6}
\end{eqnarray}%
may be obtained. Then, by applying the bound $\left( \ref{b6}\right) $ and
first changing variables of integration $\left\{ \boldsymbol{\xi },%
\boldsymbol{\xi }_{\ast }\right\} \rightarrow \left\{ \mathbf{g},\mathbf{G}%
_{\alpha \beta }\right\} $, with unitary Jacobian, and then to spherical
coordinates, 
\begin{eqnarray*}
&&\int_{\left( \mathbb{R}^{3}\right) ^{2}}\left( k_{\alpha \beta 2}^{\left(
\beta \right) }(\boldsymbol{\xi },\boldsymbol{\xi }_{\ast })\right) ^{2}d%
\boldsymbol{\xi \,}d\boldsymbol{\xi }_{\ast } \\
&\leq &C\int_{\left( \mathbb{R}^{3}\right) ^{2}}e^{-\rho \left( m_{\alpha
}+m_{\beta }\right) \left\vert \mathbf{G}_{\alpha \beta }\right\vert
^{2}/2-\mu _{\alpha \beta }\rho \left\vert \mathbf{g}\right\vert
^{2}/2}\left( \left\vert \mathbf{g}\right\vert +\frac{1}{\left\vert \mathbf{g%
}\right\vert ^{1-\gamma }}\right) ^{2}d\mathbf{g}\boldsymbol{\,}d\mathbf{G}%
_{\alpha \beta } \\
&\leq &C\int_{0}^{\infty }R^{2}e^{-R^{2}}dR\int_{0}^{\infty }\left(
r^{2}+r^{\gamma }\right) ^{2}e^{-r^{2}/4}dr \\
&\leq &C\int_{0}^{\infty }R^{2}e^{-R^{2}}dR\int_{0}^{\infty }\left(
1+r^{4}\right) e^{-r^{2}/4}dr=C\text{.}
\end{eqnarray*}%
Hence,%
\begin{equation*}
K_{\alpha \beta }^{(2)}=\int_{\mathbb{R}^{3}}k_{\alpha \beta 2}^{\left(
\beta \right) }(\boldsymbol{\xi },\boldsymbol{\xi }_{\ast })\,h_{\beta \ast
}\,d\boldsymbol{\xi }_{\ast }
\end{equation*}%
are Hilbert-Schmidt integral operators and as such continuous and compact on 
$L^{2}\left( d\boldsymbol{\xi }\right) $ \cite[Theorem 7.83]{RenardyRogers}
for all $\left\{ \alpha ,\beta \right\} \subseteq \left\{ 1,...,s\right\} $.

On the other hand, if $m_{\alpha }=m_{\beta }$, then 
\begin{eqnarray*}
k_{\alpha \beta 2}^{\left( \beta \right) }(\boldsymbol{\xi },\boldsymbol{\xi 
}_{\ast }) &=&\int_{\left( \mathbb{R}^{3}\right) ^{\perp _{\mathbf{n}%
}}}4\left( M_{\beta }^{\prime }M_{\alpha \ast }^{\prime }\right) ^{1/2}\frac{%
\sigma _{\alpha \beta }}{\left\vert \mathbf{g}\right\vert }\,\left(
\left\vert \widetilde{\mathbf{g}}\right\vert ,-\frac{\widetilde{\mathbf{g}}%
\cdot \mathbf{g}_{\ast }}{\left\vert \widetilde{\mathbf{g}}\right\vert
\left\vert \mathbf{g}_{\ast }\right\vert }\right) \boldsymbol{\,}d\mathbf{w}%
\text{,} \\
&&\text{with }\widetilde{\mathbf{g}}=\boldsymbol{\xi }-\boldsymbol{\xi }%
^{\prime }\text{ and }\mathbf{g}_{\ast }=\boldsymbol{\xi }_{\ast }-%
\boldsymbol{\xi }_{\ast }^{\prime }\text{.}
\end{eqnarray*}%
Here%
\begin{equation*}
\left\{ 
\begin{array}{l}
\boldsymbol{\xi }^{\prime }=\boldsymbol{\xi }_{\ast }+\mathbf{w} \\ 
\boldsymbol{\xi }_{\ast }^{\prime }=\boldsymbol{\xi }+\mathbf{w}%
\end{array}%
\right. \text{, with }\mathbf{w\perp g}\text{ and }\mathbf{g}=\boldsymbol{%
\xi }-\boldsymbol{\xi }_{\ast }\text{.}
\end{equation*}%
Then similar arguments to the ones for $k_{\alpha \beta }^{\left( \alpha
\right) }(\boldsymbol{\xi },\boldsymbol{\xi }_{\ast })$ (with $m_{\alpha
}=m_{\beta }$) above, can be applied.

Concluding, the operator 
\begin{equation*}
K=(K_{1},...,K_{s})=\sum\limits_{\beta =1}^{s}\left( (K_{1\beta
}^{(3)},...,K_{s\beta }^{(3)})-(K_{1\beta }^{(1)},...,K_{s\beta
}^{(1)})+(K_{1\beta }^{(2)},...,K_{s\beta }^{(2)})\right)
\end{equation*}%
is a compact self-adjoint operator on $\left( L^{2}\left( d\boldsymbol{\xi }%
\right) \right) ^{s}$. Self-adjointness is due to the symmetry relations $%
\left( \ref{sa3}\right) ,\left( \ref{sa4}\right) $, cf. \cite[p.198]%
{Yoshida-65}.
\end{proof}

\subsubsection{\label{PT4}Bounds on the collision frequency}

This section concerns the proof of Theorem \ref{Thm4}.

\begin{proof}
For a hard sphere model $\sigma _{\alpha \beta }=C_{\alpha \beta }$ for some
positive constant $C_{\alpha \beta }$ for any $\left\{ \alpha ,\beta
\right\} \subset \left\{ 1,...,s\right\} $. Then each collision frequency $%
\upsilon _{1},...,\upsilon _{s}$ equals%
\begin{eqnarray*}
\upsilon _{\alpha } &=&\sum\limits_{\beta =1}^{s}\int_{\left( \mathbb{R}%
^{3}\right) ^{3}}M_{\beta \ast }W_{\alpha \beta }(\boldsymbol{\xi },%
\boldsymbol{\xi }_{\ast }\left\vert \boldsymbol{\xi }^{\prime },\boldsymbol{%
\xi }_{\ast }^{\prime }\right. )\,d\boldsymbol{\xi }_{\ast }d\boldsymbol{\xi 
}^{\prime }d\boldsymbol{\xi }_{\ast }^{\prime } \\
&=&\sum\limits_{\beta =1}^{s}\int_{\left( \mathbb{R}^{3}\right)
^{3}}M_{\beta \ast }\frac{C_{\alpha \beta }}{\left\vert \mathbf{g}%
\right\vert }\delta _{3}\left( \mathbf{G}_{\alpha \beta }-\mathbf{G}_{\alpha
\beta }^{\prime }\right) \delta _{1}\left( \left\vert \mathbf{g}\right\vert
-\left\vert \mathbf{g}^{\prime }\right\vert \right) \,d\boldsymbol{\xi }%
_{\ast }d\mathbf{G}^{\prime }d\mathbf{g}^{\prime } \\
&=&\sum\limits_{\beta =1}^{s}\int_{\mathbb{R}^{3}}M_{\beta \ast }\left\vert 
\mathbf{g}\right\vert \,d\boldsymbol{\xi }_{\ast }\int_{\mathbb{S}^{2}}\,d%
\boldsymbol{\omega } \\
&=&\sum\limits_{\beta =1}^{s}n_{\beta }m_{\beta }\sqrt{\frac{2m_{\beta }}{%
\pi }}C_{\alpha \beta }\int_{\mathbb{R}^{3}}e^{-m_{\beta }\left\vert 
\boldsymbol{\xi }_{\ast }\right\vert ^{2}/2}\left\vert \boldsymbol{\xi }%
_{\ast }-\boldsymbol{\xi }\right\vert \,d\boldsymbol{\xi }_{\ast }\text{.}
\end{eqnarray*}%
Now consider the two different cases $\left\vert \boldsymbol{\xi }%
\right\vert \leq 1$ and $\left\vert \boldsymbol{\xi }\right\vert \geq 1$
separately. Firstly, if $\left\vert \boldsymbol{\xi }\right\vert \geq 1$,
then 
\begin{eqnarray*}
\upsilon _{\alpha } &\geq &C\sum\limits_{\beta =1}^{s}\int_{\mathbb{R}%
^{3}}e^{-m_{\beta }\left\vert \boldsymbol{\xi }_{\ast }\right\vert
^{2}}\left\vert \left\vert \boldsymbol{\xi }\right\vert -\left\vert 
\boldsymbol{\xi }_{\ast }\right\vert \right\vert \,d\boldsymbol{\xi }_{\ast }
\\
&\geq &C\sum\limits_{\beta =1}^{s}\int_{\left\vert \boldsymbol{\xi }_{\ast
}\right\vert \leq 1/2}e^{-m_{\beta }\left\vert \boldsymbol{\xi }_{\ast
}\right\vert ^{2}}\left( \left\vert \boldsymbol{\xi }\right\vert -\left\vert 
\boldsymbol{\xi }_{\ast }\right\vert \right) \,d\boldsymbol{\xi }_{\ast } \\
&\geq &C\sum\limits_{\beta =1}^{s}e^{-m_{\beta }/4}\frac{\left\vert 
\boldsymbol{\xi }\right\vert }{2}\int_{\left\vert \boldsymbol{\xi }_{\ast
}\right\vert \leq 1/2}\,d\boldsymbol{\xi }_{\ast }\geq C\left\vert 
\boldsymbol{\xi }\right\vert \geq C\left( 1+\left\vert \boldsymbol{\xi }%
\right\vert \right) \text{.}
\end{eqnarray*}%
Secondly, if $\left\vert \boldsymbol{\xi }\right\vert \leq 1$, then 
\begin{eqnarray*}
\upsilon _{\alpha } &\geq &C\sum\limits_{\beta =1}^{s}\int_{\mathbb{R}%
^{3}}e^{-m_{\beta }\left\vert \boldsymbol{\xi }_{\ast }\right\vert
^{2}}\left\vert \left\vert \boldsymbol{\xi }\right\vert -\left\vert 
\boldsymbol{\xi }_{\ast }\right\vert \right\vert \,d\boldsymbol{\xi }_{\ast }
\\
&\geq &C\sum\limits_{\beta =1}^{s}\int_{\left\vert \boldsymbol{\xi }_{\ast
}\right\vert \geq 2}e^{-m_{\beta }\left\vert \boldsymbol{\xi }_{\ast
}\right\vert ^{2}}\left( \left\vert \boldsymbol{\xi }_{\ast }\right\vert
-\left\vert \boldsymbol{\xi }\right\vert \right) \,d\boldsymbol{\xi }_{\ast }
\\
&\geq &C\sum\limits_{\beta =1}^{s}\int_{\left\vert \boldsymbol{\xi }_{\ast
}\right\vert \geq 2}e^{-m_{\beta }\left\vert \boldsymbol{\xi }_{\ast
}\right\vert ^{2}}\frac{\left\vert \boldsymbol{\xi }_{\ast }\right\vert }{2}%
\,d\boldsymbol{\xi }_{\ast } \\
&=&C\sum\limits_{\beta =1}^{s}\int_{2}^{\infty }r^{3}e^{-m_{\beta
}r^{2}}\,dr\int_{\mathbb{S}^{2}}\,d\boldsymbol{\omega }=C\geq C\left(
1+\left\vert \boldsymbol{\xi }\right\vert \right) \text{.}
\end{eqnarray*}%
Hence, there is a positive constant $\upsilon _{-}>0$, such that $\upsilon
_{\alpha }\geq \upsilon _{-}\left( 1+\left\vert \boldsymbol{\xi }\right\vert
\right) $ for all $\alpha \in \left\{ 1,...,s\right\} $ and $\boldsymbol{\xi 
}\in \mathbb{R}^{3}$.

On the other hand,%
\begin{eqnarray*}
\upsilon _{\alpha } &\leq &C\sum\limits_{\beta =1}^{s}\int_{\mathbb{R}%
^{3}}e^{-m_{\beta }\left\vert \boldsymbol{\xi }_{\ast }\right\vert
^{2}}\left( \left\vert \boldsymbol{\xi }\right\vert +\left\vert \boldsymbol{%
\xi }_{\ast }\right\vert \right) \,d\boldsymbol{\xi }_{\ast } \\
&=&C\left\vert \boldsymbol{\xi }\right\vert \sum\limits_{\beta =1}^{s}\int_{%
\mathbb{R}^{3}}e^{-m_{\beta }\left\vert \boldsymbol{\xi }_{\ast }\right\vert
^{2}}d\boldsymbol{\xi }_{\ast }+C\sum\limits_{\beta =1}^{s}\int_{\mathbb{R}%
^{3}}e^{-m_{\beta }\left\vert \boldsymbol{\xi }_{\ast }\right\vert
^{2}}\left\vert \boldsymbol{\xi }_{\ast }\right\vert \,d\boldsymbol{\xi }%
_{\ast } \\
&\leq &C\left( 1+\left\vert \boldsymbol{\xi }\right\vert \right) \text{.}
\end{eqnarray*}%
Hence, there is a positive constant $\upsilon _{+}>0$, such that $\upsilon
_{\alpha }\leq \upsilon _{+}\left( 1+\left\vert \boldsymbol{\xi }\right\vert
\right) $ for all $\alpha \in \left\{ 1,...,s\right\} $ and $\boldsymbol{\xi 
}\in \mathbb{R}^{3}$, and the theorem follows.
\end{proof}

\bibliographystyle{siamproc}
\bibliography{biblo1}

\section{Appendix}

This appendix concerns an alternative, to the one presented in \cite{BGPS-13}%
, proof of Lemma $\ref{L3}$.

\begin{proof}
Denote $q:=\left\vert \boldsymbol{\xi }-\boldsymbol{\xi }_{\ast }^{\prime
}\right\vert $. By the relations $\left( \ref{vrel2}\right) $, cf. Figure $%
\ref{fig5}$, 
\begin{equation*}
\left\{ 
\begin{array}{l}
\boldsymbol{\xi }=\mathbf{w}+r\boldsymbol{\eta } \\ 
\boldsymbol{\xi }_{\ast }^{\prime }=\mathbf{w}+\left( r-q\right) \boldsymbol{%
\eta }%
\end{array}%
\right. \text{, with }\mathbf{w}\perp \boldsymbol{\eta }\text{, \ }
\end{equation*}%
while%
\begin{equation*}
\left\{ 
\begin{array}{l}
\boldsymbol{\xi }^{\prime }=\widetilde{\mathbf{w}}+\left( r_{\ast }-\dfrac{%
m_{\alpha }}{m_{\beta }}q\right) \boldsymbol{\eta } \\ 
\boldsymbol{\xi }_{\ast }=\widetilde{\mathbf{w}}+r_{\ast }\boldsymbol{\eta }%
\end{array}%
\right. \text{, with }\widetilde{\mathbf{w}}\perp \boldsymbol{\eta }\text{,}
\end{equation*}%
where it follows by relation $\left( \ref{vrel3}\right) $ that 
\begin{equation*}
r_{\ast }=r+\frac{m_{\alpha }-m_{\beta }}{2m_{\beta }}q\text{,}
\end{equation*}%
implying that%
\begin{equation*}
\left\{ 
\begin{array}{l}
\boldsymbol{\xi }^{\prime }=\widetilde{\mathbf{w}}+\left( r-\dfrac{m_{\alpha
}+m_{\beta }}{2m_{\beta }}q\right) \boldsymbol{\eta } \\ 
\boldsymbol{\xi }_{\ast }=\widetilde{\mathbf{w}}+\left( r+\dfrac{m_{\alpha
}-m_{\beta }}{2m_{\beta }}q\right) \boldsymbol{\eta }%
\end{array}%
\right. \text{, with }\widetilde{\mathbf{w}}\perp \boldsymbol{\eta }\text{. }
\end{equation*}%
Then%
\begin{eqnarray*}
&&m_{\beta }\left\vert \boldsymbol{\xi }^{\prime }\right\vert ^{2}+m_{\alpha
}\left\vert \boldsymbol{\xi }_{\ast }^{\prime }\right\vert ^{2} \\
&=&m_{\alpha }\left\vert \mathbf{w}\right\vert ^{2}+m_{\beta }\left\vert 
\widetilde{\mathbf{w}}\right\vert ^{2}+\left( m_{\alpha }+m_{\beta }\right)
r^{2}-\left( 3m_{\alpha }+m_{\beta }\right) qr \\
&&+\frac{m_{\alpha }^{2}+6m_{\alpha }m_{\beta }+m_{\beta }^{2}}{4m_{\beta }}%
q^{2} \\
&=&m_{\alpha }\left\vert \mathbf{w}\right\vert ^{2}+m_{\beta }\left\vert 
\widetilde{\mathbf{w}}\right\vert ^{2}+\left( m_{\alpha }+m_{\beta }\right)
\left( r-\frac{3m_{\alpha }+m_{\beta }}{2\left( m_{\alpha }+m_{\beta
}\right) }q\right) ^{2} \\
&&+\frac{\left( m_{\alpha }-m_{\beta }\right) ^{2}}{m_{\alpha }+m_{\beta }}%
\frac{m_{\alpha }}{4m_{\beta }}q^{2}\text{,}
\end{eqnarray*}%
while%
\begin{eqnarray*}
&&m_{\alpha }\left\vert \boldsymbol{\xi }\right\vert ^{2}+m_{\beta
}\left\vert \boldsymbol{\xi }_{\ast }\right\vert ^{2} \\
&=&m_{\alpha }\left\vert \mathbf{w}\right\vert ^{2}+m_{\beta }\left\vert 
\widetilde{\mathbf{w}}\right\vert ^{2}+\left( m_{\alpha }+m_{\beta }\right)
r^{2}+\left( m_{\alpha }-m_{\beta }\right) qr+\frac{\left( m_{\alpha
}-m_{\beta }\right) ^{2}}{4m_{\beta }}q^{2} \\
&=&m_{\alpha }\left\vert \mathbf{w}\right\vert ^{2}+m_{\beta }\left\vert 
\widetilde{\mathbf{w}}\right\vert ^{2}+\left( m_{\alpha }+m_{\beta }\right)
\left( r+\frac{m_{\alpha }-m_{\beta }}{2\left( m_{\alpha }+m_{\beta }\right) 
}q\right) ^{2} \\
&&+\frac{\left( m_{\alpha }-m_{\beta }\right) ^{2}}{m_{\alpha }+m_{\beta }}%
\frac{m_{\alpha }}{4m_{\beta }}q^{2}\text{.}
\end{eqnarray*}%
Denote 
\begin{equation*}
a:=-\dfrac{3m_{\alpha }+m_{\beta }}{2\left( m_{\alpha }+m_{\beta }\right) }q%
\text{, }b:=\dfrac{m_{\alpha }-m_{\beta }}{2\left( m_{\alpha }+m_{\beta
}\right) }q\text{, and }c^{2}:=\dfrac{\left( m_{\alpha }-m_{\beta }\right)
^{2}}{4\left( m_{\alpha }+m_{\beta }\right) ^{2}}\dfrac{m_{\alpha }}{%
m_{\beta }}q^{2}.
\end{equation*}%
Then 
\begin{align*}
& \frac{m_{\beta }\left\vert \boldsymbol{\xi }^{\prime }\right\vert
^{2}+m_{\alpha }\left\vert \boldsymbol{\xi }_{\ast }^{\prime }\right\vert
^{2}-\rho \left( m_{\alpha }\left\vert \boldsymbol{\xi }\right\vert
^{2}+m_{\beta }\left\vert \boldsymbol{\xi }_{\ast }\right\vert ^{2}\right) }{%
m_{\alpha }+m_{\beta }} \\
\geq & \left( r+a\right) ^{2}-\rho \left( r+b\right) ^{2}+\left( 1-\rho
\right) c^{2} \\
=& \left( 1-\rho \right) \left( r^{2}+2\frac{a-\rho b}{1-\rho }r+\frac{%
a^{2}-\rho b^{2}}{1-\rho }+c^{2}\right) \\
=& \left( 1-\rho \right) \left( r+\frac{a-\rho b}{1-\rho }\right) ^{2}+\frac{%
c^{2}}{1-\rho }\left( \rho ^{2}-2\rho \left( 1+\frac{\left( a-b\right) ^{2}}{%
2c^{2}}\right) +1\right) \\
\geq & \frac{c^{2}}{1-\rho }\left( \rho ^{2}-2\rho \left( 1+\frac{\left(
a-b\right) ^{2}}{2c^{2}}\right) +1\right) \geq 0 \\
& \text{if }0\leq \rho \leq 1+\frac{\left( a-b\right) ^{2}}{2c^{2}}-\frac{1}{%
2c^{2}}\sqrt{\left( a-b\right) ^{4}+4c^{2}\left( a-b\right) ^{2}}<1\text{.}
\end{align*}%
Let%
\begin{eqnarray*}
\rho &=&1+\frac{\left( a-b\right) ^{2}}{2d^{2}}-\frac{1}{2c^{2}}\sqrt{\left(
a-b\right) ^{4}+4d^{2}\left( a-b\right) ^{2}} \\
&=&1+\frac{1-\sqrt{1+\dfrac{4d^{2}}{\left( a-b\right) ^{2}}}}{\dfrac{2d^{2}}{%
\left( a-b\right) ^{2}}}=1-\frac{2}{1+\sqrt{1+\dfrac{4d^{2}}{\left(
a-b\right) ^{2}}}} \\
&=&1-\frac{2}{1+\sqrt{1+\dfrac{\left( m_{\alpha }-m_{\beta }\right) ^{2}}{%
4m_{\alpha }m_{\beta }}}}>0\text{ if }m_{\alpha }\neq m_{\beta }\text{.}
\end{eqnarray*}
\end{proof}

\end{document}